\documentclass[10pt]{article}
\usepackage{color}
\usepackage{latexsym}
\usepackage{amssymb}
\usepackage{graphicx}
\usepackage{amsmath}
\usepackage{mathtools, amssymb}
\usepackage{mathrsfs}
\usepackage[T1]{fontenc}
\usepackage[utf8]{inputenc}
\usepackage{graphicx}
\usepackage[english]{babel}
\usepackage{overpic}
\usepackage{amssymb}
\usepackage{amsmath,upref}
\usepackage{multirow}
\usepackage{epsfig}
\usepackage{verbatim}
\usepackage{graphicx}
\usepackage{amsmath}
\usepackage{here}
\usepackage{lscape}
\usepackage{fancybox}
\usepackage{fancyhdr}
\usepackage[Lenny]{fncychap}
\usepackage{pifont}
\usepackage{color}
\usepackage{shapepar}
\def \p {\mathbb{P}}
\def \nn {\nonumber}
\def \lee {\lambda(E)}
\def \wfi {\widehat{f}_i}
\def \wf {\widehat{f}}

\usepackage[french]{minitoc}
\usepackage{fancyhdr}
\usepackage{amsxtra}
\def \sl {\underbar}
\def \fstt {\forall s\in [t,T]}
\def \p {\mathbb{P}}
\def \s{\sigma}
\usepackage{pdfpages}
\usepackage{amsmath,amssymb,mathrsfs} 
\usepackage{pifont} 
\usepackage{mathtools}
\usepackage{float}
\usepackage{bbm}
\usepackage[thmmarks,amsmath,framed]{ntheorem} 

\def\bal{\begin{align}}
\def\eal{\end{align}}
\def \lb{\label}
\def \espo{[0,T]\times \R^k}

\def \qed{\hspace*{\fill} $\Box$\par\medskip}
\def \qed{\hspace*{\fill} $\Box$\par\medskip}
\def \g {\gamma}
\def \kp {\kappa}
\def \txst{\textstyle}
\def \nnb {\nonumber}
\def \rw {\rightarrow}

\def \aimns {\mathcal{A}_s^{i,\mu_n}}
\def \aimn {\mathcal{A}^{i,\mu_n}}
\def \stt {s\in [t,T]}

\def \aimm {\mathcal{A}^{i,\mu_m}}

\def \qq {\qquad}
\def \Sup{\displaystyle\sup}

\def \txk {(t,x) \in [0,T]\times \R^{k}}
\def \spa { [0,T]\times \R^{k}}

\def \rk {\R^k}
\def \fst {\forall s\le T}

\def \xtx {X^{t,x}}

\def \vy {\vec y}
\def \rw {\rightarrow}
\def \fii {\forall \ii}

\def \mi {\mathcal{I}}

\def \Sup{\displaystyle\sup}

\def \txk {(t,x) \in [0,T]\times \R^{k}}
\def \spa { [0,T]\times \R^{k}}

\def \hdd{{\cal H}^{2,d}}
\def \g {\gamma}
\def \hld{\mathcal{H}^2(\mathcal{L}^2(\lambda))}
\def \aa {\mathcal{A}^2}
\def \d {\delta}

\def \wt {\mbox{w.r.t}}

\def \ss {{\cal S}^2}
\def \ii {i\in \mi}

\def \lb {\label}
\def \tx {(t,x)\in [0,T]\times \R^k}
\def \sol{\underline}
\def \nd {\noindent}
\def \tbf{\textbf}

\def \stt {s\in [t,T]}

\def \kp {\kappa}
\def \aimns {\mathcal{A}_s^{i,\mu_n}}
\def \aimn {\mathcal{A}^{i,\mu_n}}
\def \qed{\hspace*{\fill} $\Box$\par\medskip}
\def \rk {\R^k}
\def \xtx {X^{t,x}}

\def \vy {\vec y}

\def \mi {\mathcal{I}}

\def \Sup{\displaystyle\sup}

\def \txk {(t,x) \in [0,T]\times \R^{k}}
\def \spa { [0,T]\times \R^{k}}

\def \hdd{{\cal H}^{2,d}}
\def \hld{\mathcal{H}^2(\mathcal{L}^2(\lambda))}
\def \hln {\mathcal{H}^2(\mathcal{L}^2(\lambda_n))}
\def \ynr {({}^{n}\!Y_r^{k,t,x})_{k\in \mi}}
\def \yns {({}^{n}\!Y_s^{k,t,x})_{k\in \mi}}

\def \ymr {({}^{m}\!Y_r^{k,t,x})_{k\in \mi}}
\def \yms {({}^{m}\!Y_s^{k,t,x})_{k\in \mi}}
\def \be {\begin{equation}}
\def \ee {\end{equation}}
\def \aa {\mathcal{A}^2}
\def \d {\delta}
\def \wt {\mbox{w.r.t}}

\def \ss {{\cal S}^2}
\def \ij {i,j\in \mi}
\def \ii {i\in \mi}
\def \lb {\label}
\def \tx {(t,x)\in [0,T]\times \R^k}
\def \sol{\underline}
\def \nd {\noindent}
\def \tbf{\textbf}

\def \aimm {\mathcal{A}^{i,\mu_m}}

\def \rk {\R^k}
\def \xtx {X^{t,x}}

\def \vy {\vec y}

\def \mi {\mathcal{I}}

\def \Sup{\displaystyle\sup}

\def \txk {(t,x) \in [0,T]\times \R^{k}}
\def \spa { [0,T]\times \R^{k}}

\def \hdd{{\cal H}^{2,d}}
\def \hld{\mathcal{H}^2(\mathcal{L}^2(\lambda))}
\def \aa {\mathcal{A}^2}
\def \d {\delta}
\def \wt {\mbox{w.r.t}}

\def \ss {{\cal S}^2}
\def \ii {i\in \mi}
\def \lb {\label}
\def \tx {(t,x)\in [0,T]\times \R^k}
\def \sol{\underline}
\def \nd {\noindent}
\def \tbf{\textbf}
\def \nb {\nonumber}
\def \rk{\R^k}
\def \ms{\medskip}
 
\newcommand{\E}{\mathbb{E}}

\newcommand{\un}{{\mathbf{1}}} 

\newcommand{\pr}{\mathbb{P}}
\newcommand{\eps}{\varepsilon}

\newcommand{\R}{\mathbb{R}}

\def \proof{{\noindent \bf Proof: }}

\def \t {\tau}
\newcommand*{\rom}[2]{\expandafter\@slowromancap\romannumeral #1@}

\def \proof{{\noindent \bf Proof: }}

\numberwithin{equation}{section}
\newtheorem{definition}{Definition }[section]
\newtheorem{proposition}[definition]{Proposition }
\newtheorem{lemme}[definition]%
{Lemma }
\newtheorem{theoreme}[definition]%
{Theorem }
\newtheorem{corollaire}[definition]%
{Corollary }
\newtheorem{remarque}[definition]%
{Remark }
{Assumption}
\theoremseparator{.}

\def \cih {C_{h_i}}
\def \cif {C_{f_i}}
\def \cig {C_{\g_i}}

\def \bc {\color{blue}}
\def \rc {\color{red}}

\def \nc {\color{black}}
\title{\Large \bf Viscosity solution of systems of integral-partial differential equations with interconnected obstacles without Monotonicity Conditions and infinite L\'evy measure.
}
\author{\Large Said Hamad\`ene\thanks{LMM, Le Mans Universit\'e, Avenue Olivier Messiaen, 72085 Le Mans, Cedex 9, France. \texttt{e-mail: said.hamadene@univ-lemans.fr}} \and \Large Mohamed Mnif\thanks{University of Tunis El Manar, Laboratoire de Mod\'elisation Math\'ematique et Num\'erique \texttt{e-mail: mohamed.mnif@enit.utm.tn}} \and \Large Sarra Neffati\thanks{LMM, Le Mans Universit\'e, Avenue Olivier Messiaen, 72085 Le Mans, Cedex 9, France.
\texttt{e-mail: 
Sarra.neffati@epita.fr}}}
\begin{document}
\maketitle
\begin{abstract}
In this paper, we study a system of second-order integral-partial differential equations with interconnected obstacles. A particular case, is the Hamilton-Jacobi-Bellman (HJB for short) system associated with the optimal switching problem in the jump-diffusion model. Getting rid of the monotonicity condition on the generators with respect to the jump component, we construct a continuous viscosity solution of the system, which is unique in the class of continuous bounded functions. 
The L\'evy measure $\lambda(.)$ is of infinite activity. The main tool we use is the associated system of reflected backward stochastic differential equations with jumps and interconnected obstacles for which we also study existence and uniqueness of the solution. 
\end{abstract}
\nd {\bf Keywords}: Integral-partial differential equations; Interconnected obstacles; Non-local terms; Viscosity solution; Switching problem; Reflected backward stochastic differential equations with jumps; L\'evy measure.

\nd {\bf AMS Subject Classification}: 49L25, 35D40, 93C30, 49J40.
\ms

\section{Introduction}
 Let us consider the following system  of integral-partial differential equations (IPDEs for short) with interconnected obstacles: $\forall i \in \mathcal{I}:=\{1,\dots,m\}$,
\begin{equation}\label{sys1.1}
\begin{cases}
\min \lbrace u^i(t,x) - \displaystyle \max_{j \in {\mathcal{I}^{-i}}}(u^j(t,x)-g_{ij}(t)); -\partial_tu^i(t,x) - \mathcal{L}u^i(t,x) - \mathcal{K}u^i(t,x)\\\\
 \quad - {\bar f}_i(t,x,(u^k(t,x))_{k=1,m},(\sigma^\top (t,x) D_xu^i)(t,x),B_iu^i(t,x))\rbrace = 0, \, \, (t,x) \in [0,T) \times \R^k; \\\\
u^i(T,x) = h_i(x),\, x\in \R^k,
\end{cases}
\end{equation}
where the operators $\mathcal{L}$, $\mathcal{K}$ and $(B_i)_{\ii}$ are defined as follows:
\begin{equation}
\begin{aligned}
&\mathcal{L}v(t,x) :=  b(t,x)^\top D_x v(t,x) + \txst \frac{1}{2}\mbox{Tr}[(\sigma\sigma^\top)(t,x)D_{xx}^2 v(t,x)],\\[3pt]
&\mathcal{K}v(t,x) := \textstyle \int_E (v(t,x+\beta(x,e))- v(t,x)- \beta(x,e)^\top D_x v(t,x))\lambda(de) \, \, \, \mbox{ and }\\[5pt]
& B_iv(t,x):= \textstyle \int_E \gamma_i(x,e) (v(t,x+\beta(x,e))- v(t,x))\lambda(de),\,\,\ii.
\end{aligned}
\end{equation}
The notations $D_x v$ and $D^2_{xx}v$ stand for the gradient and Hessian matrix of $v$ with respect to $x$ respectively, while $(.)^\top$ is the transpose and $ \lambda(.)$ is a L\'evy measure on $E := \R^l-\{0\}$.

The main tool to tackle system (\ref{sys1.1}) is to deal with the following system of reflected backward stochastic differential equations (RBSDEs for short) with interconnected obstacles when the noise is driven by a Brownian motion $B := (B_s)_{s \leq T}$ and an independent Poisson random measure $\mu$. A solution for such a system are quadruples of adapted stochastic processes $(Y^{i,t,x},Z^{i,t,x},K^{i,t,x},V^{i,t,x})_{i=1,...,m}$ such that: 
 $\forall \ii$ and $s \in [0,T],$
\begin{equation}\label{sys1.0}
\begin{split}
\begin{cases}
\vspace{0.3cm}
Y_s^{i,t,x} = h_i(X_T^{t,x})+ \int_s^T \bar f_i(r,X_r^{t,x},(Y_r^{k,t,x})_{k\in \mi},Z_r^{i,t,x}, \int_E V_r^{i,t,x}(e)\gamma_i(X_r^{t,x},e) \lambda(de))dr \\\vspace{0.3cm}
 \qquad \qquad   +K_T^{i,t,x} - K_s^{i,t,x} -\int_s^T Z_r^{i,t,x}dB_r -\int_s^T \int_E V_r^{i,t,x}(e) \tilde{\mu}(dr,de);\\\vspace{0.3cm}
 Y_s^{i,t,x} \geqslant  \displaystyle \max_{j \in \mathcal{I}^{-i}}(Y_s^{j,t,x} -g_{ij}(s));\\\vspace{0.3cm}
 \textstyle {\int_0^T (Y_r^{i,t,x} -\displaystyle \max_{j\in \mathcal{I}^{-i}}(Y_r^{j,t,x} -g_{ij}(r))) dK_r^{i,t,x} = 0},
\end{cases}
\end{split}
\end{equation}
where $(t,x) \in [0,T] \times \R^k$, $ds\lambda(de)$ is the compensator  of $\mu$ and $\tilde{\mu}(ds,de):= \mu(ds,de) - ds\lambda(de)$ its compensated random measure, and finally 
$\mathcal{I}^{-i}:= \mathcal{I}-\lbrace i \rbrace$. The process $\xtx$ is the solution of the following standard differential equation:
\begin{equation}\label{sys2.0}
 \begin{cases}
 dX_s^{t,x} = b(s, X_s^{t,x}) ds + \sigma(s, X_s^{t,x}) dB_s + \int_E \beta(s,X_{s-}^{t,x}, e) \tilde{\mu}(ds,de), \hspace{0.2cm}  s\in [t, T]; \\
 X_s^{t,x} = x \in \R^k, \,s\leq t. 
 \end{cases}
 \end{equation}
The system of reflected BSDEs \eqref{sys1.0} is termed of Markovian type since randomness stems from the process $\xtx$ which is Markovian. On the other hand, it is deeply related to the optimal stochastic switching problem (see e.g. \cite{chassagneux2011note,
hamadene2013viscosity,hamadene2015systems,
hamadene2015viscosity,hu2010multi}, for more details). It has been considered in several works including \cite{hamadene2015viscosity} and \cite{hamadene_mnif-neffati2}.
In \cite{hamadene2015viscosity}, the authors showed that if, mainly, the two following monotonicity conditions: For any $\ii$,
\begin{itemize}
 \item[(a)] $\gamma_i \geqslant 0$, 
 \item[(b)] The function $ q\in \R \mapsto \bar f_i(t,x,(y_k)_{k=1,\dots,m},z,q)$ is non-decreasing, when the components $(t,x,(y_k)_{k=1,\dots,m},z)$ are fixed,
\end{itemize}
are satisfied, then system \eqref{sys1.0} has a unique solution $(Y^{i,t,x},Z^{i,t,x},V^{i,t,x},K^{i,t,x})_{\ii}$. Moreover the Feynman-Kac representation of the processes $(Y^{i,t,x})_{\ii}$ holds true, i.e., there exist deterministic continuous functions 
$(u^i)_{\ii}$ defined on $\spa$ such that 
for any $\tx$ and $\ii$, 
 \begin{equation}\label{rep2.15}Y^{i,t,x}_s=u^{i}(s,\xtx_s),\,\forall s\in [t,T]\mbox{ and then }
 u^{i}(t,x)=
Y^{i,t,x}_t.
\end{equation}  
Finally, it is proved that the functions $(u^i)_{\ii}$ are the unique continuous viscosity solution of IPDEs system \eqref{sys1.1} in the class of functions with polynomial growth. \rc The functions $(u^i)_{\ii}$ are nothing else but the value functions of the underlying optimal switching problem. \nc Later in \cite{hamadene_mnif-neffati2}, the authors also considered both systems \eqref{sys1.1} and \eqref{sys1.0} as well, but without assuming the above monotonicity conditions (a)-(b). They proved that if the Lévy measure $\lambda(.)$ associated with the Poisson random measure $\mu$ is finite, i.e. $\lambda(E)<\infty$, then system \eqref{sys1.0} has a unique Markovian solution
$(Y^{i,t,x},Z^{i,t,x},K^{i,t,x},V^{i,t,x})_{\ii}$, i.e. for which the Feynman-Kac representation \eqref{rep2.15} holds. Moreover those functions $(u^i)_{\ii}$ of \eqref{rep2.15}  are the unique viscosity solution of system \eqref{sys1.1}. In \cite{hamadene_mnif-neffati2}, a property which plays an important role is the representation of the process $(V^{i,t,x})_{\ii}$ via the continuous functions $(u^i)_{\ii}$ and the process $\xtx$ and which reads as: 
\begin{align}
\label{rep2}
      V_s^{i,t,x} (e)=  1_{\{s\ge t\}}\{ u^{i}(s,\xtx_{s-}+& \beta(\xtx_{s-},e))- u^{i}(s,\xtx_{s-})\},\nonumber\\
      & \qquad ds\otimes d\mathbb{P} \otimes d\lambda \mbox{ on } [0,T] \times \Omega \times E.
  \end{align}
Therefore the main objective of this paper is to deal with systems \eqref{sys1.0} and \eqref{sys1.1} in the case when $\lambda(.)$ is not finite, i.e., $\lambda(E)=+\infty$ and without assuming neither (a) nor (b) above. Actually in this work, and this is the novelty of the paper, we show that if $\lambda(.)$ integrates the function $(1\wedge |e|)_{e\in E}$, in combination with other regularity properties on the data $(\bar f_i)_{\ii}$, $(h_i)_{\ii}$ and $(g_{ij})_{i,j\in \cal I}$, then system \eqref{sys1.0} has a unique Markovian solution and the Feynman-Kac representation \eqref{rep2.15} holds true. Finally we show that those functions $(u^i)_{\ii}$ are the unique viscosity solution of \eqref{sys1.1}. The relation \eqref{rep2} which binds the processes 
$(V^{i,t,x})_{\ii}$, the functions $(u^i)_{\ii}$ and the process $\xtx$ is also valid. This relation plays an important role in the proof of our result. Basically our method relies  on the truncation of the measure $\lambda(.)$ in such a way to fall in the framework of a finite Lévy measure, already considered in \cite{hamadene_mnif-neffati2},  and then step by step to deal with the general case. 
\def \fii {(f_i)_{\ii}}
\def \lv{\mbox{L\'evy measure}}
\def \bfii {(\bar f_i)_{\ii}}

In this work, we first consider system \eqref{sys1.0} in the case when the generators $\bfii$ do not depend on the jump parts $(v_i)_{\ii}$ and, by a truncation procedure of the $\lv$, we show that the processes $(V^i)_{\ii}$ of its solution $(Y^i,Z^i,V^i,K^i)_{\ii}$, which exists (see \cite{hamadene2015viscosity}, Prop.4.2), verify \eqref{rep2} where $(u^i)_{\ii}$ are the successive Feynman-Kac representations of $(Y^i)_{\ii}$. To state this result we need to assume that the functions $(g_{ij})_{i,j\in \cal I}$ do not depend on $x$. Later we deal with the general case of generators $\bfii$ depending on the jump components $(v_i)_{\ii}$. Using a recursive scheme, we introduce sequences of deterministic continuous functions $(u^{i,n})_{n\ge 0}$, $\ii$, through the Feynman-Kac representations, which we show that they are convergent uniformly in compact sets of 
$[0,T]\times \R^k$ to continuous functions $(u^i)_{\ii}$. Then we construct a Markovian solution $(Y^{i,t,x},Z^{i,t,x},K^{i,t,x},V^{i,t,x})_{\ii}$ for \eqref{sys1.0} which is such that $Y^{i,t,x}_s=u^i(s,\xtx_s)$, $s\in [t,T]$. Moreover we prove that this Markovian solution is unique. Finally we show that the functions $(u^i)_{\ii}$ are the unique viscosity solution of system \eqref{sys1.1}.

The paper is organized as follows.  In Section 2, we provide  notations and assumptions needed in the study of the obliquely RBSDEs with jumps system \eqref{sys1.0} and the related IPDEs system \eqref{sys1.1} as well. Section 3 is mainly devoted to prove the relation \eqref{rep2} in the case when $(\bar f_i)_{\ii}$ do not depend on the jump parts $(v^{i})_{\ii}$. For that we first truncate the Lévy measure $\lambda(.)$ in order to fall in the case of a finite activity Lévy measure since in this latter framework the relation \eqref{rep2} holds true. Then by a limiting procedure we recover \eqref{rep2} in the general case, i.e., when $\lambda(E) =+\infty$. Next,  
we prove the existence of  a solution for system of  RBSDEs  \eqref{sys1.0}, the Feynman-Kac representations \eqref{rep2.15} and the representations \eqref{rep2}. Later we show that the Markovian solution is unique. In Section 4, we prove that the functions $(u^i)_{i=1,m}$ of \eqref{rep2.15} are a unique viscosity solution of \eqref{sys1.1}. Finally in Appendix we collect some results which we need throughout our study. \qed
\section{Preliminaries and notations}
Let $(\Omega, \mathcal{F}, \mathbb{P})$ be a complete probability space, $B = (B_t)_{0\le t\le T}$ is a standard $d$-dimensional Brownian motion and $\mu(dt, de)$ a Poisson random measure, independent of $B$, on $\R^+ \times E$ where $E:= \R^{l} - \lbrace0\rbrace$ is equipped with its Borel $\s$-field $\mathcal{B}(E)$ and with compensator $\nu(dt,de) =\lambda(de)dt$ such
that $\lambda$ is a $\s$-finite measure on $E$. 
We denote by $\mathbb{F}:=(\mathcal{F}_t)_{t \leq T}$ the natural filtration associated with $B$ and $\mu$ which is completed with the $\p$-null sets of $\mathcal  F$. Thus $\mathbb{F}$ verifies the usual conditions. On the other hand, $\tilde \mu$ is the compensated measure of $\mu$. Then for any $A \in \mathcal{B}(E)$ such that $\lambda(A) < \infty$, $\lbrace \tilde{\mu}([0,t]\times A) = (\mu - \nu)([0,t]\times A)\rbrace_{t\leq T }$ is an $\mathbb{F}$-martingale. 

Next for $n\ge 1$, let us set $\mu_n(dt,de):=\un_{\{|e|\ge \frac{1}{n}\}}\mu(dt,de)$, a truncation of $\mu$. Then $\mu_n$ still a Poisson random measure on $\R^+ \times E$ independent of $B$ whose compensator is 
$\nu_n(dt,de):=\un_{\{|e|\ge \frac{1}{n}\}}\lambda(de)dt$. We denote by $\mathbb{F}^{\mu_n}:=(\mathcal{F}^{\mu_n}_t)_{t \leq T}$ the natural filtration associated with $B$ and $\mu_n$ which is once more, completed with the $\p$-null sets of $\mathcal  F$. Note that if $m\ge n$, then for any $t\le T$, $\mathcal{F}^{\mu_n}_t\subset
\mathcal{F}^{\mu_m}_t\subset \mathcal{F}_t$. Finally we assume that 
$\lambda(.)$ integrates the function $ (1\wedge|e|)_{e\in E}$ and $\lambda(E) = \infty$. If $l=1$, $\lambda(de)=|e|^{-\frac{3}{2}}de$ is an example of such a L\'evy measure. 

Now, let us introduce the following spaces:
 \begin{itemize}
 \item[a)] $\mathcal{P}$ (resp. $\mathbf{P}$) is the $\sigma$-algebra of $\mathbb{F}$-progressively measurable (resp. $\mathbb{F}$-predictable) sets on
  $\Omega \times [0,T];$
 \item[b)] $\mathcal{L}^2(\lambda)$ is the space of Borel measurable functions $(\varphi(e))_{e \in E}$ from $E$ into $\R$ such that $\int_E|\varphi(e)|^2 \lambda(de)< \infty$;
 \item[c)] $\ss$ is the space of RCLL (right continuous with left limits), $\mathcal{P}$-measurable and $\R$-valued processes
  $Y:=(Y_s)_{s \leq T}$ such that $\mathbb{E}\big[\displaystyle \sup_{0\leq t\leq T} |Y_s|^2\big] < \infty$;
 \item[d)] $\aa$ is the subspace of $\ss$ of continuous non-decreasing processes $K:= (K_t)_{t\leq T}$ such that $K_0=0$; 
 \item[e)] $\hdd$ is the space of $\mathcal{P}$-measurable and $\R^{d}$-valued processes $Z:=(Z_s)_{s \leq T}$ such that
  $\mathbb{E}\left[\int_0^T |Z_s|^2 ds\right] < \infty$; 
 \item[f)] $\hld$ is the space of $\mathbf{P}$-measurable and $\mathcal{L}^2(\lambda)$-valued processes $V:=(V_s)_{s \leq T}$ such that
  $\mathbb{E}\left[\int_0^T \int_E |V_s(e)|^2 \lambda (de) ds \right]< \infty$.
  \end{itemize}
  \vspace{10px}
  \def \rcl {\mathfrak{rcll}}
  \def \spa {\ss \times \hdd \times \hld}
For {\bc {an RCLL}}  process $(\theta_s)_{s\leq T}$, we define, for any $s\in (0,T],$
 $\theta_{s-} := \lim_{r\nearrow s}\theta_r$ and $\Delta_s\theta := \theta_s- \theta_{s-}$ is the jump size of $\theta$ at $s$.\\
 Now, for any $(t,x) \in [0,T]\times \R^k$, let  $ (X_s^{t,x})_{s \leq T}$ be the process solution of the following stochastic differential equation (SDE for short) of diffusion-jump type: 
 \begin{equation}\label{sys2.1}
 \begin{cases}
 dX_s^{t,x} = b(s, X_s^{t,x}) ds + \sigma(s, X_s^{t,x}) dB_s + \int_E \beta(s,X_{s{-}}^{t,x}, e) \tilde{\mu}(ds,de), \hspace{0.2cm}  s\in [t, T];\\
 X_s^{t,x} = x \in \R^k, \hspace{0.3cm} 0\leq s\leq t,
 \end{cases}
 \end{equation}
 where $b : [0,T]\times \R^k \rightarrow \R^k$ and $\sigma : [0,T]\times \R^k \rightarrow \R^{k\times d}$ are two continuous functions in $(t,x)$ and Lipschitz  $\wt$ $x$, $i.e.$, there exists a positive constant $C$ such that  
 \begin{equation}\label{hyp1}
 |b(t,x) - b(t,x^{\prime})| + |\sigma(t,x) - \sigma(t,x^{\prime})|\leq C|x-x^{\prime}|,\hspace{0.2cm} \forall (t,x,x^{\prime}) \in [0,T]\times \R^{k+k}.
 \end{equation}
 Note that the continuity of $b$, $\sigma$ and \eqref{hyp1} imply the existence of a constant $C$ such that  \begin{equation}\label{hyp2}
 |b(t,x)| + |\sigma(t,x)| \leq C(1+|x|), \hspace{0.2cm} \forall (t,x) \in [0,T]\times \R^{k}.
\end{equation}
Next, let $\beta : \R^k\times E \rightarrow \R^k$ be a $\mathcal{B}(\R^k)\otimes \mathcal{B}(E)$-measurable function such that for some real constant $C_\beta$, $\forall e \in  E \hspace{0.1cm} \mbox{and} \hspace{0.1cm} x, x^{\prime} \in \R^k,$ 
\begin{equation}\label{hyp3}
|\beta(x,e)|\leq C_\beta(1\wedge|e|) \quad \mbox{and} \quad |\beta(x,e)- \beta(x^{\prime},e)| \leq C_\beta|x-x^{\prime}|(1\wedge|e|).
\end{equation}
Conditions \eqref{hyp1}, \eqref{hyp2} and \eqref{hyp3} ensure, for any $\txk$, the existence and uniqueness of a  solution $\xtx$ of equation \eqref{sys2.1} (see \cite{fujiwara1985stochastic} for more details). Moreover, it satisfies: For any $p\ge 2$ and $x,x'\in \R^k$, 
\begin{equation}\label{rep2.2.5}
\mathbb{E}[\sup_{s\leq T} |X_s^{t,x} -x|^p]  \leqslant C(1+|x|^p) \mbox{ and }
\mathbb{E}[\sup_{s\leq T} |X_s^{t,x} -X_s^{t,x'} -(x-x')|^p]  \leqslant C|x-x'|^p.
\end{equation}
Next, let us introduce the following deterministic functions $(\bar f_i)_{\ii}$, $(h_i)_{\ii}$, $(g_{ij})_{i,j\in \mathcal{I}}$ and $(\gamma_i)_{\ii}$ defined as follows : For any $i,j \in \mathcal{I}$,
\begin{align*}
&a)\,\,\,\bar f_i : (t,x,\vec{y},z,q) \in [0,T] \times \R^{k+m+d+1}  \longmapsto \bar f_i(t,x,\vec y,z,q)\in \R \,\,(\vec y:=(y_1,...,y_m))\,;\\
&b)\,\, g_{ij}: t \in [0,T]  \longmapsto g_{ij}(t)\in \R^+ \,; \\&c)\,\,h_i :x\in \R^{k}  \longmapsto h_i(x)\in \R;\\
& d) \,\,\gamma_i : (x,e)\in \R^k \times E \mapsto  \gamma_i  (x,e)\in \R.
\end{align*}
Now let us introduce the following assumptions (H1)-(H5) to which we will refer later. 
\begin{itemize}
\item[\textbf{(H1)}] For any $i \in \lbrace 1,...,m\rbrace$,
\item[(i)] The function  $ \bar f_i$ is continuous in $(t,x)$ uniformly w.r.t. the  variables $(\vec y,z,q)$.
\item[(ii)] The mapping  $(t,x) \mapsto \bar f_i(t,x,\vec 0, 0, 0)$ ($\vec 0=(0)_{\ii}$) is bounded, i.e., for some constant $\check C_i > 0$ we have: 
\begin{equation}
\forall (t,x) \in [0,T]\times \R^k,\,\, |\bar f_i(t,x,\vec 0,0,0)| \leq \check C_i.
\end{equation}
 \item[(iii)] The function $ \bar f_i$ is Lipschitz continuous w.r.t. the variables $(\vy,z,q)$ uniformly in $(t,x)$, i.e., there exists a positive constant $\cif$ such that for any $(t,x) \in [0,T]\times \R^k,$
 $(\vy, z,q)$ and $(\vy_1,z_1,q_1)$ elements of  $\R^{m+d+1}$: 
\begin{equation}\label{lipf}
|\bar f_i(t,x,\vy,z,q)- \bar f_i(t,x,\vy_1,z_1,q_1)| \leq \cif ( |\vy- \vy_1| + |z - z_1| + |q - q_1|).
\end{equation} 
\item[(iv)] For any $\ii$ and $ j\in \mathcal{I}^{-i}$, the mapping $w\in \R \mapsto \bar f_i(t,x,y^1,...,y^{j-1},w,y^{j+1},...,y^m,z,q)$ is non-decreasing whenever the other components   $(t,x,y^1,...,y^{j-1},y^{j+1},...,y^m,z,q)$ are fixed.
 \end{itemize}
 \bigskip
 
\begin{itemize}
\item[\textbf{(H2)}] $\forall i,j \in \lbrace 1,...,m \rbrace,$ $g_{ii} = 0$ and for $i\neq j,$ $g_{ij}(t)$ is non-negative, continuous in $t$ and satisfies the following non free loop property: 
For any $t \in [0,T]$, for any sequence of indices $i_1,...,i_k$ such that $i_1 = i_k$ and $card\lbrace i_1,...,i_{k-1} \rbrace = k-1$ ($k\ge 3$) we have:
\begin{equation}
g_{i_{1}i_{2}}(t) + g_{i_{2}i_{3}}(t) + ... + g_{i_{k-1}i_{1}}(t)> 0.
\end{equation} 
 \item[\textbf{(H3)}] For $i \in \lbrace 1,...,m\rbrace$, the function $h_i$ is continuous and bounded and satisfies the following consistency condition:
     \begin{equation}
          h_i(x) \geqslant \max_{j\in \mathcal{I}^{-i}}(h_j(x)-g_{ij}(T)),\,\,  \forall x \in \R^k.
     \end{equation}
\item[\textbf{(H4)}] For any $i\in \lbrace 1,...,m\rbrace$, $\gamma_i$ is $\mathcal{B}(\R^k) \otimes \mathcal{B}(E)$-measurable and verify: There exists a constant $\cig > 0$ such that for all $e \in E$ and $x,x'\in \rk$, 
\begin{equation}\label{gam}
 |\gamma_i(x,e)|\leq \cig(1\wedge|e|) \mbox{ and }
|\gamma_i(x,e)-\gamma_i(x',e)|\leq \cig(1\wedge|e|)|x-x'|. 
\end{equation}
\item[\textbf{(H5)}] We say that the monotonicity condition holds if: For any $\ii$, 

a) $\g_i\ge 0$.

b) the function $q\in \R\longmapsto \bar f_i(t,x,\vec y,z,q)$ is non-decreasing when the other components $t,x,\vec y$ and $z$ are fixed.\qed
\end{itemize}
\begin{remarque}The assumptions on $(\g_i)_{\ii}$ could be substantially improved. However we stick to this setting in order to be consistent with some other works on this subject such as \cite{bbp,hamadene_mnif-neffati2,hamadene2015viscosity}.\end{remarque}

The following lemma will be useful later. Its proof is given in Appendix. 

\begin{lemme}$\!\!\!{}$ Let $\ii$ be fixed. If $\bar f_i$ verifies (H1)-i), then for any $R>0$, there exists a continuous concave non-decreasing function $\Phi_R^i(.)$ from $\R^+$ to $\R^+$ verifying $\Phi_R^i(0)=0$ and, for any $t\in [0,T]$, $x,x'\in \bar{B}(0,R):=\{x\in \R^k,\,|x|\le R\}$, $\vec y \in \R^m$, $z\in \R^d$ and $q\in \R$:
    \begin{equation}\label{newhyp1}
        |\bar f_i (t,x,\vy,z,q)-\bar f_i
    (t,x',\vy,z,q)|\le \Phi_R^i(|x-x'|).  \end{equation}
\end{lemme}
Proof: See Lemma \ref{lemmeccv} in Appendix.\qed
Now, let us consider the following  system of obliquely RBSDEs with jumps:
  $\forall \ii$ and $s \in [0,T]$,
 \begin{equation}\label{sys2.13}
\begin{split}
\begin{cases}
\vspace{0.3cm} Y^{i,t,x}\in \ss,  Z^{i,t,x} \in \hdd, V^{i,t,x} \in \hld \mbox{ and } K^{i,t,x}\in \aa;\\ \vspace{0.3cm}
Y_s^{i,t,x} = h_i(X_T^{t,x})+ \int_s^T \bar f_i(r,X_r^{t,x},(Y_r^{k,t,x})_{k\in \mi},Z_r^{i,t,x}, \int_E V_r^{i,t,x}(e)\gamma_i(X_r^{t,x},e) \lambda(de))dr \\\vspace{0.3cm}
 \qquad \qquad   +K_T^{i,t,x} - K_s^{i,t,x} -\int_s^T Z_r^{i,t,x}dB_r -\int_s^T \int_E V_r^{i,t,x}(e) \tilde{\mu}(dr,de);\\\vspace{0.3cm}
 Y_s^{i,t,x} \geqslant  \displaystyle \max_{j \in \mathcal{I}^{-i}}(Y_s^{j,t,x} -g_{ij}(s));\\\vspace{0.3cm}
 \textstyle {\int_0^T (Y_r^{i,t,x} -\displaystyle \max_{j\in \mathcal{I}^{-i}}(Y_r^{j,t,x} -g_{ij}(r))) dK_r^{i,t,x} = 0.}
\end{cases}
\end{split}
\end{equation}
This system has already been considered in \cite{hamadene2015viscosity} and the authors have shown the following: 
\begin{theoreme}\label{hz}(\cite{hamadene2015viscosity}, pp.1745) Assume that for any $\ij$, it holds:

i) $\bar f_i$ verifies (H1).

ii) $g_{ij}$ and $h_i$ verify (H2) and (H3) respectively.

iii) $\g_i$ verifies (H4).

iv) The monotonicity hypothesis (H5) is satisfied.
\medskip

\nd Then the system \eqref{sys2.13} has a unique solution 
$(Y^{i,t,x},Z^{i,t,x},V^{i,t,x},K^{i,t,x})_{\ii}$. Moreover, there exist deterministic continuous of polynomial growth functions  $(u^i)_{\ii}$, defined on $[0,T] \times \R^k$, such that for any $\txk$, 
\begin{equation}\label{repY1}
\forall \ii \text{ and }s \in [t,T], \, \, Y_s^{i,t,x} = u^i(s,X_s^{t,x}). 
\end{equation} Finally $(u^i)_{\ii}$ are the unique viscosity solution of system \eqref{sys1.1}
\end{theoreme}

The system \eqref{sys2.13} has been also considered in \cite{hamadene_mnif-neffati2} without supposing [H5] but at the price of assuming $\lambda(E)<\infty$. Actually it is proved in \cite{{hamadene_mnif-neffati2}} that:
\begin{theoreme}(\label{hmn}\cite{{hamadene_mnif-neffati2}}, {\bc Theorem 2 and Corollary 1)}: Assume that for any $\ij$, the following assumptions are fulfilled:

i) $\bar f_i$ verifies (H1).

ii) $g_{ij}$ and $h_i$ verify (H2) and (H3) respectively.

iii) $\g_i$ verifies (H4).

iv) $\lambda(E)<\infty$. 
\medskip

\nd Then the system \eqref{sys2.13} has a unique Markovian solution 
$(Y^{i,t,x},Z^{i,t,x},V^{i,t,x},K^{i,t,x})_{\ii}$, i.e., a solution which satisfies the following Feynman-Kac representation:
\begin{equation}\label{repY12}
\forall \ii \text{ and }s \in [t,T], \, \, Y_s^{i,t,x} = u^i(s,X_s^{t,x})
\end{equation}
where $(u^i)_{\ii}$ are deterministic continuous of polynomial growth functions  defined on $[0,T]\times \R^k$. Moreover the following relation is satisfied:
\begin{align}\label{repV2}
V_s^{i,t,x}(e) = 1_{\{s\ge t\}}( u^{i}(s,\xtx_{s-}+& \beta(\xtx_{s-},e))- u^{i}(s,\xtx_{s-})),\nonumber\\
& \qquad ds\otimes d\mathbb{P} \otimes d\lambda \mbox{ on } [0,T] \times \Omega \times E.
\end{align}Finally $(u^i)_{\ii}$ are the unique viscosity solution of system \eqref{sys1.1}
\end{theoreme}
\begin{remarque}In the proof of this result, the representation \eqref{repV2} of the jump processes $V^{i,t,x}$ plays an important role. \qed
\end{remarque}
\section{Systems of Obliquely RBSDEs with Jumps and related IPDEs when $\lee=+\infty$.}
In this section, we also consider system \eqref{sys2.13} but without assuming neither $\lee <\infty$ nor the monotonicity conditions [H5]. 
Relaxing these hypotheses, we establish a new existence result of a solution for  the system \eqref{sys2.13}. The proof is rather long and will be obtained after several intermediary results.

The first step is to consider system \eqref{sys2.13} when $\bar{f}_i(.)$, $\ii$, do not depend on the jump part. The goal is to show that \eqref{repV2} is also valid. Actually let us consider the following system:
$\forall \ii$ and $s \in [0,T]$,
 \begin{equation}\label{sys2.14}
\begin{split}
\begin{cases}
\vspace{0.3cm} Y^{i,t,x}\in \ss,  Z^{i,t,x} \in \hdd, V^{i,t,x} \in \hld \mbox{ and } K^{i,t,x}\in \aa;\\ \vspace{0.3cm}
Y_s^{i,t,x} = h_i(X_T^{t,x})+ \int_s^T \wfi(r,X_r^{t,x},(Y_r^{k,t,x})_{k\in \mi},Z_r^{i,t,x})dr \\\vspace{0.3cm}
 \qquad \qquad   +K_T^{i,t,x} - K_s^{i,t,x} -\int_s^T Z_r^{i,t,x}dB_r -\int_s^T \int_E V_r^{i,t,x}(e) \tilde{\mu}(dr,de);\\\vspace{0.3cm}
 Y_s^{i,t,x} \geqslant  \displaystyle \max_{j \in \mathcal{I}^{-i}}(Y_s^{j,t,x} -g_{ij}(s));\\\vspace{0.3cm}
 \textstyle {\int_0^T (Y_r^{i,t,x} -\displaystyle \max_{j\in \mathcal{I}^{-i}}(Y_r^{j,t,x} -g_{ij}(r))) dK_r^{i,t,x} = 0.}
\end{cases}
\end{split}
\end{equation}
\def \iji {i,j \in \i}
By Hamadene-Zhao's result (see {\bc\cite{hamadene2015viscosity}}, pp.1745), the solution of this system exists and is unique when for any 
$\ij$, $\wf_i$ verify (H1) and $g_{ij}$, $h_i$, $\g_i$ verify (H2),(H3) and (H4) respectively. On the other hand there exist deterministic continuous functions of polynomial growth $(u^i)_{\ii}$ such that: 
For any $\ii$, 
\be\lb{fkrep1}
Y^{i,tx}_s=u^i(s,X^{t,x}_s),\forall s\in [t,T].\ee
The main task now is to show the characterization \eqref{repV2} in this specific case of $\wfi(.), \ii$. For this objective we are going to assume that 
the functions $(\widehat f_i)_{\ii}$ verify some boundedness property, i.e., there exists a constant $\hat C$ such that for any $\ii$, 
\begin{equation} \label{bornitude}
|\widehat{f}_i(t,x,\vec 0,0)|\leq \hat C, \,\,\forall (t,x)\in [0,T]\times \R^k.
\end{equation}
\subsection{Truncation of the $\lv$}
For any $n\geq 1,$ let us introduce a new Poisson random measure $\mu_n$ and its associated compensator $\nu_n$ as follows :
\begin{align*}
\mu_n(ds,de)=  \mathbf{1}_{\lbrace|e|\geq \frac{1}{n}\rbrace} \mu(ds,de) \mbox{ and }\nu_n(ds,de) =  \lambda_n(de)ds := \mathbf{1}_{\lbrace|e|\geq \frac{1}{n}\rbrace} \lambda(de)ds.
\end{align*}
The random measure $\tilde{\mu}_n(ds,de) := (\mu_n- \nu_n)(ds,de),$ is the compensated one associated with $\mu_n$. The main point to notice is that
$\lambda_n(E)=\textstyle  \int_E \lambda_n(de) < \infty, \,\forall n\ge 1.$ \\
Next, let us introduce the process ${}^{n}\!X^{t,x}$ solution of the following SDE of jump-diffusion type:
  \begin{align}\label{SDEtranqué}
 \begin{cases}
 d{}^{n}\!X_s^{t,x} =  b(s, {}^{n}\!X_s^{t,x}) ds + \sigma(s, {}^{n}\!X_s^{t,x}) dB_s + \int_E \beta({}^{n}\!X_{s-}^{t,x}, e) \tilde{\mu}_n(ds,de), \hspace{0.2cm}  s\in [t, T] ;\\[3pt]
 {}^{n}\!X_s^{t,x} = x \in \R^k, \hspace{0.3cm} 0\leq s\leq t.  
 \end{cases}
 \end{align}
 Note that thanks to the assumptions on $b, \sigma$ and $\beta$, the solution of \eqref{SDEtranqué} exists and is unique. Moreover, it satisfies the same estimates as in \eqref{rep2.2.5}. Indeed, the measure $\lambda_n(.)$ is just a truncation at the origin of the measure $\lambda(.)$ which integrates $(1 \wedge |e|)_{e \in E}$ and then $(1 \wedge |e|^2)_{e \in E}$. In the following lemma we collect some results of the process ${}^{n}\!X^{t,x}$.
 \begin{lemme}\label{truncation}
 The process ${}^{n}\!X^{t,x}$ satisfies the following properties: 
 \begin{itemize}
     \item[(i)] For any $p\geq 2$, there exists a constant $C$ such that 
     \begin{equation}\label{estimXn}
         \E \Big[\displaystyle \sup_{s\leq T} |{}^{n}\!X_s^{t,x}|^p\Big] \leq C (1+|x|^p).
     \end{equation}
     \item[(ii)] For any $p\ge 1$ and $m\geq n \geq 1$,
     \begin{equation}\label{convergenceXn}
          \displaystyle \E \Big[\displaystyle \sup_{s\leq T} |{}^{n}\!X_s^{t,x}-{}^{m}\!X_s^{t,x} |^{2p}\Big]\le C \Big\{\Big(\underbrace{\int_{\lbrace\frac{1}{m} \leq |e| \leq  \frac{1}{n}\rbrace} (1\wedge|e|^2) \lambda(de) \Big)^{p}+
          \int_{\lbrace\frac{1}{m} \leq |e| \leq  \frac{1}{n}\rbrace} (1\wedge|e|^{2p}) \lambda(de)}_{\Lambda^{n,m}_p}
          \Big\}.
     \end{equation}
     \item[(iii)] For any $p\ge 1$, 
     \begin{equation}\label{cvgence2}
         \lim_{n\rw \infty} \displaystyle \E \Big[\displaystyle \sup_{s\leq T} |{}^{n}\!X_s^{t,x}-X_s^{t,x} |^{2p}\Big]=0.
     \end{equation}
 \end{itemize}
 \end{lemme}
\proof See the proof in Appendix.\qed
\def \bs {\bigskip}
\bs
\subsection{Introduction and analyze of the approximating scheme.}
Let us consider the following obliquely RBSDEs with jumps (similar as RBSDEs \eqref{sys2.14} but with $^n\!X$ instead of $X$):\\ For any $\ii$ and $s \in [0,T]$,
 \begin{equation}\label{BSDEtranqué}
\begin{split}
\begin{cases}
\vspace{0.3cm} {}^{n}\!Y^{i,t,x}\in \ss,  {}^{n}\!Z^{i,t,x} \in \hdd, {}^{n}\!V^{i,t,x} \in \hln, \mbox{ and } {}^{n}\!K^{i,t,x}\in \aa;\\ \vspace{0.3cm}
{}^{n}\!Y_s^{i,t,x} = h_i({}^{n}\!X_T^{t,x})+ \int_s^T  \widehat{f}_i(r,{}^{n}\!X_r^{t,x},({}^{n}\!Y_r^{k,t,x})_{k\in \mi},{}^{n}\!Z_r^{i,t,x})dr \\\vspace{0.3cm}
 \qquad \qquad   +{}^{n}\!K_T^{i,t,x} - {}^{n}\!K_s^{i,t,x} -\int_s^T {}^{n}\!Z_r^{i,t,x}dB_r -\int_s^T \int_E {}^{n}\!V_r^{i,t,x}(e) \tilde{\mu}_n(dr,de);\\\vspace{0.3cm}
 {}^{n}\!Y_s^{i,t,x} \geqslant  \displaystyle \max_{j \in \mathcal{I}^{-i}}({}^{n}\!Y_s^{j,t,x} -g_{ij}(s));\\\vspace{0.3cm}
 \textstyle {\int_0^T ({}^{n}\!Y_s^{i,t,x} -\displaystyle \max_{j\in \mathcal{I}^{-i}}({}^{n}\!Y_s^{j,t,x} -g_{ij}(s))) d{}^{n}\!K_s^{i,t,x} = 0.}
\end{cases}
\end{split}
\end{equation}
We then have:
\begin{proposition}\label{firstresult}
\noindent Assume that:

(i) the functions $(g_{ij})_{i,j \in \mathcal{I}}$ and $(h_i)_{i \in \mathcal{I}}$  verify Assumptions (H2)-(H3).

(ii) For any $\ii$, the function $\widehat{f}_i$ 
 verifies (H1) (and then \eqref{bornitude}). 
 
\noindent Then:
 \begin{itemize}
 \item[a)] The system \eqref{BSDEtranqué} has a unique solution $({}^{n}\!Y^{i,t,x}, {{}^{n}\!Z}^{i,t,x}, {}^{n}\!V^{i,t,x}, {}^{n}\!K^{i,t,x})_{i\in \mi}$.
 \item[b)] There exist deterministic continuous uniformly bounded functions  $(u^{i,n})_{\ii}$, defined on $[0,T] \times \R^k$, such that:
\begin{equation}\lb{repynuin}
 \forall s \in [t,T], \, \,{}^{n}\!Y_s^{i,t,x} = u^{i,n}(s,{}^{n}\!X_s^{t,x}).
\end{equation}
\item[c)]For any $i \in \mathcal{I}$,
\begin{align}\label{repvi}
    {}^{n}\!V_s^{i,t,x}(e) = 1_{\{s\ge t\}}( u^{i,n}(s,{}^{n}\!\xtx_{s-}+& \beta({}^{n}\!\xtx_{s-},e))- u^{i,n}(s,{}^{n}\!\xtx_{s-})),\nonumber\\ & \qquad ds\otimes d\mathbb{P} \otimes d\lambda_n \mbox{ on } [0,T] \times \Omega \times E.
\end{align}
\end{itemize} 
\end{proposition}
\proof The existence and uniqueness of the solution of system \eqref{BSDEtranqué} and the deterministic functions $(u^{i,n})_{\ii}$ that satisfy (\ref{repynuin}) are given in \cite{hamadene2015viscosity}, pp.1745.  The property \eqref{repvi} is a consequence of the Feynman-Kac representation \eqref{repynuin} and the fact that $\lambda_n(E)<\infty $  (see \cite{hamadene_mnif-neffati2}, Corollary 1 for more details). Finally it remains to show that $u^{i,n}$, $\ii$, are uniformly bounded with respect to $n$.

So let $(\bar Y, \bar Z)$ be the solution of the following standard BSDE: for any $s\leq T$, 
\begin{equation*}
    \begin{cases}
    \bar Y \in \ss,\, \bar Z \in \hdd;\\
    \bar Y_s = \bar C + \int_s^T \big\lbrace \bar C+ m\,C^y_{\hat f}\,\bar Y_r + C^z_{\hat f}\,|\bar Z_r| \big\rbrace dr - \int_s^T \bar Z_rdB_r;
    \end{cases}
\end{equation*}
where $C_{\hat f}^y$ , $C_{\hat f}^z$ are the maximum of the Lipschitz constants of the functions $(\widehat f_i)_{\ii}$ w.r.t. $\vy$ and $z$ respectively and $\bar C$ is a constant such that for any $\ii$, 
$$|h_i(x)|+|\widehat{f}_i(t,x,\vec 0,0)|\leq \bar C, \,\,\forall (t,x)\in [0,T]\times \R^k.
$$
Then by comparison, which is valid in our case since $(\widehat f_i)_{\ii}$, do not depend the jump component, we have:
\begin{eqnarray*}\label{majbor}
|^nY_s^{i,t,x}|\le \bar Y_s,\,\, \forall s\in [t,T].
\end{eqnarray*}
But $\bar Z=0$, $ds\otimes d\pr$-a.e and $\bar Y$ is deterministic and continuous. Finally it is enough to take $s=t$ in the previous equality to conclude that $(u^{i,n})_{\ii}$ are uniformly bounded. 
\medskip
\begin{remarque}
    The following characterization of ${}^{n}\!Y^{i,t,x}$ as a Snell envelope holds true: 
\begin{align}\label{envsnell}
     \forall s\leq T,\,\,{}^{n}\!Y_s^{i,t,x}= \displaystyle\mbox{esssup}_{\tau \geq s}\E\Big[\int_s^{\tau}&\widehat{f}_i(r,{}^{n}\!X_r^{t,x},({}^{n}\!Y_r^{k,t,x})_{k\in \mi},{}^{n}\!Z_r^{i,t,x})dr +  h_i({}^{n}\!X_{T}^{t,x})\mathbf{1}_{\lbrace \tau=T\rbrace}\nonumber\\
     &+\max_{j \in \mathcal{I}^{-i}}({}^{n}\!Y_{\tau}^{j,t,x} -g_{ij}(\tau))\mathbf{1}_{\lbrace \tau<T\rbrace}\big|\mathcal{F}_s\Big]
\end{align} since its jump times are inaccessible. The stopping times $\t$ are w.r.t $\mathbb{F}$ since 
$\mu_n$ is a truncation of $\mu$ and then $\mathbb{F}^{\mu_n}$-martingales are $\mathbb{F}$-martingales (see e.g. \cite{hamahassani}, pp. 134 for more details). 
\qed
\end{remarque}
Now we are going to focus on the convergence of the sequences $(u^{i,n})_{\ii}$. 
\begin{proposition}\label{cvuniloc}Assume that assumptions $(i)-(ii)$ of Proposition \ref{firstresult} hold. Then there exist continuous bounded functions $(\check u^i)_{\ii}$ such that for $\ii$, $M\ge 1$,
 $$
 \sup_{t\in [0,T], |x|\le M}|u^{i,n}(t,x)-\check u^i(t,x)|\rw 0 \mbox{ as }n\rw \infty. 
 $$    
\end{proposition}
\proof The proof is rather long and divided into two steps. The main difficulty comes from the fact that the Poisson random measures $\mu_n$ change and do not generate the same filtration. However there exists an embedding phenomenon, i.e. $\mathbb{F}^{\mu_n}\subset \mathbb{F}^{\mu_m}$ if $m\ge n$, which we are going to exploit.  
\medskip 

\noindent \tbf{Step 1}: Representation of ${}^{n}\!Y^{i,t,x}$ as a value function of an optimal switching problem.
\bs

 \noindent In this step, we aim at representing ${}^{n}\!Y^{i,t,x}$ as the value of  an optimal switching problem. 
 Let  $\delta:= (\theta_k, \alpha_k)_{k \geq 0}$ be an admissible strategy of switching, $i.e.$, $(\theta_k)_{k \geq 0}$ is an increasing sequence of $\mathbb{F}^{\mu_n}$-stopping times with values in $[0,T]$ such that $\mathbb{P}[\theta_k < T, \forall k \geq 0] = 0$ and for any $k \geq 0$, $\alpha_k$ is a random variable $\mathcal{F}^{\mu_n}_{\theta_{k}}$-measurable with values in $\mathcal{I}$. Next, with the admissible strategy $\delta:= (\theta_k, \alpha_k)_{k \geq 0}$ is associated a switching cost process $(A_s^{\delta})_{s \leq T}$ defined by:
 \begin{equation}
 \forall s < T,\,\, A^{\delta}_s := \displaystyle \sum_{k \geq 1} g_{\alpha_{k-1} \alpha_{k}}(\theta_k)\mathbf{1}_{\lbrace \theta_{k} \leq s \rbrace} \mbox{ and } A^{\delta}_T= \lim_{s\rightarrow T}A^{\delta}_s.
 \end{equation}
The process $(A^{\delta}_s)_{s \leq T}$ is RCLL and non-decreasing. Now, for $s \leq T$, let us set $\eta_s:= \alpha_0\un_{\lbrace\theta_{0\rbrace}}(s) +\displaystyle \sum_{k\geq 1}\alpha_{k}\mathbf{1}_{\{\theta_{k} <s\leq \theta_{k+1}\}}$ which stands for the mode indicator of the system at time $s$. The process $(\eta_s)_{s\le T}$ is in bijection with the strategy $\delta$. Finally, for any fixed $s\leq T$ and $i \in \mathcal{I}$, let us denote by $\aimns$ the following set of admissible strategies:
\begin{align*}
\mathcal{A}_s^{i,\mu_n}:= \lbrace\delta:= (\theta_k, \alpha_k)_{k \geq 0}& \ \mbox{admissible strategy such that}\ \theta_0=s, \alpha_0=i \\
&\mbox{and}\ \mathbb{E}[(A^{\delta}_T)^2] < \infty \rbrace.
\end{align*}
Let  $\delta:= (\theta_k, \alpha_k)_{k \geq 0} \in \aimn_t$ and let us define the triplet of $\mathbb{F}^{\mu_n}$-progressively measurable processes $({}^{n}\!P_s^{\d,t,x},{}^{n}\!N_s^{\d,t,x},{}^{n}\!Q_s^{\d,t,x})_{s \leq T}$ as follows:  \begin{equation}\label{equ8}\left\{\begin{array}{l}
 {}^{n}\!P^{\d,t,x} \mbox{ is RCLL and } \E[\sup_{s\le T}|{}^{n}\!P^{\d,t,x}_s|^2]<\infty \,;\\\\
 \E[\int_0^T|{}^{n}\!N^{\d,t,x}_s|^2ds]+\E[\int_0^T\int_E |{}^{n}\!Q^{\d,t,x}_s(e)|^2\lambda_n(de)ds]<\infty;\\\\ 
{}^{n}\!P_s^{\d,t,x} = h^{\d}({}^{n}\!X_T^{t,x})  + \int_s^T  \widehat{f}^{\d}(r,{}^{n}\!X_r^{t,x},\ynr,{}^{n}\!N_r^{\d,t,x})dr - \textstyle \int_s^T {}^{n}\!N_r^{\d,t,x}dB_r \\\\
 \qquad \qquad - \textstyle\int_s^T \int_E {}^{n}\!Q_r^{\d,t,x}(e)\tilde{\mu}_n(dr,de)- A_T^{\d} + A_s^{\d},\, s\le T, \end{array}\right.
 \end{equation}
 where for any $x\in \rk$ and $z\in \R^d$, 
\begin{align}\label{equfd}
 h^{\d}(x)&:= \displaystyle \sum_{k \geq 0}h_{\alpha_{k}}(x) \mathbf{1}_{[\theta_{k} \leq T < \theta_{k+1}]}\quad\mbox{ and }\nonumber\\[3pt]
\widehat{f}^{\delta}(s,x,\yns,z) &:=  \displaystyle \sum_{k \geq 0} \widehat{f}_{\alpha_{k}}(s,x,\yns ,z)\mathbf{1}_{[\theta_{k} \leq s < \theta_{k+1}]}.
 \end{align}
 Those series contain only a finite many terms as $\d$ is admissible.
Note that, in \eqref{equ8}, the generators $\widehat{f}^{\d}$ do not depend  neither on ${}^{n}\!P^{\d,t,x}$ nor on  ${}^{n}\!Q^{\d,t,x} \in \hln.$ However it depends on $\ynr$ which is already defined.
Next, by a change of variable, the existence of $({}^{n}\!P^{\d,t,x} - A^{\d},{}^{n}\!N^{\d,t,x}, {}^{n}\!Q^{\d,t,x})$ stems from the standard existence result of solutions of BSDEs with jumps by Tang-Li \cite{tang1994necessary} since its generator$ \widehat{f}^{\delta}(s,{}^{n}\!X_s^{t,x},\yns,z)$ is Lipschitz $\wt$ $z$ and $A_T^\d$ is square integrable. On the other hand let 
$({}^{n}\!\underbar Y^{i,t,x},  {}^{n}\!\underbar Z^{i,t,x}, {}^{n}\!\underbar V^{i,t,x}, {}^{n}\!\underbar K^{i,t,x})_{\ii}$ be the solution of the following system: $\forall i = 1,...,m$ and  $s \in [0,T]$,
 \begin{equation}
 \label{BSDEtronquebis1}
\begin{split}
\begin{cases}
\vspace{0.3cm} {}^{n}\!\underbar Y^{i,t,x}\in \ss,  {}^{n}\!\underbar Z^{i,t,x} \in \hdd, {}^{n}\!\underbar V^{i,t,x} \in \hln \mbox{ and } {}^{n}\!\underbar K^{i,t,x}\in \aa;\\ \vspace{0.3cm}
{}^{n}\!\underbar Y_s^{i,t,x} = h_i({}^{n}\!X_T^{t,x})+ \int_s^T  \widehat{f}_i(r,{}^{n}\!X_r^{t,x},({}^{n}\!Y_r^{k,t,x})_{k\in \mi},{}^{n}\!\underbar Z_r^{i,t,x})dr +{}^{n}\!\underbar K_T^{i,t,x} - {}^{n}\!\underbar K_s^{i,t,x} \\\vspace{0.3cm}
 \qquad \qquad   -\int_s^T {}^{n}\!\underbar Z_r^{i,t,x}dB_r -\int_s^T \int_E {}^{n}\!\underbar V_r^{i,t,x}(e) \tilde{\mu}_n(dr,de);\\\vspace{0.3cm}
 {}^{n}\!\underbar Y_s^{i,t,x} \geqslant  \displaystyle \max_{j \in \mathcal{I}^{-i}}({}^{n}\!\underbar Y_s^{j,t,x} -g_{ij}(s));\\\vspace{0.3cm}
 \textstyle {\int_0^T ({}^{n}\!\underbar Y_s^{i,t,x} -\displaystyle \max_{j\in \mathcal{I}^{-i}}({}^{n}\!\underbar Y_s^{j,t,x} -g_{ij}(s))) d{}^{n}\!\underbar K_s^{i,t,x} = 0},
\end{cases}
\end{split}
\end{equation}
 whose solution exists and is unique, as mentionned previously. Therefore the link between the solution of system \eqref{BSDEtronquebis1} and optimal switching problems implies that (one can see \cite{hamadene2015systems} for more details) :
 \begin{equation}\label{fv1}
 {}^{n}\!\underbar Y_s^{i,t,x}=\mbox{esssup}_{\d \in \mathcal{A}_s^{i,\mu_n}}({}^{n}\!P_s^{\d,t,x} - A_s^{\d}).\end{equation} But the processes 
 $({}^{n}\!Y^{i,t,x},{}^{n}\!Z^{i,t,x},{}^{n}\!V^{i,t,x},{}^{n}\!K^{i,t,x})_{\ii}$ solution of \eqref{BSDEtranqué}
 is also a solution of \eqref{BSDEtronquebis1}, therefore by uniqueness of the solution of \eqref{BSDEtronquebis1} one deduces that         
\begin{equation}\label{fv2}
 {}^{n}\!\underbar Y_s^{j,t,x}={}^{n}\!Y_s^{i,t,x} = \displaystyle \underset{\d \in \mathcal{A}_s^{i,\mu_n}}{\mbox{esssup}}\,({}^{n}\!P_s^{\d,t,x} - A_s^{\d})  = ({}^{n}\!P_s^{\delta^*\!,t,x} - A_s^{\delta^*}).
 \end{equation}
for some $\delta^{*} \in \mathcal{A}_s^{i,\mu_n}$. It means that $\d^*$ is an optimal strategy of this switching control problem. 
\medskip

\noindent \tbf{Step 2}: Uniform convergence of $(u^{i,n})_{n\ge 1}$, $\ii$, on compact sets of $[0,T]\times \R^k$.\medskip

\noindent So let $m\ge n \geq 1$. Then we obviously have $\mathbb{F}^{\mu_n}\subset \mathbb{F}^{\mu_m}$. Next for any $\ii$ and $s\in [t,T]$, let us set:
\begin{align*}
F_i^{n,m}( s,\omega,z) &:=
  \widehat{f}_i(s,{}^{n}\!\xtx_s, \yns, z) \vee  \widehat{f}_i(s,{}^{m}\!\xtx_s, \yms, z),\,z\in \R^d, \mbox{ and }\\[3pt]
H_i^{n,m}&:=h_{i}({}^{n}\!X_T^{t,x}) \vee   h_{i}({}^{m}\!X_T^{t,x}).
\end{align*}
Let us consider now the solution, denoted by  $({}^{n,m}\! Y^{i,t,x}, {}^{n,m}\! Z^{i,t,x}, {}^{n,m}\! V^{i,t,x},{}^{n,m}\! K^{i,t,x})_{\ii}$, of the obliquely reflected BSDEs with jumps associated with $((F^{n,m}_i)_{i\in \mathcal{I}}, (g_{ij})_{i,j \in \mathcal{I}},(H_i^{n,m})_{i \in \mathcal{I}})$ which exists and is unique since 
$$\forall \ii, H_i^{n,m}\ge \max_{j\neq i }(H_j^{n,m}-g_{ij}(T)).$$
Note that this condition is satisfied since the switching costs $g_{ij}$, $i,j\in \cal I$, do not depend on $x$ and this is the main reason of assuming $g_{ij}$ independent of $x$. Next, 
let $\stt$ and $\delta:= (\theta_k, \alpha_k)_{k \geq 0}$ a strategy of $\aimm_s$ and let us define the triplet of $\mathbb{F}^{\mu_m}$-progressively measurable processes $({}^{n,m}\!P_r^{\d,t,x},{}^{n,m}\!N_r^{\d,t,x},{}^{n,m}\!Q_r^{\d,t,x})_{r \leq T}$ as follows:  \begin{equation}\label{eq8x}\left\{\begin{array}{l}
 {}^{n,m}\!P^{\d,t,x} \mbox{ is RCLL and } \E[\sup_{r\le T}|{}^{n,m}\!P^{\d,t,x}_r|^2]<\infty; \\\\
 \E[\int_0^T|{}^{n,m}\!N^{\d,t,x}_r|^2dr]+\E[\int_0^T\int_E |{}^{n,m}\!Q^{\d,t,x}_r(e)|^2\lambda_m(de)dr]<\infty;\\\\ 
{}^{n,m}\!P_\t^{\d,t,x} = H_\d^{n,m}  + \int_\t^T  F_\d^{n,m}( r,{}^{n,m}\!N_r^{\d,t,x})dr - \textstyle \int_\t^T {}^{n,m}\!N_r^{\d,t,x}dB_r\\\\ \qquad \qquad \qq - \textstyle\int_\t^T \int_E {}^{n,m}\!Q_r^{\d,t,x}(e)\tilde{\mu}_m(dr,de)- A_T^{\d} + A_\t^{\d},\,\, \t\le T, \end{array}\right.
 \end{equation}
 where, as in \eqref{equfd}, 
\begin{align}\label{eqfd2}
 H_\d^{n,m}:= \displaystyle \sum_{k \geq 0}H_{\alpha_k}^{n,m}\mathbf{1}_{[\theta_{k} \leq T < \theta_{k+1}]}\mbox{ and }
F_\d^{n,m}( r,z)=\displaystyle \sum_{k \geq 0} F^{n,m}_{\alpha_{k}}(r,z)\mathbf{1}_{[\theta_{k} \leq r < \theta_{k+1}]}.
 \end{align}
As in \eqref{fv1}, we have: 
\begin{equation}\label{{fv3}}
 {}^{n,m}\! {Y}_s^{i,t,x} = \displaystyle \mbox{esssup}_{\d \in \aimm_s}({}^{n,m}\! {P}_s^{\d,t,x} - A_s^{\d,t,x}) = ({}^{n,m}\!  P_s^{\bar\delta,t,x} - A_s^{\bar\delta,t,x}),
\end{equation}
where $\bar \d$ belongs to $\aimm_s$ and depends on $n,m$ which we omit as there is no possible confusion ($\bar \d$ is the optimal strategy for the underlying optimal switching problem). Therefore by \eqref{99}, we have the following estimate for $\bar \d$: For any $q\ge 1$, there exists a constant $C_q$ which does not depend on $n,m$ such that 
\begin{equation}\lb{apdix31}
\E[(A_T^{\bar \d})^q]\le C_q.\end{equation}
Next the assumptions {\bf{(H1)}} and {\bf{(H3)}} on $h_i(.)$ and $f_i(.)$, combined with \eqref{apdix31}, and finally by using basic methods in BSDEs (of which It\^o's formula with $w\in \R\mapsto w^q$ ($q\ge 2$)) one deduces that for any $q\ge 1$, there exists a constant which we still denote $C_q$ such that:  
\begin{equation}\lb{estimhdnm}
\E[\{\int_0^T|{}^{n,m}\!N^{\bar \d,t,x}_r|^2dr+\int_0^T\int_E |{}^{n,m}\!Q^{\bar \d,t,x}_r(e)|^2\lambda_m(de)dr\}^q]\le C_q.
\end{equation}
{\bc The proof of inequality~\eqref{estimhdnm} is analogous to that of inequality~\eqref{estilemanx} in the proof of Lemma~\ref{estima}.}
Now recall that 
$({}^{n}\!Y^{i,t,x},{}^{n}\!Z^{i,t,x},{}^{n}\!V^{i,t,x}, {}^{n}\!K^{i,t,x})_{\ii}$ is a solution of the following system of obliquely RBSDEs with jumps: $\forall \ii$ and  $s \in [0,T]$,
 \begin{align}\label{BSDEtranq2}
\begin{cases}
{}^{n}\!Y_s^{i,t,x} = h_i({}^{n}\!X_T^{t,x})+ \int_s^T  \widehat{f}_i(r,{}^{n}\!X_r^{t,x},({}^{n}\!Y_r^{k,t,x})_{k\in \mi},{}^{n}\!Z_r^{i,t,x})dr \\[5pt]
 \qquad \qquad   +{}^{n}\!K_T^{i,t,x} - {}^{n}\!K_s^{i,t,x} -\int_s^T {}^{n}\!Z_r^{i,t,x}dB_r -\int_s^T \int_E {}^{n}\!V_r^{i,t,x}(e) \tilde{\mu}_n(dr,de);\\[5pt]
 {}^{n}\!Y_s^{i,t,x} \geqslant  \displaystyle \max_{j \in \mathcal{I}^{-i}}({}^{n}\!Y_s^{j,t,x} -g_{ij}(s));\\[5pt]
 \textstyle {\int_0^T ({}^{n}\!Y_s^{i,t,x} -\displaystyle \max_{j\in \mathcal{I}^{-i}}({}^{n}\!Y_s^{j,t,x} -g_{ij}(s))) d{}^{n}\!K_s^{i,t,x} = 0.}
\end{cases}
\end{align}
But $\tilde{\mu}_n=\tilde{\mu}_m1_{\{|e|\ge \frac{1}{n}\}}$ and 
$\mathbb{F}^{\mu_n}\subset \mathbb{F}^{\mu_m}$, then $({}^{n}\!Y^{i,t,x},{}^{n}\!Z^{i,t,x},{}^{n}\!V^{i,t,x}1_{\{|e|\ge \frac{1}{n}\}}, {}^{n}\!K^{i,t,x})_{\ii}$ is a solution of the following system of obliquely RBSDEs with jumps: $\forall i = 1,...,m$ and  $s \in [0,T]$,
 \begin{align}\label{BSDEtranq21}
\begin{cases}
{}^{n}\!Y_s^{i,t,x} = h_i({}^{n}\!X_T^{t,x})+ \int_s^T  \widehat{f}_i(r,{}^{n}\!X_r^{t,x},({}^{n}\!Y_r^{k,t,x})_{k\in \mi},{}^{n}\!Z_r^{i,t,x})dr \\[5pt]
 \qquad \qquad   +{}^{n}\!K_T^{i,t,x} - {}^{n}\!K_s^{i,t,x} -\int_s^T {}^{n}\!Z_r^{i,t,x}dB_r -\int_s^T \int_E {}^{n}\!V_r^{i,t,x}(e)1_{\{|e|\ge \frac{1}{n}\}} \tilde{\mu}_m(dr,de);\\[5pt]
 {}^{n}\!Y_s^{i,t,x} \geqslant  \displaystyle \max_{j \in \mathcal{I}^{-i}}({}^{n}\!Y_s^{j,t,x} -g_{ij}(s));\\[5pt]
 \textstyle {\int_0^T ({}^{n}\!Y_s^{i,t,x} -\displaystyle \max_{j\in \mathcal{I}^{-i}}({}^{n}\!Y_s^{j,t,x} -g_{ij}(s))) d{}^{n}\!K_s^{i,t,x} = 0.}
\end{cases}
\end{align}
Therefore we also have:
\begin{equation}\label{fv4}
 {}^{n}\! {Y}_s^{i,t,x} = \displaystyle \mbox{esssup}_{\d \in \aimm_s}({}^{n}\! \bar {P}_s^{\d,t,x} - A_s^{\d,t,x}),
\end{equation}
where for any $\d \in \aimm_s$, $({}^{n}\! \bar {P}^{\d,t,x},{}^{n}\! \bar {N}^{\d,t,x},{}^{n}\! \bar {Q}^{\d,t,x})$ is the solution of 
\begin{equation}\label{equ81}\left\{\begin{array}{l}
 {}^{n}\!\bar P^{\d,t,x} \mbox{ is RCLL and } \E[\sup_{s\le T}|{}^{n}\!\bar P^{\d,t,x}_s|^2]<\infty \,;\\\\
 \E[\int_0^T|{}^{n}\!\bar N^{\d,t,x}_s|^2ds]+\E[\int_0^T\int_E |{}^{n}\!\bar Q^{\d,t,x}_s(e)|^2\lambda_m(de)ds]<\infty;\\\\ 
{}^{n}\!\bar P_s^{\d,t,x} = h^{\d}({}^{n}\!X_T^{t,x})  + \int_s^T  \widehat{f}^{\d}(r,{}^{n}\!X_r^{t,x},\ynr,{}^{n}\!\bar N_r^{\d,t,x})dr - \textstyle \int_s^T {}^{n}\!\bar N_r^{\d,t,x}dB_r \\\\
 \qquad \qquad - \textstyle\int_s^T \int_E {}^{n}\!Q_r^{\d,t,x}(e)\tilde{\mu}_m(dr,de)- A_T^{\d} + A_s^{\d},\, s\le T. \end{array}\right.
 \end{equation}
Next by the comparison result (see Proposition 4.2 in \cite{hamadene2015viscosity}), between the solutions ${}^{n}\!Y^{i,t,x}$ and ${}^{n,m}\!{Y}^{i,t,x}$, and ${}^{m}\!Y^{i,t,x}$ and ${}^{n,m}\!{Y}^{i,t,x}$ (this is possible since: (i) the generators of the systems do not depend on the jump parts; (ii) the corresponding generators and terminal values are comparable), one deduces
that: For any $\ii$,
  \begin{equation*}
  {}^{n}\!Y_s^{i,t,x} \leq {}^{n,m}\! {Y}_s^{i,t,x}\, \, \,   \mbox{ and } 
  \, \, \,   {}^{m}\!Y_s^{i,t,x} \leq {}^{n,m}\! {Y}_s^{i,t,x}. 
  \end{equation*}
   This combined with \eqref{fv2} and \eqref{fv4}, lead to: 
   \begin{equation*}
  {}^{n}\! \bar {P}_s^{\bar \d,t,x}-A_s^{\bar \d,t,x} \leq {}^{n}\!Y_s^{i,t,x} \leq {}^{n,m}\! {Y}_s^{i,t,x}={}^{n,m}\! {P}_s^{{\bar \delta}}-A_s^{\bar \d,t,x}\, \,   \mbox{ and }
  \end{equation*}
  \begin{equation*}
  \, \, \,  {}^{m}\!P_s^{{\bar \delta}} -A_s^{\bar \d,t,x}\leq {}^{m}\!Y_s^{i,t,x} \leq {}^{n,m}\! {Y}_s^{i,t,x}= {}^{n,m}\! {P}_s^{{\bar \delta}}-A_s^{\bar \d,t,x},
  \end{equation*}
  which implies: 
\begin{equation}\label{comp1}
\vspace{0.2cm}|{}^{n}\!Y_s^{i,t,x} - {}^{m}\!Y_s^{i,t,x}| \leq  |{}^{n,m}\! {P}_s^{{\bar \delta}} - {}^{n}\!\bar P_s^{{\bar \delta}}| + |{}^{n,m}\! {P}_s^{{\bar \delta}} - {}^{m}\!P_s^{{\bar \delta}}|.
\end{equation}
Since both terms on the right-hand side of \eqref{comp1} are treated similarly, we focus only on the first one. Applying It\^o's formula with $e^{\alpha \t}|{}^{n,m}\! {P}_\t^{{\bar \delta},t,x} - {}^{n}\!\bar P_\t^{{\bar \delta},t,x}|^2$ and $\alpha \in \R$, yields:  $ \forall \t\le T$,
\begin{equation*}
\begin{array}{ll}
& e^{\alpha \t}|{}^{n,m}\! {P}_\t^{{\bar \delta}} - {}^{n}\!\bar P_\t^{{\bar \delta}}|^2 +  \int_\t^T e^{\alpha r} |{}^{n,m}\! {N}_r^{{\bar \delta}} - {}^{n}\!\bar N_r^{{\bar \delta}}|^2dr +\textstyle{\sum_{\t< r\le T}}\,e^{\alpha r}\Delta_r ({}^{n,m}\! {P}^{{\bar \delta}} - {}^{n}\!\bar P^{{\bar \delta}})^2    \\\\[5pt]
&=  e^{\alpha T} |h^{\bar \d}({}^{m}X_T^{t,x})\vee h^{\bar \d}({}^{n}X_T^{t,x}) -h^{\bar \d}({}^{n}X_T^{t,x})|^2 -\alpha  \int_\t^T  e^{\alpha r} |{}^{n,m}\! {P}_r^{{\bar \delta}} - {}^{n}\!\bar P_r^{{\bar \delta}}|^2 dr  \\\\[5pt] 
&  \qquad \qquad  + 2  \int_\t^T e^{\alpha r} ({}^{n,m}\! {P}_r^{{\bar \delta}} - {}^{n}\!\bar P_r^{{\bar \delta}})
\times \Big\lbrace \widehat{f}^{\bar \d}(r, {}^{m}\!\xtx_r,\ymr,{}^{n,m}\! {N}_r^{{\bar  \delta}})  \vee  \\\\[5pt]
& \qquad  \qq \widehat{f}^{\bar \d}(r, {}^{n}\!\xtx_r,\ynr,{}^{n,m}\! {N}_r^{{\bar \delta}})- \widehat{f}^{\bar \d}(r, {}^{n}\!\xtx_r,\ynr,{}^{n}\! \bar {N}_r^{{\bar  \delta}}) \Big\rbrace dr \\\\[5pt]
&\qquad \qq -2 \int_\t^T e^{\alpha r} ({}^{n,m}\! {P}_r^{{\bar  \delta}} - {}^{n}\!\bar P_r^{{\bar  \delta}})({}^{n,m}\! {N}_r^{{\bar  \delta}} - {}^{n}\!\bar N_r^{{\bar  \delta}})dB_r\\\\[5pt]
& \qquad \qq -2 \int_\t^T\int_E e^{\alpha r} ({}^{n,m}\! {P}_{r-}^{{\bar  \delta}} - {}^{n}\!\bar P_{r-}^{{\bar  \delta}})({}^{n,m}\!  {Q}_r^{{\bar  \delta}}(e) - {}^{n}\!\bar Q_r^{{\bar  \delta}}(e))\tilde{\mu}_m(de,dr).
\end{array}
\end{equation*} 
Next, taking  expectation and using the inequality  $|x\vee y-y|\le |x-y|$ $(x,y \in \R$) to obtain: 
$\forall \t \in [0,T]$,
\begin{equation}\label{3.32}
\begin{array}{ll}
& \E\Big[ e^{\alpha \t}|{}^{n,m}\! {P}_\t ^{{\bar \delta}} - {}^{n}\!\bar P_\t ^{{\bar \delta}}|^2 +  \int_\t^T e^{\alpha r} |{}^{n,m}\! {N}_r^{{\bar \delta}} - {}^{n}\!\bar N_r^{{\bar \delta}}|^2dr +\textstyle{\sum_{\t< r\le T}}\,e^{\alpha r} \Delta_r ({}^{n,m}\! {P}^{{\bar \delta}} - {}^{n}\!\bar P^{{\bar \delta}})^2   \Big]\\\\
& \qq\le \E\Big[ e^{\alpha T} |h^{\bar \d}({}^{m}X_T^{t,x}) -h^{\bar \d}({}^{n}X_T^{t,x})|^2 -
\alpha  \int_\t^T  e^{\alpha r} |{}^{n,m}\! {P}_r^{{\bar \delta}} - {}^{n}\!\bar P_r^{{\bar \delta}}|^2 dr \\\\
&\qq  + 2  \int_\t^T e^{\alpha r} |{}^{n,m}\! {P}_r^{{\bar \delta}} - {}^{n}\!\bar P_r^{{\bar \delta}}|\times \Big\lbrace 
|\underbrace{ \widehat{f}^{\bar \d}(r, {}^{m}\!\xtx_r,\ymr,{}^{n,m}\! {N}_r^{{\bar \delta}}) - \widehat{f}^{\bar \d}(r, {}^{n}\!\xtx_r,\ynr,{}^{n,m}\! {N}_r^{{\bar \delta}})}_{\Sigma^{n,m,\bar \d,t,x}(r)}|\\&\qquad 
\qquad  +|\widehat{f}^{\bar \d}(r, {}^{n}\!\xtx_r,\ynr,{}^{n,m}\! {N}_r^{{\bar \delta}})  -\widehat{f}^{\bar \d}(r, {}^{n}\!\xtx_r,\ynr,{}^{n}\! \bar {N}_r^{{\bar  \delta}})|\Big\rbrace dr \Big].
\end{array}
\end{equation}
First let us deal with the term 
$\E\Big[ |h^{\bar \d}({}^{m}X_T^{t,x}) -h^{\bar \d}({}^{n}X_T^{t,x})|^2\Big]$. We first have, $$|h^{\bar \d}({}^{m}X_T^{t,x}) -h^{\bar \d}({}^{n}X_T^{t,x})|\le \sum_{i\in {\cal I}}|h_{i}({}^{m}X_T^{t,x}) -h_i({}^{n}X_T^{t,x})|. 
$$
\def \fr {\forall}
Next let $R\ge 0$. As $h_i$ is continuous then it is uniformly continuous on $\bar B(0,R)$ and, on the other hand, since it is bounded we have: $\fr \ii$, 
$$
|h_{i}({}^{m}X_T^{t,x}) -h_i({}^{n}X_T^{t,x})|\leq \Psi_R^i\big (
|{}^{m}X_T^{t,x} -{}^{n}X_T^{t,x}| 1_{\{
|{}^{m}X_T^{t,x}|+|{}^{n}X_T^{t,x}|\le R\}}\big)
+2\cih 1_{\{
|{}^{m}X_T^{t,x}|+|{}^{n}X_T^{t,x}|> R\}}
$$where $\Psi_R^i(.)$ is the concave modulus of continuity of $h^i$ on $\bar B(0,R)$. Once more using the boundedness of $h_i(.)$, Jensen's inequality and Markov one to deduce:
\begin{align*}
\E[|h_{i}({}^{m}X_T^{t,x}) -h_i({}^{n}X_T^{t,x})|^2]& \leq 2C_{h_i}\big \{\Psi_R^i\big (
\E[|{}^{m}X_T^{t,x} -{}^{n}X_T^{t,x}| 1_{\{
|{}^{m}X_T^{t,x}|+|{}^{n}X_T^{t,x})|\le R\}}]\big )
\\&\qq\qq +2C_{h_i}\p\{
|{}^{m}X_T^{t,x}|+|{}^{n}X_T^{t,x})|> R\}\\
&\le 
2C_{h_i}\big \{\Psi_R^i\big (
\sqrt{\Lambda_1^{n,m}}\big )+C R^{-2}(1+|x|^{2})\}
\end{align*}
where $C$ is a constant and $\Lambda_1^{n,m}$ is given in \eqref{convergenceXn}.
Now as $(\vec y,z)\mapsto \widehat{f}^{\bar \d}(t,x,\vec y,z)$ is Lipschitz uniformly w.r.t. $(t,x)$ then we have:
$$
|\widehat{f}^{\bar \d}(r, {}^{n}\!\xtx_r,\ynr,{}^{n,m}\! {N}_r^{{\bar \delta}})-\widehat{f}^{\bar \d}(r, {}^{n}\!\xtx_r,\ynr,{}^{n}\! \bar {N}_r^{{\bar  \delta}}) |
\leq  
\underbrace{\sum_{i=1,m}C_{f_i}}_
{\kappa_1}|{}^{n,m}\! {N}_r^{{\bar \delta}}-{}^{n}\! \bar {N}_r^{{\bar  \delta}}|.$$
On the other hand, by \eqref{newhyp1} and since $\ynr$ and $\widehat{f}_i(t,x,0,0)$ are bounded, then:
\begin{equation}
    \begin{aligned}
        |\Sigma^{n,m,\bar \d,t,x}(r)|&\le  \kappa_1|\ymr-\ynr|\\[5pt]
        & \quad \qq +|\widehat{f}^{\bar \d}(r, {}^{m}\!\xtx_r,\ynr,{}^{n,m}\! {N}_r^{{\bar \delta}})-\widehat{f}^{\bar \d}(r, {}^{n}\!\xtx_r,\ynr,{}^{n,m}\! {N}_r^{{\bar  \delta}})|\\[5pt]
        & \le \kappa_1|\ymr-\ynr|+
\sum_{\ell\in {\cal I}}\Phi_R^\ell(|{}^{m}\!\xtx_r- {}^{n}\!\xtx_r|)1_{\{|{}^{m}\!\xtx_r|+
|{}^{n}\!\xtx_r|\le R\}}\\&\qq +|\widehat{f}^{\bar \d}(r, {}^{m}\!\xtx_r,\ynr,{}^{n,m}\! {N}_r^{{\bar \delta}})-\widehat{f}^{\bar \d}(r, {}^{n}\!\xtx_r,\ynr,{}^{n,m}\!  {N}_r^{{\bar  \delta}})|1_{\{|{}^{m}\!\xtx_r|+
|{}^{n}\!\xtx_r|\ge R\}}\\
& \le \kappa_1|\ymr-\ynr|+
\sum_{\ell\in {\cal I}}\Phi_R^\ell\big (|{}^{m}\!\xtx_r- {}^{n}\!\xtx_r|1_{\{|{}^{m}\!\xtx_r|+
|{}^{n}\!\xtx_r|\le R\}}\big )\\&\qq +C(1+|
{}^{n,m}\!  {N}_r^{{\bar  \delta}}|)1_{\{\sup_{r\le T}(|{}^{m}\!\xtx_r|+
|{}^{n}\!\xtx_r|)\ge R\}}\end{aligned}
\end{equation}
where $C$ is a constant independent of $n,m$ which may change from line to line. Therefore we have: $\forall \t\le T$, 
\begin{equation*}
\begin{array}{ll}
&\E\Big[ e^{\alpha \t }|{}^{n,m}\! {P}_\t ^{{\bar \delta}} - {}^{n}\!\bar P_\t ^{{\bar \delta}}|^2 +  \int_\t ^T e^{\alpha r} |{}^{n,m}\! {N}_r^{{\bar \delta}} - {}^{n}\!\bar N_r^{{\bar \delta}}|^2dr +\textstyle{\sum_{\t < r\le T}}\,e^{\alpha r} \Delta_r ({}^{n,m}\! {P}^{{\bar \delta}} - {}^{n}\!\bar P^{{\bar \delta}})^2   \Big] \\\\[5pt]
& \leq  e^{\alpha T}\E\big [
\sum_{i \in {\cal I}}
\cih \big \{\Psi_R^i\big (
\sqrt{\Lambda_1^{n,m}}\big )+C R^{-2}(1+|x|^{2})\}\big]
+(4- \alpha + 4\kp_1^2) \E \Big[  \int_\t ^T  e^{\alpha r} |{}^{n,m}\! {P}_r^{{\bar \delta}} - {}^{n}\!\bar P_r^{{\bar \delta}}|^2 dr \Big]   \\\\[5pt]
&\qquad + \frac{1}{2}    \E \Big[ \int_\t ^T e^{\alpha r} |\ymr-\ynr|^2 dr \Big]+ \frac{1}{2} \E \Big[\int_\t ^T e^{\alpha r}  |{}^{n,m}\! {N}_r^{{\bar \delta}}-{}^{n}\! \bar {N}_r^{{\bar  \delta}}|^2dr  \Big]+\\\\[5pt]
&\qquad +\frac{1}{2}    \E \Big[ \int_\t ^T\big \{
\sum_{\ell\in {\cal I}}\Phi_R^\ell\big (|{}^{m}\!\xtx_r- {}^{n}\!\xtx_r|1_{\{|{}^{m}\!\xtx_r|+
|{}^{n}\!\xtx_r|\le R\}}\big )\}^2dr\Big ]\\\\&\qq +
\frac{1}{2}\underbrace{\textstyle \E\Big[ 1_{\{\sup_{r\le T}(|{}^{m}\!\xtx_r|+
|{}^{n}\!\xtx_r|)\ge R\}}\int_\t ^T(1+|
{}^{n,m}\!  {N}_r^{{\bar  \delta}}|)^2dr\Big ]}_{\Xi(\t)}.
\end{array}
\end{equation*} 
\def \txtl {\textstyle}
Let us now deal with $\Xi(\t)$, the last term of the $rhs$ of the last inequality. By using estimates \eqref{estimXn} and estimate \eqref{estimhdnm}, we deduce that for any $\t\le T$, 
\begin{align*} 
    |\Xi(\t)|\textstyle\le C(1+|x|^2)R^{-2}\big \{\E[\big \{\int_\t ^T(1+|
{}^{n,m}\!  {N}_r^{{\bar  \delta}}|)^2dr\big\}^2\big ]\big \}^{\frac{1}{2}}
\leq C(1+|x|^2)R^{-2}.
\end{align*}Next take {\bc{$\alpha=\alpha_0=4+4\kp_1^2$}} to deduce that for any $\t\le T$, 
\begin{equation*}
\begin{aligned}
&\E\Big[ e^{\alpha_0 \t }|{}^{n,m}\! {P}_\t ^{{\bar \delta}} - {}^{n}\!\bar P_\t ^{{\bar \delta}}|^2\Big ]  \leq \tilde \Xi := e^{\alpha_0 T}
\E\big [
\sum_{i \in {\cal I}}2
\cih\big \{\Psi_R^i(
\sqrt{\Lambda_1^{n,m}})+C R^{-2}(1+|x|^{2})\}\big]\\&\quad \txtl 
+\frac{1}{2}    \E \Big[ \int_\t ^T e^{\alpha_0 r} |\ymr-\ynr|^2 dr \Big]+ \\[5pt]
&\quad \txtl +\frac{1}{2}    \E \Big[ \int_\t ^T\big \{
\sum_{\ell\in {\cal I}}\Phi_R^\ell\big (|{}^{m}\!\xtx_r- {}^{n}\!\xtx_r|1_{\{|{}^{m}\!\xtx_r|+
|{}^{n}\!\xtx_r|\le R\}}\big )\}^2dr\Big ]+ C(1+|x|^{2})R^{-2}.
\end{aligned}
\end{equation*} 
In the same way  as previously we have also for any $\t\le T$,
\begin{equation}\label{eq325}
\begin{aligned}
\E\Big[ e^{\alpha_0 \t }|{}^{n,m}\! {P}_\t ^{{\bar \delta}} - {}^{m}\!\bar P_\t ^{{\bar \delta}}|^2\Big ] \le 
\tilde \Xi 
\end{aligned}
\end{equation} 
Take now $\t=s$ and using the inequality \eqref{comp1} to deduce that:
\begin{equation}\label{equ3.28}
\begin{aligned}
&\E\Big[ e^{\alpha_0 s }|{}^{n}\!Y_s^{i,t,x} - {}^{m}\!Y_s^{i,t,x}|^2\Big ] \\&\leq  2\tilde \Xi := 2e^{\alpha_0 T}
\E\big [
\sum_{i \in {\cal I}}
2\cih\big \{\Psi_R^i\big (
\sqrt{\Lambda_1^{n,m}}\big )+C R^{-2}(1+|x|^{2})\}\big]\txtl
+   \E \Big[ \int_s ^T e^{\alpha_0 r} |\ymr-\ynr|^2 dr \Big]\\&
+mC_R\E \Big[ \int_s ^T
\{\sum_{\ell\in {\cal I}}\Phi_R^\ell(|{}^{m}\!\xtx_r- {}^{n}\!\xtx_r|1_{\{|{}^{m}\!\xtx_r|+
|{}^{n}\!\xtx_r|\le R\}})\}dr\Big ]+ 2C(1+|x|^{2})R^{-2}
\end{aligned}
\end{equation}
where $C_R$ is a constant boundedness of $\max_{\ell=1,...,m}\Phi^\ell_R(.)$ on 
$\bar B(0,R)$. Next by the concavity of $\Phi_R^\ell(.)$, Jensen's inequality applied with $\Phi^\ell_R(.)$, estimate \eqref{convergenceXn} and finally  Gronwall's one we deduce that:
\begin{equation}\label{equ3.283}
\begin{aligned}
&\E\Big[ e^{\alpha_0 s }\sum_{i=1,m}|{}^{n}\!Y_s^{i,t,x} - {}^{m}\!Y_s^{i,t,x}|^2\Big ] \\[5pt]
&\leq  C\Big \{2 e^{\alpha_0 T}
\sum_{i \in {\cal I}}
2 \cih\big \{\Psi_R^i\big (
\sqrt{\Lambda_1^{n,m}}\big )+C R^{-2}(1+|x|^{2})\}
 +mTC_R
\sum_{\ell\in {\cal I}}\Phi_R^\ell(\sqrt{\Lambda_1^{n,m}})+ 2C(1+|x|^{2})R^{-2}\Big \}.
\end{aligned}
\end{equation}
Next in taking $s=t$ we deduce that: 
\begin{equation}\label{equ3.281}
\begin{aligned}
& e^{\alpha_0 t }\sum_{i=1,m}|u^{i,n}(t,x)-u^{i,m}(t,x)|^2\\ &\leq  C\Big \{2 e^{\alpha_0 T}
\sum_{i \in {\cal I}}
2 \cih \big \{\Psi_R^i\big (
\sqrt{\Lambda_1^{n,m}}\big )+C R^{-2}(1+|x|^{2})\}
 +mTC_R
\sum_{\ell\in {\cal I}}\Phi_R^\ell(\sqrt{\Lambda_1^{n,m}})+ 2C(1+|x|^{2})R^{-2}\Big \}.
\end{aligned}
\end{equation}
Thus for any $M\ge1$,
\begin{equation}
\begin{aligned}
&\sup_{(t,x)\in [0,T]\times \bar B(0,M)}\sum_{i=1,m}|u^{i,n}(t,x)-u^{i,m}(t,x)|^2\\ 
&\leq  C\Big \{2 e^{\alpha_0 T}
\sum_{i \in {\cal I}}
2 \cih \big \{\Psi_R^i\big (
\sqrt{\Lambda_1^{n,m}}\big )+C R^{-2}(1+|M|^{2})\}
 +mTC_R
\sum_{\ell\in {\cal I}}\Phi_R^\ell(\sqrt{\Lambda_1^{n,m}})+ 2C(1+|M|^{2})R^{-2}\Big \}.
\end{aligned}\label{equ3.29}
\end{equation}
It implies that:
\begin{equation}
\begin{aligned}
&\limsup_{n,m\rw \infty}\sup_{(t,x)\in [0,T]\times \bar B(0,M)}\sum_{i=1,m}|u^{i,n}(t,x)-u^{i,m}(t,x)|^2\le C R^{-2}(1+|M|^{2}).
\end{aligned}\label{equ3.292}
\end{equation}
As $R$ is arbitrary then for any $M\ge 1$, 
\begin{equation}
\begin{aligned}
&\limsup_{n,m\rw \infty}\sup_{(t,x)\in [0,T]\times \bar B(0,M)}\sum_{i=1,m}|u^{i,n}(t,x)-u^{i,m}(t,x)|^2=0.
\end{aligned}\label{equ3.30}
\end{equation}
Then, in combination with the uniform boundedness of $u^{i,n}$ w.r.t $n$, there exist continuous bounded functions $\check u^i$, $\ii$, such that for $\ii$, $M\ge 1$,
 $$
 \sup_{t\in [0,T], |x|\le M}|u^{i,n}(t,x)- \check u^i(t,x)|\rw 0 \mbox{ as }n\rw \infty.
 $$
This completes the proof. \qed
\bigskip

\begin{proposition} \lb{repvi2}
Assume that assumptions $(i)-(ii)$ of Proposition \ref{firstresult} hold. Then for any $\ii$, the components 
$(V_s^{i,t,x})_{s\in [0,T]}$ of the solution of system \eqref{sys2.14} (which exists) satisfy the following relation:
\begin{align}\label{repV22}
V_s^{i,t,x}(e) = 1_{\{s\ge t\}}( \check u^{i}(s,\xtx_{s-}+& \beta(\xtx_{s-},e))- \check u^{i}(s,\xtx_{s-})),\nonumber\\
& \qquad ds\otimes d\mathbb{P} \otimes d\lambda \mbox{ on } [0,T] \times \Omega \times E.
\end{align}
\end{proposition}
\proof It will be obtained after three steps. 

\nd\underline{Step 1}: For any $\ii$, the sequence of processes $(({}^{n}\!Y_s^{i,t,x})_{s\in [t,T]})_{n\geq 1}$ converges to $({\check u}^i(s,\xtx_s))_{s\in [t,T]}$ in $\mathcal{S}^2_{[t,T]}$ (which is 
$\mathcal{S}^2$ reduced to ${[t,T]}$).
\ms

Let $\ii$ be fixed. From Proposition \ref{cvuniloc}, we know that $(u^{i,n})_{n\ge 0}$ converges locally uniformly to $\check u^i$. Then for any $n\geq 1$, we have:
\begin{equation}\label{majuin}
 \begin{aligned}
 & \E\big[\sup_{t\leq s \leq T} \big| u^{i,n}(s,{}^{n}\!\xtx_s) -\check u^{i}(s,\xtx_s)\big|^2 \big]\\ & \leq  2\E\big[\sup_{t\leq s \leq T}  \big|u^{i,n}(s,{}^{n}\!\xtx_s) -\check u^{i}(s,{}^{n}\!\xtx_s)\big|^2 \big]+ 2 \E\big[\sup_{t\leq s \leq T}  \big|\check u^{i}(s,{}^{n}\!\xtx_s) - \check u^{i}(s,\xtx_s)\big|^2 \big].
 \end{aligned}
 \end{equation}
 Let us deal with the first term of the r.h.s of the inequality (\ref{majuin}). Let $\rho >0$ and $x\in \bar B(0,\rho)$, then we have
\begin{align}\label{uinui}
&\E\big[\sup_{t\leq s \leq T}  \big|u^{i,n}(s,{}^{n}\!\xtx_s)- \check{u}^i(s,{}^{n}\!\xtx_s)\big|^2 \big]\nonumber\\[3pt]
&\leq C \Big\{ \E\big[\mathbf{1}_{\lbrace\displaystyle \sup_{s\le T}|{}^{n}\!\xtx_s| > \rho \rbrace}\sup_{t\leq s \leq T} \big|u^{i,n}(s,{}^{n}\!\xtx_s)- \check{u}^i(s,{}^{n}\!\xtx_s)\big| \big]\nonumber \\[3pt] 
& \qquad + \E\big[\mathbf{1}_{\{\displaystyle \sup_{s\le T}
|{}^{n}\!\xtx_s| \le \rho \rbrace}\sup_{t\leq s \leq T} \big|u^{i,n}(s,{}^{n}\!\xtx_s) - \check{u}^i(s,{}^{n}\!\xtx_s)\big| \big]\Big \} ,      
\end{align}
\def \mx{\mbox}
where $C$ is the sum of the boundedness constants of $u^{i,n}$ and $\check{u}^i$. Next, Markov's inequality implies that the first term is dominated by $C_x \rho^{-1}$, where $C_x$ is an appropriate constant which may depend on $x$, since $\check u^i$ and ${}^{n}\!{u}^i$ are bounded and $\E[\sup_{s\le T}
     |{}^{n}\!\xtx_s|]\leq C(1+|x|).$ On the other hand, the second one verifies
\begin{eqnarray*}
 &   \E\big[\mathbf{1}_{ \{\sup_{s\le T}
|{}^{n}\!\xtx_s| \le \rho \rbrace}\sup_{t\leq s \leq T} \big|u^{i,n}(s,{}^{n}\!\xtx_s) - \check{u}^i(s,{}^{n}\!\xtx_s)\big| \big]\\
&\qq \qq \qq \qq \leq \sup_{(t,x)\in [0,T]\times \bar B(0,\rho)}|u^{i,n}(t,x)-\check{u}^i(t,x)|\rw 0 \mx{ as }n\rw \infty.
\end{eqnarray*}     
which goes to zero when $n$ goes to infinity as $u^{i,n}$ converges locally uniformly to $\check{u}^i$ when $n$ goes to infinity .\\
We now focus on the second term of the $r.h.s$ of inequality (\ref{majuin}). Let us denote by $\omega^i_\rho(.)$ the concave modulus of continuity of $\check{u}^i$ on $ [0,T] \times \bar B(0,\rho)$. Then by the boundedness of $\check{u}^i$, we have:
     \begin{align}\label{supcheckui}
         &\E\big[\sup_{t\leq s \leq T}  \big|\check{u}^i(s,{}^{n}\!\xtx_s) - \check{u}^i(s,\xtx_s)\big|^2 \big]\nonumber\\[3pt]
         &
         \leq C \Big\{ \E\big[\mathbf{1}_{\lbrace\displaystyle \sup_{s\le T}|{}^{n}\!\xtx_s|+
     |\xtx_s| \le \rho \rbrace}\sup_{t\leq s \leq T} \big|\check{u}^i(s,{}^{n}\!\xtx_s) - \check{u}^i(s,\xtx_s)\big| \big]\nonumber \\[3pt] 
         & \qquad + \E\big[\mathbf{1}_{\displaystyle\{ \sup_{s\le T}|{}^{n}\!\xtx_s|+
     |\xtx_s| > \rho \rbrace}\sup_{t\leq s \leq T}  \big|\check{u}^i(s,{}^{n}\!\xtx_s) - \check{u}^i(s,\xtx_s)\big|\big]\Big \}\nonumber \\[3pt]
         &
         \leq C \Big\{\E\big[\sup_{t\leq s \leq T} ( \big|\check{u}^i(s,{}^{n}\!\xtx_s) - \check{u}^i(s,\xtx_s)\big |)\mathbf{1}_{\lbrace\displaystyle \sup_{s\le T}\{|{}^{n}\!\xtx_s|+
     |\xtx_s|\} > \rho \rbrace} \big\rbrace\big]\nonumber \\[3pt] 
         & \qquad + \E\big[\sup_{t\leq s \leq T}  \omega^i_\rho(|{}^{n}\!\xtx_s - \xtx_s|\mathbf{1}_{\lbrace\displaystyle \sup_{s\le T}\{|{}^{n}\!\xtx_s|+
     |\xtx_s|\} \le \rho \rbrace})\big]\Big \}.\nonumber \\[3pt]
     \end{align}
     As previously, Markov's inequality implies that the first term is dominated by 
     $C_x \rho^{-1}$ since $\check u^i$ is bounded and $\E[\sup_{s\le T}\{|{}^{n}\!\xtx_s|+
     |\xtx_s|\}]\leq C(1+|x|).$ On the other hand, the second one verifies:
     \begin{align}
      \nb  & \E\big[\sup_{t\leq s \leq T} \omega^i_\rho(|{}^{n}\!\xtx_s - \xtx_s|\mathbf{1}_{\lbrace\displaystyle \sup_{s\le T}\{|{}^{n}\!\xtx_s|+
     |\xtx_s|\} \le \rho \rbrace})\big] \\\nb &\leq
      \E\big[  \omega^i_\rho(\sup_{t\leq s \leq T}|{}^{n}\!\xtx_s - \xtx_s|\mathbf{1}_{\lbrace\displaystyle \sup_{s\le T}\{|{}^{n}\!\xtx_s|+
     |\xtx_s|\} \le \rho \rbrace})\big]\\\nb &\nb \le 
      \omega^i_\rho\Big (\E\big[\sup_{t\leq s \leq T}|{}^{n}\!\xtx_s - \xtx_s|\times \mathbf{1}_{\lbrace\displaystyle \sup_{s\le T}\{|{}^{n}\!\xtx_s|+
     |\xtx_s|\} \le \rho \rbrace}\big]\Big ).
 \end{align}
     The last inequality stems from concavity of $\omega^i_\rho(.)$ and Jensen's inequality. Finally using \eqref{cvgence2} we deduce that 
     $$\lim_{n\rw \infty}\E\big[\sup_{t\leq s \leq T} \omega^i_\rho(|{}^{n}\!\xtx_s - \xtx_s|\mathbf{1}_{\lbrace\displaystyle \sup_{s\le T}\{|{}^{n}\!\xtx_s|+
     |\xtx_s|\} \le \rho \rbrace})\big]=0.$$
     Now as $\rho$ is arbitrary then going back to \eqref{supcheckui}, take the superior limit w.r.t $n$ then the limit w.r.t $\rho$ to deduce that 
     $$
     \limsup_{n\rw \infty} \E\big[\sup_{t\leq s \leq T}  \big|\check{u}^i(s,{}^{n}\!\xtx_s) - \check{u}^i(s,\xtx_s)\big|^2 \big]=0.
     $$
which means that for any $\ii$, the sequence $({}^{n}\!Y^{i,t,x})_{n\ge 0}$ is convergent in $\ss_{[t,T]}$ and its limit is $(\check Y^{i,t,x}_s=\check{u}^i(s,\xtx_s))_{\stt}$. 
\ms

\nd \underline{Step 2}: There exists a constant $\bar C\ge 0$ such that for $n\ge 1$,  \begin{equation}\label{estimZK1}
\begin{aligned}
 \sum_{i=1,m}\E \big[\int_t^T|{}^{n}\!Z_r^{i,t,x}|^2dr + \big({}^{n}\!K_T^{i,t,x}-{}^{n}\!K_t^{i,t,x}\big)^2  +\int_t^T\int_E |{}^{n}\!V_r^{i,t,x}(e)|^2\lambda_n(de)dr\big] \leq \bar{C}.
\end{aligned}
\end{equation}Applying It\^o's formula with $|{}^{n}\!Y_s^{i,t,x}|^2$, we obtain:  $\forall \ii$ and $s \in [t,T]$, 
 \begin{equation*}
\begin{aligned}
&\E \big[|{}^{n}\!Y_s^{i,t,x}|^2 \big] + \E \big[\int_s^T|{}^{n}\!Z_r^{i,t,x}|^2dr \big] + \E \big[\int_s^T\int_E|{}^{n}\!V_r^{i,t,x}(e)|^2 \lambda_n(de)dr \big] \\[3pt] 
& \qquad = \E \big[|h_i({}^{n}\!X_T^{t,x})|^2 \big] + 2\E \big[ \int_s^T {}^{n}\!Y_r^{i,t,x} \widehat{f}_i(r,{}^{n}\!X_r^{t,x},({}^{n}\!Y_r^{k,t,x})_{k\in \mi}, {}^{n}\!Z_r^{i,t,x})dr\big]\\[3pt]
&\qquad \quad+ 2 \E \big[ \int_s^T {}^{n}\!Y_r^{i,t,x} d{}^{n}\!K_r^{i,t,x}\big].
\end{aligned}
\end{equation*}
Then by a linearization procedure of $\widehat{f}_i$,  which is possible since it is Lipschitz $\wt$ $(\vec y,z)$ and using the inequality $2ab\leq \epsilon a^2 + \frac{1}{\epsilon}b^2$, for any $\epsilon > 0$ and $a,b \in \R$, we have: $\fstt$
 \begin{equation*}
\begin{aligned}
 &\E \big[\int_s^T|{}^{n}\!Z_r^{i,t,x}|^2dr \big]  + \E \big[\int_s^T\int_E|{}^{n}\!V_r^{i,t,x}(e)|^2 \lambda_n(de)dr \big]\\[3pt]
 & \qquad\leq  \E \big[|h_i({}^{n}\!X_T^{t,x})|^2 \big]   + 2\E \big[ \int_s^T |{}^{n}\!Y_r^{i,t,x}| \times \Big\lbrace |\widehat{f}_i(r,{}^{n}\!X_r^{t,x},0,0)| + \sum_{l=1,m}|a_r^{i,l,n}||{}^{n}\!Y_r^{l,t,x}|\\[4pt]
& \qquad +  |b_r^{i,n}||{}^{n}\!Z_r^{i,t,x}|\Big\rbrace dr\big]+ \frac{1}{\epsilon}\E \big[ \sup_{t\leq s\leq T} |{}^{n}\!Y_s^{i,t,x}|^2 \big] + \epsilon \E \big[ \big({}^{n}\!K_T^{i,t,x}-{}^{n}\!K_s^{i,t,x}\big)^2\big],
\end{aligned}
\end{equation*}
where $a^{i,l,n}\in \R$ and $b^{i,n}\in \R^d$  are $\cal P$-measurable bounded processes. Using again the inequality $2ab\leq \nu a^2 + \frac{1}{\nu}b^2$, for $\nu > 0$, we get: 
\begin{equation*}
\begin{aligned}
 &\E \big[\int_s^T|{}^{n}\!Z_r^{i,t,x}|^2dr \big]   + \E \big[\int_s^T\int_E|{}^{n}\!V_r^{i,t,x}(e)|^2 \lambda_n(de)dr \big]\\[3pt]
 &\qquad \leq \E \big[|h_i({}^{n}\!X_T^{t,x})|^2 \big]  + \frac{1}{\nu}\E \big[ \int_s^T |{}^{n}\!Y_r^{i,t,x}|^2dr \big]+ \nu  \E \big[ \int_s^T\big\{ |\widehat{f}_i(r,{}^{n}\!X_r^{t,x},0,0,0)|\\[3pt]
 &\qquad \quad +\sum_{l=1,m}a_r^{i,l,n}|{}^{n}\!Y_r^{l,t,x}| +  b_r^{i,n}|{}^{n}\!Z_r^{i,t,x}| \big\}^2dr\big]\\[3pt]
 &\qquad \quad+ \frac{1}{\epsilon}\E \big[ \sup_{t\leq s\leq T} |{}^{n}\!Y_s^{i,t,x}|^2 \big] + \epsilon \E \big[ \big({}^{n}\!K_T^{i,t,x} - {}^{n}\!K_s^{i,t,x}\big)^2\big].
\end{aligned}
\end{equation*}
From the boundedness of $\widehat{f}_i(t,x,0,0)$ and $h_i(x)$, the inequality $|a + b + c|^2 \leq 3 \{|a|^2 + |b|^2 + |c|^2\}$ $\forall a, b, c\in \R$ and finally the Cauchy-Schwarz one, we obtain:
\begin{equation*}
\begin{aligned}
 &\E \big[\int_s^T|{}^{n}\!Z_r^{i,n}|^2dr \big]  + \E \big[\int_s^T\int_E|{}^{n}\!V_r^{i,t,x}(e)|^2 \lambda_n(de)dr \big]\\[4pt]
 &\qquad \leq  \bar C^2+3\nu \bar C^2 (T-s)  + \frac{1}{\nu}\E \big[ \int_s^T |{}^{n}\!Y_r^{i,t,x}|^2dr \big]  +  3\nu C\E\big[\int_s^T\sum_{l=1,m}|{}^{n}\!Y_r^{l,t,x}|^2 dr \big]\\[4pt]
 & \qquad \quad+3\nu C\E\big[\int_s^T |{}^{n}\!Z_r^{i,t,x}|^2dr \big] + \frac{1}{\epsilon}\E \big[ \sup_{t\leq s\leq T} |{}^{n}\!Y_s^{i,t,x}|^2 \big] + \epsilon \E \big[ \big({}^{n}\!K_T^{i,t,x} -{}^{n}\!K_s^{i,t,x}\big)^2\big], 
\end{aligned}
\end{equation*}
  for a suitable positive constant $C$. Choose now $\nu$ such that $3\nu C < 1$ and taking the summation over all $i \in \mi$, we obtain:
 \begin{equation*}
\begin{aligned}
 &\sum_{i=1,m}\Big(\E \big[\int_s^T|{}^{n}\!Z_r^{i,t,x}|^2dr \big] +\E \big[\int_s^T\int_E|{}^{n}\!V_r^{i,t,x}(e)|^2 \lambda_n(de)dr \big]\Big)\\[4pt]
 & \qquad \leq C\Big( 1  + \sum_{i=1,m}\E \big[\sup_{t\leq s \leq T}|{}^{n}\!Y_s^{i,t,x}|^2 \big] \Big)+\epsilon \sum_{i=1,m}\E \big[ \big({}^{n}\!K_T^{i,t,x}-{}^{n}\!K_s^{i,t,x}\big)^2\big], 
\end{aligned}
\end{equation*} 
where $C = C(T,m,\nu,\epsilon) >0$ is an appropriate constant independent of $n$. But 
$Y^{i,t,x}$ is uniformly bounded, then get: 
\begin{equation}\label{estimZ1}
\begin{aligned}
 &\sum_{i=1,m}\Big(\E \big[\int_s^T|{}^{n}\!Z_r^{i,t,x}|^2dr \big]  + \E \big[\int_s^T\int_E|{}^{n}\!V_r^{i,t,x}(e)|^2 \lambda_n(de)dr \big]\Big)\nonumber\\
 & \qquad \leq C+ \epsilon \sum_{i=1,m}\E \big[ \big({}^{n}\!K_T^{i,n}-{}^{n}\!K_s^{i,t,x}\big)^2\big].
\end{aligned}
\end{equation} 
Now, from the equality:
\begin{equation}
\begin{aligned}
{}^{n}\!K_T^{i,t,x}-{}^{n}\!K_s^{i,t,x} = &\, {}^{n}\!Y_s^{i,t,x} - h_i({}^{n}\!X_T^{t,x}) - \int_s^T \widehat{f}_i(r,{}^{n}\!X_r^{t,x},({}^{n}\!Y_r^{k,t,x})_{k\in \mi}, {}^{n}\!Z_r^{i,t,x})dr \\[3pt]
& \qquad + \int_s^T {}^{n}\!Z_r^{i,t,x} dB_r + \int_s^T \int_E {}^{n}\!V_r^{i,t,x}(e)\tilde{\mu}_n(dr,de), \,s\le T, 
\end{aligned}
\end{equation}
and, once again, by a linearization procedure of the Lipschitz function $\widehat{f}_i$, the boundedness of $\widehat{f}_i(t,x,\vec 0,0)$, $h_i(x)$ and $u^{i,n}(t,x)$ (and then of $({}^{n}\!Y_r^{i,t,x})_{r\in [t,T]}$ by \eqref{repynuin}) and finally the use of the Burkholder-Davis-Gundy inequality, there exists a positive constant $C^{\prime}$ such that: $\forall s\in [t,T]$, 
 \begin{equation*}
\begin{aligned}
 &\sum_{i=1,m}\E \big[\big({}^{n}\!K_T^{i,t,x}-{}^{n}\!K_s^{i,t,x}\big)^2 \big]\\[3pt]
 &  \quad \leq C^{\prime} \Big( 1 + \sum_{i=1,m}\E \big[\sup_{t\leq s \leq T}|{}^{n}\!Y_s^{i,t,x}|^2 \big]+  \sum_{i=1,m}\E \big[\int_s^T|{}^{n}\!Z_r^{i,t,x}|^2dr \big]\\[4pt]
 &\qquad \qquad\qquad + \sum_{i=1,m}\E \big[\int_s^T\int_E |{}^{n}\!V_r^{i,t,x}(e)|^2\lambda_n(de)dr\big]\Big)\\[4pt]
 &\quad \leq C^{\prime} \Big( 1+\sum_{i=1,m}\E \big[\int_s^T|{}^{n}\!Z_r^{i,t,x}|^2dr \big]+ \sum_{i=1,m}\E \big[\int_s^T\int_E |{}^{n}\!V_r^{i,t,x}(e)|^2\lambda_n(de)dr\big]\Big).
\end{aligned}
\end{equation*} 
Combining this last estimate with \eqref{estimZ1} and choosing $\epsilon$ small enough,  since it is arbitrary, there exists a constant $\bar{C}$ independent of $n$ such that: $\forall s\in [t,T]$,  
 \begin{equation*}
\begin{aligned}
 \sum_{i=1,m}\E \big[\int_s^T|{}^{n}\!Z_r^{i,t,x}|^2dr + \big({}^{n}\!K_T^{i,t,x}-{}^{n}\!K_s^{i,t,x}\big)^2  +\int_s^T\int_E |{}^{n}\!V_r^{i,t,x}(e)|^2\lambda_n(de)dr\big] \leq \bar{C}
\end{aligned}
\end{equation*}
which implies \eqref{estimZK1} in taking $s=t$. 
\ms

\nd \underline{Step 3}: For any $\ii$, the sequences 
of processes $({}^{n}\!Z^{i,t,x})_{n\ge 0}$ and $({}^{n}\!V^{i,t,x}(e)1_{\{|e|\ge \frac{1}{n}\}})_{n\ge 1}$ are Cauchy in $\hdd$ and $\hld$ respectively. 
\medskip

Actually for any $n, m \geq 1$, by It\^o's formula we have: $\fstt$, 
\begin{equation*}
\begin{aligned}
 &\E \big[\int_s^T|{}^{n}\!Z_r^{i,t,x} -{}^{m}\!Z_r^{i,t,x} |^2dr \big]\nonumber\\[3pt]
 &\qquad + \E \big[\int_s^T \int_E|{}^{n}\!V_r^{i,t,x}(e)\mathbf{1}_{\lbrace|e| \geq \frac{1}{n}\rbrace} -{}^{m}\!V_r^{i,t,x}(e)\mathbf{1}_{\lbrace|e| \geq \frac{1}{m}\rbrace} |^2 \lambda(de)dr \big]\\[3pt]
 &\quad \leq  \E\Big[  |h_i({}^{n}\!X_T^{t,x})- h_i({}^{m}\!X_T^{t,x})|^2\Big] + 2\E \big[ \int_s^T \big({}^{n}\!Y_r^{i,t,x} -{}^{m}\!Y_r^{i,t,x}\big)\times\nonumber\\[3pt]
 & \qquad \Big\{\widehat{f}_i(r,{}^{n}\!X_r^{t,x},({}^{n}\!Y_r^{k,t,x})_{k\in \mi}, {}^{n}\!Z_r^{i,t,x})- \widehat{f}_i(r,{}^{m}\!X_r^{t,x},({}^{m}\!Y_r^{k,t,x})_{k\in \mi}, {}^{m}\!Z_r^{i,t,x})\Big\}dr\big]\nonumber\\[3pt]
&\qquad+ 2 \E \big[ \int_s^T \big({}^{n}\!Y_r^{i,t,x} -{}^{m}\!Y_r^{i,t,x}\big)  \big(d{}^{n}\!K_r^{i,t,x} -d{}^{m}\!K_r^{i,t,x}\big) \big].
\end{aligned}
\end{equation*}
The Cauchy-Schwarz inequality and the inequality $2ab\leq \frac{1}{\eta} a^2 + \eta b^2$ for $\eta > 0$, yield:
\begin{align}\label{convZ}
 &\E \big[\int_s^T|{}^{n}\!Z_r^{i,t,x} -{}^{m}\!Z_r^{i,t,x}|^2dr \big]\nonumber\\[3pt]
 &\qquad + \E \big[\int_s^T \int_E|{}^{n}\!V_r^{i,t,x}(e)\mathbf{1}_{\lbrace|e| \geq \frac{1}{n}\rbrace} -{}^{m}\!V_r^{i,t,x}(e)\mathbf{1}_{\lbrace|e| \geq \frac{1}{m}\rbrace} |^2 \lambda(de)dr \big]\nonumber\\[3pt]
&\quad \leq \E\Big[  |h_i({}^{n}\!X_T^{t,x})- h_i({}^{m}\!X_T^{t,x})|^2\Big]+2 \sqrt{\E\Big[\displaystyle \sup_{t\leq s\leq T}|{}^{n}\! Y_s^{i,t,x}- {}^{m}\!Y_s^{i,t,x}|^2\Big]}\times\nonumber\\[3pt]
    &\qquad \Big \{\E\Big[\int_s^T|\widehat{f}_i(r,{}^{n}\!X_r^{t,x},({}^{n}\!Y_r^{k,t,x})_{k\in \mi},{}^{n}\!Z_r^{i,t,x}) - \widehat{f}_i(r,{}^{m}\!X_r^{t,x},({}^{m}\!Y_r^{k,t,x})_{k\in \mi},{}^{m}\!Z_r^{i,t,x})|^2 dr\Big]\Big \}^{\frac{1}{2}}\nonumber\\[3pt]
    & \qquad + \frac{1}{\eta}\E\Big[\displaystyle \sup_{t\leq s\leq T}|{}^{n}\! Y_s^{i,t,x}- {}^{m}\!Y_s^{i,t,x}|^2\Big] + \eta \E \Big[ ({}^{n}\!K_T^{i,t,x}+{}^{m}\!K_T^{i,t,x})^2\Big].
\end{align}
But,  {\bc  as ${}^{n}\!Y^{k,t,x}$ and ${}^{m}\!Y^{k,t,x}$ are bounded, $\hat f_i(r,x,0,0)$ is bounded,  $\hat f_i(r,x,y,z)$ is Lipschitz w.r.t $(y,z)$ uniformly in $(t,x)$ and using inequality \eqref{estimZK1}} and a linearization procedure of $\hat f_i(.)$ one deduces the existence of a constant $C \geq 0$ (independent of $n$),  such that, for all $n\geq 1$,
\begin{equation}
\begin{aligned}
 \E \big[ \int_s^T \big|\hat f_i(r,{}^{n}\!X_r^{t,x},({}^{n}\!Y_r^{k,t,x})_{k\in \mi}, {}^{n}\!Z_r^{i,t,x})\big|^2dr\big] \leq C.
\end{aligned}
\end{equation} 
Besides, as $h_i$ is continuous and bounded then
\begin{align*}
 \E\big[|h_i({}^{n}\!X_T^{t,x})-h_i({}^{m}\!X_T^{t,x})|\big]\rw 0 \mbox{ as $n,m\rw +\infty$}.
\end{align*}
Then taking the limit w.r.t $n,m$ in \eqref{convZ} and taking into account of \eqref{estimZK1} and the convergence of  ${}^{n}\!Y^{i,t,x}$ in $\mathcal{S}^2_{[t,T]}$, we deduce that:
\begin{align}\label{3.38}
 \limsup_{n,m\rightarrow \infty}\Big\{&\E \big[\int_t^T|{}^{n}\!Z_r^{i,t,x} -{}^{m}\!Z_r^{i,t,x} |^2dr \big]\nonumber\\[3pt]
 &+ \E \big[\int_t^T \int_E|{}^{n}\!V_r^{i,t,x}(e)\mathbf{1}_{\lbrace|e| \geq \frac{1}{n}\rbrace} -{}^{m}\!V_r^{i,t,x}(e)\mathbf{1}_{\lbrace|e| \geq \frac{1}{m}\rbrace} |^2 \lambda(de)dr \big]\Big\}\leq \eta\, C.
\end{align}
 As $\eta$ is arbitrary then $({}^{n}\!Z^{i,t,x})_{n\ge 1}$ and  $({}^{n}\!V^{i,t,x}\mathbf{1}_{\lbrace|e| \geq \frac{1}{n}\rbrace})_{n\ge 1}$  are Cauchy sequences in the complete spaces $\hdd$ and $\hld$ respectively. Then there exist processes $\check Z^{i,t,x}$ and $\check V^{i,t,x}$, respectively $\mathcal{P}$-measurable and $\mathbf{P}$-measurable such that the sequences   $({}^{n}\!Z^{i,t,x})_{n\ge 1}$ and $({}^{n}\!V^{i,t,x}\mathbf{1}_{\lbrace|e| \geq \frac{1}{n}\rbrace})_{n\ge 1}$  converge respectively toward  $\check Z^{i,t,x}$ and $\check V^{i,t,x}$ in $\hdd$ and $\hld$.
Now, going back to \eqref{envsnell} and taking the limit w.r.t. $n$, to obtain: for any $\ii$ and  $s \in [t,T]$,
\begin{align*}
    \check Y_s^{i,t,x}= \displaystyle\mbox{esssup}_{\tau \geq s}\E\Big[\int_s^{\tau}&\widehat{f}_i(r,X_r^{t,x},(\check Y_r^{k,t,x})_{k\in \mi},\check Z_r^{i,t,x})dr +  h_i(X_{T}^{t,x})\mathbf{1}_{\lbrace \tau=T\rbrace}\nonumber\\
     &+\max_{j \in \mathcal{I}^{-i}}(Y_{\tau}^{j,t,x} -g_{ij}(\tau))\mathbf{1}_{\lbrace \tau<T\rbrace}\big|\mathcal{F}_s\Big].
\end{align*}
Therefore by Hamadene-Ouknine's result \cite{ham-ouk}, there exist processes $(\tilde Z^{i,t,x}, \tilde V^{i,t,x},\tilde K^{i,t,x})_{\ii}\in \hdd\times \hld \times \mathcal{A}^2$ such that: for any $\ii$ and $\stt$,
\begin{align}
    \begin{cases}
    \check Y_s^{i,t,x} = h_i(X_T^{t,x})+ \int_s^T  \widehat{f}_i(r,X_r^{t,x},(\check Y_r^{k,t,x})_{k\in \mi},\check Z_r^{i,t,x})dr  +\tilde K_T^{i,t,x} - \tilde K_s^{i,t,x} \\[5pt]
 \qquad \qquad   -\int_s^T\tilde  Z_r^{i,t,x}dB_r -\int_s^T \int_E \tilde V_r^{i,t,x}(e) \tilde{\mu}(dr,de);\\[5pt]
 \check Y_s^{i,t,x} \geqslant  \displaystyle \max_{j \in \mathcal{I}^{-i}}(\check  Y_s^{j,t,x} -g_{ij}(s));\\[5pt]
 \int_t^T (\check Y_s^{i,t,x} -\displaystyle \max_{j\in \mathcal{I}^{-i}}(\check Y_s^{j,t,x} -g_{ij}(s))) d\tilde K_s^{i,t,x} = 0.
\end{cases}
\end{align}\label{3.39}
Now, applying It\^o's formula with $|{}^{n}\!Y_s^{i,t,x}- \check{Y}_s^{i,t,x}|^2$, yields: $\forall s\in [t,T]$,
\begin{align*}
   &\E\big[|{}^{n}\!Y_s^{i,t,x}- \check{Y}_s^{i,t,x}|^2 + \textstyle \int_s^T |{}^{n}\!Z_r^{i,t,x}- \tilde{Z}_r^{i,t,x}|^2dr \\[7pt]
   &\quad  +\textstyle \int_s^T \int_E |{}^{n}\!V_r^{i,t,x}\mathbf{1}_{\lbrace|e| \geq \frac{1}{n}\rbrace}- \tilde{V}_r^{i,t,x}|^2 \lambda(de)dr\big]\\[7pt]
   & \leq  \E\big[  |h_i({}^{n}\!X_T^{t,x})- h_i(X_T^{t,x})|^2\big] +2 \E\big[\textstyle \int_s^T ({}^{n}\! Y_r^{i,t,x}- \check{Y}_r^{i,t,x})\times\\[7pt]
   & \qquad \quad  \{\widehat{f}_i(r,{}^{n}\!X_r^{t,x},({}^{n}\!Y_r^{k,t,x})_{k\in \mi},{}^{n}\!Z_r^{i,t,x}) -\widehat{f}_i(r,X_r^{t,x},(\check{Y}_r^{k,t,x})_{k\in \mi},\check Z_r^{i,t,x})\}dr\big]\\[7pt]
   & \quad+ 2\E\big[ \textstyle \int_s^T ({}^{n}\! Y_r^{i,t,x}- \check Y_r^{i,t,x})d({}^{n}\! K_r^{i,t,x}- \tilde K_r^{i,t,x})\big].
\end{align*}
Next, the same procedure as the one which leads to  \eqref{3.38} can be used here to deduce that: \begin{align*}
    \E \big[ \textstyle \int_t^T |{}^{n}\!Z_r^{i,t,x}- \tilde{Z}_r^{i,t,x}|^2dr + \int_t^T \int_E |{}^{n}\!V_r^{i,t,x}\mathbf{1}_{\lbrace|e| \geq \frac{1}{n}\rbrace}- \tilde{V}_r^{i,t,x}|^2 \lambda(de)dr\big] \rightarrow 0 \mbox{ as } n\rightarrow \infty.
\end{align*}
This implies that the sequences $({}^{n}\!Z^{i,t,x})_{n\ge 1}$ and $({}^{n}\!V^{i,t,x}\mathbf{1}_{\lbrace|e| \geq \frac{1}{n}\rbrace})_{n\ge 1}$  converge toward  $\tilde Z^{i,t,x}$ and $\tilde V^{i,t,x}$ in $\hdd$ and $\hld$ respectively. Therefore, by uniqueness of limits  we deduce that for any $\ii$,
$$\tilde Z^{i,t,x}= \check Z^{i,t,x} \mbox{ and } \tilde V^{i,t,x}= \check V^{i,t,x}$$ in their respective spaces. 
Then $(\check Y^{i,t,x},\check Z^{i,t,x},\check V^{i,t,x},\tilde K^{i,t,x})_{\ii}$ verify: for any $s\in [t,T]$ and $\ii$,
\begin{align*}
    \begin{cases}
    \check Y_s^{i,t,x} = h_i(X_T^{t,x})+ \int_s^T  \widehat{f}_i(r,X_r^{t,x},(\check Y_r^{k,t,x})_{k\in \mi},\check Z_r^{i,t,x})dr  +\tilde K_T^{i,t,x} - \tilde K_s^{i,t,x} \\[5pt]
 \qquad \qquad   -\int_s^T\check  Z_r^{i,t,x}dB_r -\int_s^T \int_E \check V_r^{i,t,x}(e) \tilde{\mu}(dr,de);\\[5pt]
 \check Y_s^{i,t,x} \geqslant  \displaystyle \max_{j \in \mathcal{I}^{-i}}(\check  Y_s^{j,t,x} -g_{ij}(s));\\[5pt]
 \int_t^T (\check Y_s^{i,t,x} -\displaystyle \max_{j\in \mathcal{I}^{-i}}(\check Y_s^{j,t,x} -g_{ij}(s))) d\tilde K_s^{i,t,x} = 0.
\end{cases}
\end{align*}
It means that the quadruples of processes 
$((\check Y^{i,t,x},\check  Z^{i,t,x}, \check V^{i,t,x},\tilde K^{i,t,x}))_{\ii}$ is a solution of system \eqref{sys2.14} in $[t,T]$. But the solution of this latter is unique, therefore 
we have: 
$$(\check Y^{i,t,x}, \check Z^{i,t,x}, \check V^{i,t,x},\tilde K^{i,t,x})_{\ii}=( Y^{i,t,x},  Z^{i,t,x}, V^{i,t,x}, K^{i,t,x})_{\ii}.
$$
Now, let us consider a subsequence which be denoted by $\lbrace n_k \rbrace$ such that  
$$|{}^{n_{k}}\! V_s^{i,t,x}\mathbf{1}_{\lbrace|e| \geq \frac{1}{n_k}\rbrace}- V_s^{i,t,x}|^2 \longrightarrow_{n_k} 0,\, ds\otimes d\mathbb{P} \otimes d\lambda -a.e. \mbox{ on }[t,T]\times E\times \Omega .$$
Besides, the process ${}^{n}\! V^{i,t,x}$ has a representation in terms of $u^{i,n}$ (see  \eqref{repvi}) and 
$(u^{i,n})_{n\ge 0}$ converges uniformly to $u^i$. Then,  if $(x_n)_n$ is a sequence of $\R^k$ which converges to $x$, then $u^{i,n}(t,x_n)$ converges to $u^i(t,x)$. Next, let us consider a subsequence which we  denote by $\lbrace n_{k,l} \rbrace$ such that
\begin{equation*}
    |{}^{n_{k,l}}\!X_{s^{}-}^{t,x}- X_{s^{}-}^{t,x}|^2 \leq  \sup_{s\leq T}|{}^{n_{k,l}}\!X_s^{t,x}- X_s^{t,x}|^2 \longrightarrow_{n_{k,l}} 0,\, \, \mathbb{P}-a.s.
\end{equation*}
 As the mapping $x\mapsto \beta(x,e)$ is Lipschitz, then the sequence
\begin{equation*}
    \begin{aligned}
    &({}^{n_{k,l}}\!V_s^{i,t,x}(e)\mathbf{1}_{\lbrace|e| \geq \frac{1}{n_{k,l}}\rbrace})_{n_{k,l} \geq 1}\\
    \quad & = ((u^{i,n_{k,l}}(s,{}^{n_{k,l}}\!\xtx_{s-}+ \beta({}^{n_{k,l}}\!\xtx_{s-},e))- u^{i,n_{k,l}}(s,{}^{n_{k,l}}\!\xtx_{s-}))\mathbf{1}_{\lbrace|e| \geq \frac{1}{n_{k,l}}\rbrace})_{n_{k,l}\geq 1}\,\, \longrightarrow_{n_{k,l}\rightarrow \infty}\\\nonumber 
    & \quad u^{i}(s,\xtx_{s-}+ \beta(\xtx_{s-},e))- u^{i}(s,\xtx_{s-}),\,\, ds\otimes d\mathbb{P} \otimes d\lambda -a.e. \mbox{ on } [t,T] \times \Omega \times E.
    \end{aligned}
\end{equation*}
Therefore, we have that:
\begin{equation}
    V_s^{i,t,x}(e)=u^{i}(s,\xtx_{s-}+ \beta(\xtx_{s-},e))- u^{i}(s,\xtx_{s-}),\,\, ds\otimes d\mathbb{P} \otimes d\lambda -a.e.\mbox{ on } [t,T] \times \Omega \times E
\end{equation}
which is the desired result. $\Box$
\begin{remarque}
For any $\ii$, $\check{u}^i= u^i$, where $u^i$ is the Feynman-Kac representation of $Y^{i,t,x}$ given in \eqref{fkrep1} $\Box$
\end{remarque}

We are now in position to show existence of a solution for system \eqref{sys2.13}  in the case when $\lambda(.)$ is not finite and integrates $(1\wedge |e|)_{e\in E}$. 

\begin{theoreme} \label{existence2}
Assume that the assumptions (H1)-(H4) are satisfied. Then the system \eqref{sys2.13} has a solution $(Y^{i,t,x},Z^{i,t,x},V^{i,t,x},K^{i,t,x})_{i\in \mi}$. Moreover:

\nd a) There exist  bounded continuous functions $(u^i)_{\ii}$ such that for any $\ii$, $\tx$,
\begin{align} \label{repY13}
Y^{i,t,x}_s=u^i(s,\xtx_s), \,\,\forall s\in [t,T].
\end{align}
b) For any $\ii$ and $(t,x)\in [0,T]\times \R^k$,
\begin{align} \label{repV1}
V_s^{i,t,x}(e) =  1_{\{s\ge t\}} (u^{i}(s,\xtx_{s-}+& \beta(\xtx_{s-},e))- u^{i}(s,\xtx_{s-})),\nonumber\\
& \qquad ds\otimes d\mathbb{P} \otimes d\lambda \mbox{ on } [0,T] \times \Omega \times E.
\end{align}
\end{theoreme}
\proof The proof follows the same steps as in the proof of Proposition 3.2 in \cite{hamadene_mnif-neffati2}, except that in our framework we should take into account of the fact that $\lee=\infty$. This difficulty is tackled by using the inequalities $|\gamma_i(x,e)|\leq \cig(1\wedge|e|)$, $\ii$, and since $\int_E (1\wedge|e|)\lambda(de) <\infty$.
\ms

\noindent \underline{Step 1}: The iterative construction

Let us consider  $((Y^{i,n;t,x},Z^{i,n;t,x},V^{i,n;t,x},K^{i,n;t,x})_{i\in \mi})_{n\geq 0}$   the sequence of processes defined recursively as follows:
 \begin{equation*}
   \vspace{0.1cm} (Y^{i,0;t,x},Z^{i,0;t,x},V^{i,0;t,x},K^{i,0;t,x}) = (0,0,0,0)\, \, \mbox{ for all } \ii \, \,  \mbox{ and for } n\ge 1 \mbox{ and } s \leq T,
  \end{equation*}
\begin{equation}\label{recurrencepp3}
\begin{split}
\begin{cases}
\vspace{0.2cm} Y^{i,n;t,x}\in \ss,  Z^{i,n;t,x} \in \hdd, V^{i,n;t,x} \in \hld \mbox{ and } K^{i,n;t,x}\in \aa;\\ \vspace{0.2cm}
Y^{i,n;t,x}_s = h_i(X_T^{t,x}) +\int_s^T \bar f_{i}(r,\xtx_r,(Y^{k,n;t,x}_r)_{k\in \mathcal{I}},Z^{i,n;t,x}_r, \int_E V^{i,n-1;t,x}_r(e)\times \\ \vspace{0.4cm}
 \qquad \qquad  \gamma_i(X_r^{t,x},e)\lambda(de))dr +K_T^{i,n;t,x}- K_s^{i,n;t,x}-\int_s^T Z_r^{i,n;t,x}dB_r-\int_s^T \int_E V_r^{i,n;t,x}(e) \tilde{\mu}(dr,de); \\\vspace{0.4cm}
Y_s^{i,n;t,x} \geqslant  \displaystyle \max_{j \in \mathcal{I}^{-i}}(Y_s^{j,n;t,x} -g_{ij}(s)); \\\vspace{0.2cm}
\textstyle \int_0^T (Y_s^{i,n;t,x} -\displaystyle \max_{j\in \mathcal{I}^{-i}}(Y_s^{j,n;t,x} -g_{ij}(s))) dK_s^{i,n;t,x} = 0.
\end{cases}
 \end{split}
\end{equation}
Thanks to the result in \cite{hamadene2015viscosity}, the system admits a unique solution. Indeed, the generators  $\check {F}_{i}(\dots)$, $\ii$, of the system do not depend on $V^{i,n;t,x}$, noting that $V^{i,n-1;t,x}$ is already determined.
Next, by an induction argument on $n$,  we are going to show that there exist deterministic continuous bounded functions  $(u^{i,n})_{\ii}$ such that for any $\tx$ and any $s \in [t,T]$:
\begin{equation}\label{Y,V}
\begin{split}
\begin{cases}
i) \,Y_s^{i,n;t,x} = u^{i,n}(s,\xtx_s);\\ 
\\
ii) \,\mbox {There exists a constant C such that for any } n\ge 0 \mbox{ and } \ii, \mbox {we have:} \\
\qq\qq\qq\qq\qq|u^{i,n}(t,x)|\le C, \,\,\forall \tx;\\\\
iii) \,V_s^{i,n;t,x}(e) = u^{i,n}(s,\xtx_{s-}+ \beta(\xtx_{s-},e))- u^{i,n}(s,\xtx_{s-}),\,\, ds\otimes d\mathbb{P} \otimes d\lambda\,\, a.e\mbox{ on } [t,T] \times \Omega \times E.
\end{cases}
\end{split}
\end{equation} 

Indeed,  for $n=0$,  the property holds true with $u^{i,0}=0$, $\ii$. Suppose now that it is satisfied for some $n$. Then $(Y^{i,n+1;t,x},Z^{i,n+1;t,x},V^{i,n+1;t,x},K^{i,n+1;t,x})$ verifies (we omit the dependence on $t,x$ as there is no confusion): $ \forall s \in [t,T]$ and $\ii$,
\begin{equation}\label{eq343}
\begin{split}
\begin{cases}
\vspace{0.2cm} Y_s^{i,n+1} = h_i(X_T^{t,x}) +\int_s^T \bar f_i(r,X_r^{t,x},(Y_r^{k,n+1})_{k \in \mi},Z_r^{i,n+1},\int_E \{u^{i, n}(r,\xtx_{r-}+ \\\vspace{0.3cm}
 \qquad \quad \beta(\xtx_{r-},e))- u^{i,n}(r,\xtx_{r-})\}\gamma_i(\xtx_r,e)\lambda(de))dr+K_T^{i,n+1}- K_s^{i,n+1}\\\vspace{0.3cm}
 \qquad \quad - \int_s^T Z_r^{i,n+1}dB_r-\int_s^T \int_E V_r^{i,n+1}(e) \tilde{\mu}(dr,de);\\\vspace{0.3cm}
Y_s^{i,n+1} \geqslant  \displaystyle \max_{j \in \mathcal{I}^{-i}}(Y_s^{j,n+1} -g_{ij}(s)); \\\vspace{0.2cm}   
 \int_t^T (Y_s^{i,n+1} -\displaystyle \max_{j\in \mathcal{I}^{-i}}(Y_s^{j,n+1} -g_{ij}(s))) dK_s^{i,n+1} = 0.
\end{cases}
\end{split}
\end{equation} 
The generators of the system of reflected BSDEs \eqref{eq343} are given by: 
$$\check F_i^{n+1}(t,x,\vec y,z)=\bar f_i(t,x,\vec y,z,\txst \int_E \{u^{i, n}(t,x+ \beta(x,e))- u^{i,n}(t,x)\}\gamma_i(x,e)\lambda(de)), \ii \mbox{ and }n\ge 1,$$
since the process $\xtx$ is RCLL and the functions $u^{i,n}, \beta^i$ and $\g_i$ are continuous w.r.t. $x$. Next:

a) The function 
$(t,x)\mapsto \int_E \{u^{i, n}(t,x+ \beta(x,e))- u^{i,n}(t,x)\}\gamma_i(x,e)\lambda(de))$ is continuous and bounded by the induction hypothesis and by applying the Lebesgue dominated convergence theorem.

b) For any $\ii$, $(t,x)\mapsto \check F_i^{n+1}(t,x,\vec y,z)$ is continuous uniformly w.r.t $(\vec y,z)$ since $\bar f_i(t,x,\vec y,z,q)$ verifies (H1)-i) and is Lipschitz w.r.t $(\vec y,z,q)$ uniformly in $(t,x)$. 
\medskip 

\nd Therefore, by Proposition 4.2 in \cite{hamadene2015viscosity}, there exist deterministic continuous functions of polynomial growth $(u^{i,n+1})_{\ii}$ such that for any $\ii$ and $\stt$
$$Y_s^{i,n+1;t,x} = u^{i,n+1}(s,\xtx_s).$$

c) Let $(\bar Y, \bar Z)$ be the solution of the following standard BSDE: 
\begin{equation*}
    \begin{cases}
    \bar Y \in \ss,\, \bar Z \in \hdd;\\
    \bar Y_s = \bar C + \int_s^T \big\lbrace \bar C+ m\,C^y_f\,\bar Y_r + C^z_f\,|\bar Z_r| +2\theta \bar Y_r\big\rbrace dr - \int_s^T \bar Z_rdB_r,\,s\le T,;
    \end{cases}
\end{equation*}
where: (i) $C_f^y$ , $C_f^z$ and $C_f^q$ are the maximum (w.r.t $i$) of the Lipschitz constants of the functions $\bar f_i(t,x,\vy,z,q)$ w.r.t. $\vy$, $z$ and $q$ respectively; (ii) $\bar C$ is a constant of boundedness of $h_i$ and $\bar f_i(t,x,\vec 0, 0,0)$; (iii)   the constant $\theta$ is given by:
\begin{equation*}
\begin{array}{c}
    \theta = C^q_f\,(\underbrace{\max_{i=1,...,m}\cig}_{C_\gamma})\int_E(1\wedge|e|)\lambda(de).
    \end{array}
\end{equation*}
Therefore if for any $\ii$ and $\stt$, $|u^{i,n}(s,\xtx_s)|\le \bar Y_s$, then 
for any $\ii$ and $\stt$, $|u^{i,n+1}(s,\xtx_s)|\le \bar Y_s$ (one can see \cite{hamadene_mnif-neffati}, pp.12 for more details). As the process $\bar Y$ is deterministic continuous, then it is bounded. Finally the proof of $ii)$ is proved by an obvious induction since it is satisfied for $n=0$. 

iii) The representation of the processes $V^{i,n+1;t,x}$ stems from Proposition \ref{repvi2} since by induction hypothesis on $n$, the functions 
$(\check F_i^{n+1})_{\ii}$, $(h_i)_{i \in \mathcal{I}}$ and $(g_{ij})_{i,j \in \mathcal{I}}$ fulfill the requirements of Proposition \ref{firstresult}. 

It follows that points $i)$, $ii)$ and $iii)$ above are valid for $n+1$ and then they hold true for any $n\ge 0$. 
 
\begin{remarque}\label{sur0t}
 For $s\in [0,t]$, $\xtx_s=x$, $Y_t^{i,n;t,x}=u^{i,n}(t,x)$, then in considering the declination of system \eqref{recurrencepp3} on the time interval $[0,t]$, we can easily show by induction that 
 $Z^{i,n}_s1_{[s\le t]}=0, \,\,ds\otimes d\mathbb{P}-a.e$ and 
 $V^{i,n}_s(e)1_{[s\le t]}=0,\,\,  ds\otimes d\mathbb{P}\otimes d\lambda-a.e.$ since the data are continuous and deterministic.
 \end{remarque}

 \noindent \sol{Step 2}: Convergence of $(u^{i,n})_{n\ge 1}$
 
 \noindent Following similar arguments as in the proof of Theorem 3.2, pp.6 in \cite{hamadene_mnif-neffati2}, we obtain for some $\alpha_0>0$ and $\kappa$, constants which depend only on the Lipschitz constants of $(\bar f_i)_{\ii}$ w.r.t $\vec y$, $z$ and $q$: $\forall n,p\geq 1, s \in [t,T]$,
\begin{align*}
     &\E \Big[  e^{\alpha_0 s}|(Y_s^{k,n;t,x})_{k\in \mathcal{I}}- (Y_s^{i,p;t,x})_{k\in \mathcal{I}}|^2\Big]\nnb \\[6pt]
     &\quad \leq \kappa \E\Big[\textstyle \int_s^T e^{\alpha_0 r} \big(\int_E \sum_{k=1,m}|\{V_r^{k,n-1;t,x}(e)-V_r^{k,p-1;t,x}(e)\} \gamma_k(\xtx_r,e)| \lambda(de)\big)^2dr\Big].
\end{align*}
Take now $s=t$ and making use of \eqref{Y,V} to obtain:   $\forall \ii$,
\begin{align}\label{eqpp3}
    &|u^{i,n}(t,x)- u^{i,p}(t,x)|^2\nonumber\\[6pt]
    &\quad \leq \kappa  \E\Big[\textstyle \int_t^T e^{\alpha_0 (r-t)} \big(\int_E \sum_{k=1,m} |\{ u^{k, n-1}(r,\xtx_{r-}+ \beta(\xtx_{r-},e))- u^{k,n-1}(r,\xtx_{r-}) \nonumber \\[6pt]
& \qquad - ( u^{k, p-1}(r,\xtx_{r-}+ \beta(\xtx_{r-},e))-u^{k,p-1}(r,\xtx_{r-}))\}\gamma_k(\xtx_r,e) |\lambda(de)\big)^2dr\Big].
\end{align}
Now, let $\eta>0$ and let us set $$\|u^{i,n}- u^{i,p}\|_{\infty,\eta}:=\sup_{(t,x)\in [T-\eta, T]\times \R^k}|u^{i,n}(t,x) - u^{i,p}(t,x)|.
$$
Since $u^{i,n}$ is uniformly bounded, we then have: For any $t \in [T-\eta, T]$ and $\ii$, 
\begin{align*}
    &|u^{i,n}(t,x) - u^{i,p}(t,x)|^2\\[6pt]
    & \quad \leq \kappa \underbrace{(\max_{i=1,m}\cig)^2}_{C^2_\g}\E \Big[\textstyle  \int_{t}^T   e^{\alpha_0 (r-t)} \Big(\int_E \displaystyle \sum_{k=1,m} |\{ u^{k, n-1}(r,\xtx_{r-}+ \beta(\xtx_{r-},e))- u^{k,n-1}(r,\xtx_{r-}) \nonumber \\[6pt]
&  \qquad - ( u^{k, p-1}(r,\xtx_{r-}+ \beta(\xtx_{r-},e))-u^{k,p-1}(r,\xtx_{r-}))\}|(1\wedge |e|)\lambda(de)\Big)^2dr\Big]\\
&\quad \le \textstyle 4\kappa C^2_\g \{\int_E  (1\wedge |e|)\lambda(de)\}^2\times  \textstyle \int_{t}^T dr  e^{\alpha_0 (r-t)} \{\displaystyle \sum_{k=1,m}\|u^{k,n-1}- u^{k,p-1}\|_{\infty,\eta}\}^2\nonumber\\[6pt]
&\quad \le \underbrace{\textstyle 4m\kappa  \{C_\g\int_E  (1\wedge |e|)\lambda(de)\}^2}_{\Xi(\kappa,\g)} \times \displaystyle \sum_{k=1,m}\|u^{k,n-1}- u^{k,p-1}\|_{\infty,\eta}^2\textstyle \times  \underbrace{\int_{t}^T dr  e^{\alpha_0 (r-t)}}_{\le \alpha_0^{-1}(e^{\eta \alpha_0}-1)}.\nonumber
\end{align*}
Now, let $\eta$ be a constant such that $ \frac{\Xi(\kappa,\g)}{\alpha_0} m(e^{\alpha_0 \eta }-1)= \frac{3}{4}$. Note that $\eta$ does not depend on the terminal conditions $(h_i)_{\ii}$ neither on $T$. Therefore we deduce from the last inequality, in taking the supremum over $(t,x)$, that for any $n,q\ge 1$,
\begin{align}
\sum_{i=1,m}\|u^{i,n}- u^{i,p}\|^2_{\infty,\eta} \leq  \frac{3}{4}\sum_{i=1,m}\|u^{i,n-1}- u^{i,p-1}\|^2_{\infty,\eta} \mbox{ and then }\limsup_{n,p\rw \infty}\sum_{i=1,m}\|u^{i,n}- u^{i,p}\|^2_{\infty,\eta}=0.
\lb{uinversz}
\end{align}
It means that the sequence $((u^{i,n})_{\ii})_{n\geq 0}$ is uniformly convergent in $[T-\eta,T]\times \R^k$. So for  $(t,x) \in [T-\eta,T]\times \R^k$, let us set $u^i(t,x) = \lim_{n \rightarrow \infty} u^{i,n}(t,x)$, $\ii$. Note that $(u^i)_{\ii}$ are continuous bounded functions on $[T-\eta,T]\times \R^k$.

Next let us set 
$$\|u^{i,n}- u^{i,p}\|_{\infty,2\eta}:=\sup_{(t,x)\in [T-2\eta, T-\eta]\times \R^k}|u^{i,n}(t,x) - u^{i,p}(t,x)|
$$
and let $t\in [T-2\eta, T-\eta]$. From \eqref{eqpp3} we have:
\begin{align*}
    |u^{i,n}(t,x)&- u^{i,p}(t,x)|^2 \\&\leq \kappa  \E\Big[\textstyle \int_t^{T-\eta} e^{\alpha_0 (r-t)} \big(\int_E \sum_{k=1,m} |\{ u^{k, n-1}(r,\xtx_{r-}+ \beta(\xtx_{r-},e)) - u^{k,n-1}(r,\xtx_{r-}) \nonumber \\[6pt]
&\qq \qq\qq - ( u^{k, p-1}(r,\xtx_{r^{-}}+ \beta(\xtx_{r-},e)) -u^{k,p-1}(r,\xtx_{r-}))\}\gamma_k(\xtx_r,e) |\lambda(de)\big)^2dr\Big]\\&
\qq +\kappa  \E\Big[\textstyle \int_{T-\eta}^T e^{\alpha_0 (r-t)} \big(\int_E \sum_{k=1,m} |\{ u^{k, n-1}(r,\xtx_{r-}+ \beta(\xtx_{r-},e))- u^{k,n-1}(r,\xtx_{r-})\\&\qq \qq\qq - ( u^{k, p-1}(r,\xtx_{r-}+ \beta(\xtx_{r-},e)) -u^{k,p-1}(r,\xtx_{r-}))\}\gamma_k(\xtx_r,e) |\lambda(de)\big)^2dr\Big]\\
&\leq\kappa  \E\Big[\textstyle \int_{t}^{T-\eta} e^{\alpha_0 (r-t)} \big(\int_E \sum_{k=1,m} |\{ u^{k, n-1}(r,\xtx_{r-}+ \beta(\xtx_{r-},e)) \nonumber - u^{k,n-1}(r,\xtx_{r-})\\&\qq \qq\qq- ( u^{k, p-1}(r,\xtx_{r-}+ \beta(\xtx_{r-},e))-u^{k,p-1}(r,\xtx_{r-}))\}\gamma_k(\xtx_r,e) |\lambda(de)\big)^2dr\Big]\\&
\qq +\textstyle\frac{3}{4m}e^{\alpha_0 \eta}\sum_{k=1,m}\|u^{k,n-1}- u^{k,p-1}\|_{\infty,\eta}^2\\&
\leq \Xi(\kappa,\g)\sum_{k=1,m}\|u^{k,n-1}- u^{k,p-1}\|_{\infty,2\eta}^2\textstyle \times  \int_{t}^{T-\eta} dr  e^{\alpha_0 (r-t)}+\textstyle\frac{3}{4m}e^{\alpha_0 \eta}\sum_{k=1,m}\|u^{k,n-1}- u^{k,p-1}\|_{\infty,\eta}^2\\&
\leq\textstyle \frac{3}{4m}\sum_{k=1,m}\|u^{k,n-1}- u^{k,p-1}\|_{\infty,2\eta}^2+\frac{3}{4m}e^{\alpha_0 \eta}\sum_{k=1,m}\|u^{k,n-1}- u^{k,p-1}\|_{\infty,\eta}^2.
\end{align*}It implies that 
\begin{align*}
  \textstyle \sum_{k=1,m}\|u^{k,n}- u^{k,p}\|_{\infty,2\eta}^2\le \frac{3}{4}\sum_{k=1,m}\|u^{k,n-1}- u^{k,p-1}\|_{\infty,2\eta}^2+\frac{3}{4}e^{\alpha_0 \eta}\sum_{k=1,m}\|u^{k,n-1}- u^{k,p-1}\|_{\infty,\eta}^2.
\end{align*}
Using now \eqref{uinversz} to deduce that 
$$
\limsup_{n,p\rw \infty}\sum_{k=1,m}\|u^{k,n}- u^{k,p}\|_{\infty,2\eta}^2=0.$$
Then once more the sequence $((u^{i,n})_{\ii})_{n\geq 0}$ is uniformly convergent in $[T-2\eta,T-\eta]\times \R^k$. So for  $(t,x) \in [T-2\eta,T-\eta]\times \R^k$, let us set $u^i(t,x) = \lim_{n \rightarrow \infty} u^{i,n}(t,x)$, $\ii$. Note that $(u^i)_{\ii}$ are continuous bounded functions on $[T-2\eta,T]\times \R^k$ by concatenation. Now by applying   repeatedly the same reasoning on each time interval $ [T-(j+1)\eta, T-j\eta]$ of fixed length $\eta$ and pasting the solutions, we obtain the uniform convergence of $((u^{i,n})_{\ii})_{n}$ in $[0,T] \times \R^k$. So for  $\ii$ and $(t,x) \in [0,T]\times \R^k$, let us set $u^i(t,x) = \lim_{n \rightarrow \infty} u^{i,n}(t,x)$, $\ii$. Note that $(u^i)_{\ii}$ are continuous bounded functions on $[0,T]\times \R^k$.
\ms

\noindent \underline{Step 3}: Convergence of $(Y^{i,n;t,x},Z^{i,n;t,x},V^{i,n;t,x},K^{i,n;t,x})_{n\ge 1}$
\ms 

For any any $\ii$, the convergence of $(Y^{i,n;t,x},Z^{i,n;t,x},V^{i,n;t,x},K^{i,n;t,x})_{n\ge 1}$ can be obtained exactly in the same way as in the proof of Proposition \ref{repvi} since for any $\ii$, $(u^{i,n})_{n\ge 1}$ converges uniformly to $u^i$. Consequently, there exists a quadruple of processes 
$(Y^{i,t,x},Z^{i,t,x},V^{i,t,x},K^{i,t,x})$ which belongs to $\ss \times \hdd \times \hld \times \aa$ such that: 
\begin{align*}
    &\E\Big[ \sup_{s\leq T} |Y_s^{i,n:t,x} -Y_s^{i,t,x}|^2 + \textstyle \int_0^T |Z_s^{i,n:t,x} -Z_s^{i,t,x}|^2ds\\[4pt]
    & +\textstyle \int_0^T\int_E |V_s^{i,n:t,x}(e) -V_s^{i,t,x}(e)|^2\lambda(ds)de+ \displaystyle \sup_{s\leq T}|K_s^{i,n:t,x} -K_s^{i,t,x}|^2\Big] \rightarrow 0\,\, \mbox{ as } n\rightarrow \infty.
\end{align*}
Moreover we have: For any $\ii$ and $s\in [t,T]$, 
\begin{equation}
\begin{split}
\begin{cases}
\vspace{0.3cm}
Y_s^{i,t,x} = h_i(X_T^{t,x})+ \int_s^T \bar f_i(r,X_r^{t,x},(Y_r^{k,t,x})_{k\in \mi},Z_r^{i,t,x}, \int_E V_r^{i,t,x}(e)\gamma_i(X_r^{t,x},e) \lambda(de))dr \\\vspace{0.3cm}
 \qquad \qquad   +K_T^{i,t,x} - K_s^{i,t,x} -\int_s^T Z_r^{i,t,x}dB_r -\int_s^T \int_E V_r^{i,t,x}(e) \tilde{\mu}(dr,de);\\\vspace{0.3cm}
 Y_s^{i,t,x} \geqslant  \displaystyle \max_{j \in \mathcal{I}^{-i}}(Y_s^{j,t,x} -g_{ij}(s));\\\vspace{0.3cm}
 \int_t^T (Y_s^{i,t,x} -\displaystyle \max_{j\in \mathcal{I}^{-i}}(Y_s^{j,t,x} -g_{ij}(s))) dK_s^{i,t,x} =0.
\end{cases}
\end{split}
\end{equation} 
Therefore $(Y^{i,t,x},Z^{i,t,x},V^{i,t,x},K^{i,t,x})_{\ii}$ is a solution of \eqref{sys2.13}. Finally as $((u^{i,n})_{\ii})_{n\ge 1}$ converges uniformly to $(u^{i})_{\ii}$ and 
$((V^{i,n;t,x})_{\ii})_{n\ge 1}$ converges to $(V^{i,t,x})_{\ii}$ in $\hld$
then by \eqref{Y,V}-iii) we obtain that \eqref{repV1} holds. This completes the proof.
\qed

Finally we focus on the uniqueness of the Markovian solution of the system of reflected BSDES \eqref{sys2.13} which we precise {\bc in the following definition}. 

\begin{definition}
A solution 
$(Y^{i,t,x},Z^{i,t,x},V^{i,t,x},K^{i,t,x})_{\ii}$ of system \eqref{sys2.13} is called Markovian if there exist bounded continuous deterministic functions $(\tilde u^i)_{\ii}$, on $[0,T]\times \R^k$, such that for any $\ii$, $\tx$,
\begin{align} \label{repY1uniq}
Y^{i,t,x}_s=\tilde u^i(s,\xtx_s), s\in [t,T],
\end{align} and 
\begin{align} \label{repV1uniq}
V_s^{i,t,x}(e) =  1_{\{s\ge t\}} (\tilde u^{i}(s,\xtx_{s-}+& \beta(\xtx_{s-},e))- \tilde u^{i}(s,\xtx_{s-})),\nonumber\\
& \qquad ds\otimes d\mathbb{P} \otimes d\lambda-a.e \mbox{ on } [0,T] \times \Omega \times E.
\end{align}
\end{definition}
We then have:
\begin{theoreme}\label{unic}
The Markovian solution of system \eqref{sys2.13} is unique. 
\end{theoreme}
\proof Let 
$(\bar Y^{i,t,x},\bar Z^{i,t,x},\bar V^{i,t,x},\bar K^{i,t,x})_{\ii}$ be a Markovian solution of \eqref{sys2.13} associated with the continuous bounded functions $(\tilde u^i)_{\ii}$. Therefore for any $\ii$ and $s\in [t,T]$,
$$\bar Y^{i,t,x}_s=\tilde u^i(s,\xtx_s)
$$
and  
\begin{equation}
\begin{split}
\begin{cases}
\vspace{0.3cm}
\bar Y_s^{i,t,x} = h_i(X_T^{t,x})+\bar K_T^{i,t,x} - \bar K_s^{i,t,x} -\int_s^T \bar Z_r^{i,t,x}dB_r -\int_s^T \int_E \bar V_r^{i,t,x}(e)  \\\vspace{0.3cm}
 \qquad + \int_s^T \bar f_i(r,X_r^{t,x},(\bar Y_r^{k,t,x})_{k\in \mi},\bar Z_r^{i,t,x}, \int_E(\tilde u^{i}(r,\xtx_{r-}+ \beta(\xtx_{r-},e))- \tilde u^{i}(r,\xtx_{r-}))\gamma_i(X_r^{t,x},e) \lambda(de))dr;\\\vspace{0.3cm}
 Y_s^{i,t,x} \geqslant  \displaystyle \max_{j \in \mathcal{I}^{-i}}(Y_s^{j,t,x} -g_{ij}(s));\\\vspace{0.3cm}
 \int_t^T (Y_s^{i,t,x} -\displaystyle \max_{j\in \mathcal{I}^{-i}}(Y_s^{j,t,x} -g_{ij}(s))) dK_s^{i,t,x} =0.
\end{cases}
\end{split}
\end{equation} 
In this equation, by continuity of the functions $\bar f_i$, $\beta$ and $\tilde u^i$ and since $X^{t,x}_{r}$ is rcll one can replace $X^{t,x}_{r-}$ with $X^{t,x}_{r}$. 

Next let us consider the triplet of processes $(P^{\d}, N^{\d}, Q^{\d})$ associated with the admissible strategy $\d \in \mathcal{A}_s^i$ and which solves: $$\left\{
  \begin{array}{l}
  (P^\d,N^{\d},Q^{\d})\in \spa;\\\\
  P_s^{\d} = h^{\d}(X_T^{t,x}) + \int_s^T f^{\d}(r,X_r^{t,x},  N_r^{\d}) dr - \int_s^T N_r^{\d} dB_r - \int_s^T \int_E Q_r^{\d}(e) \tilde{\mu}(dr,de) - A_T^{\d} + A_s^{\d},\,s\in [t,T]
  \end{array}\right.$$
  where, for any $s\in [t,T]$ when $\d_s =i$, $f^{\d}(s,X_s^{t,x},z)$ is equal to 
  $$\begin{array}{l}
  \bar{f}_i(s,X_s^{t,x}, (\tilde{u}^k(s,X_s^{t,x})_{k \in \mi},z, \int_E \gamma_i(X_s^{t,x},e) \{\tilde{u}^{i}(s,\xtx_{s-}+ \beta(\xtx_{s-},e))- \tilde{u}^{i}(s,\xtx_{s-})\} \lambda(de)).
 \end{array}$$
Therefore, we have the following representation of $\bar{Y}^i$:
\begin{equation*}
\bar{Y}_s^{i} = \displaystyle \mbox{esssup}_{\d \in \mathcal{A}_s^i}(P_s^{\d} - A_s^{\d}), \, \, \, s \leq T. 
\end{equation*} 
\def \bart{\tilde}
Next, arguing as in the Step 2 of the proof of Theorem \ref{existence2}, we deduce that for any $\ii$ and $(t,x)\in [0,T]\times \R^k$,
\begin{equation}\label{estiu5}
\begin{aligned}
|u^{i}(t,x) &- \bart{u}^{i}(t,x)|^2\leq
  \kappa_1  \E \big[ \textstyle \int_t^T e^{\alpha_0 (r-t)}\{\sum_{k=1,m}
  |u^{k}(r,\xtx_{r})- \bart{u}^{k}(r,\xtx_{r})|+ \\\\& \txtl
  \int_E |\{\sum_{k=1}^m(u^{k}- \bart{u}^{k})(r,\xtx_{r-} + \beta(\xtx_{r-},e))-(u^{k}- \bart{u}^{k})(r,\xtx_{r-})\}\gamma_k(\xtx_{r},e)|\lambda(de)\}^2dr\big]
\end{aligned}
\end{equation}
where $\kappa_1$
and $\alpha_0>0$ are constants which depend only on the Lipschitz constants of $(\bar f_i)_{\ii}$ w.r.t $\vec y$, $z$ and $q$. Next for $\eta >0$ and $t\in [T-\eta,T]$, let us set,  
\begin{eqnarray*}
||u^i-\bart u^i||_{\infty,\eta}:=\Sup_{(t,x)\in [T-\eta,T]\times \R^k}|u^i(t,x)-\bart u^i(t,x)|.
\end{eqnarray*}
From the boundedness of $u^i,\,\bart u^i$ and as $|\gamma_i(x,e)|\leq \cig (1\wedge |e|)$  for all $\ii$, then we have for all $t\in [T-\eta, T]$ and $x\in \R^k$,
\begin{equation}\label{estiu2}
\begin{array}{ll}|u^{i}(t,x) - \bart{u}^{i}(t,x)|^2&\leq  3m\kappa_1    \textstyle \int_t^T dr e^{\alpha_0 (r-t)}
\big \{\sum_{k=1,m}\|u^k-\bart u^k\|^2_{\infty,\eta}\\{}&\qq\qq+4(\max_{i=1,m}\cig)^2 (\int_E1\wedge |e|\lambda(de))^2\sum_{k=1,m}\|u^k-\bart u^k\|^2_{\infty,\eta}\big \}\\\\{}&
=\underbrace{\textstyle m(3\kappa_1 +4(C_\g)^2(\int_E1\wedge |e|\lambda(de))^2) }_{\Xi_1}\{\sum_{k=1,m}\|u^k-\bart u^k\|^2_{\infty,\eta}\textstyle \int_t^T dr e^{\alpha_0 (r-t)}\\{}&
\le \Xi_1\frac{e^{\alpha_0 \eta}-1}{\alpha_0}\sum_{k=1,m}\|u^k-\bart u^k\|^2_{\infty,\eta}.
\end{array}
\end{equation}
It implies that 
\begin{equation*}
\begin{array}{l}\sum_{k=1,m}\|u^k-\bart u^k\|^2_{\infty,\eta}\leq m \Xi_1\frac{e^{\alpha_0 \eta}-1}{\alpha_0}\sum_{k=1,m}\|u^k-\bart u^k\|^2_{\infty,\eta}. 
\end{array}
\end{equation*}
Take now $\eta_0$ such that $m \Xi_1\frac{e^{\alpha_0 \eta_0}-1}{\alpha_0}=\frac{1}{2}$ to obtain that 
$$
\sum_{k=1,m}\|u^k-\bart u^k\|^2_{\infty,\eta_0}=0$$
which means that for any $\ii$ and $(t,x)\in [T-\eta_0, T]\times \R^k$, 
$u^i(t,x)=\tilde u^i(t,x)$. Note that $\eta_0$ does not depend on $T$. Next from \eqref{estiu5}, we have: For any $t\in [T-2\eta_0,T-\eta_0]$,  
\begin{equation*}
\begin{aligned}
|u^{i}(t,x) &- \bart{u}^{i}(t,x)|^2\leq
  \kappa_1  \E \big[ \textstyle \int_t^{T-\eta_0} e^{\alpha_0 (r-t)}\{\sum_{k=1,m}
  |u^{k}(r,\xtx_{r})- \bart{u}^{k}(r,\xtx_{r})|+ \\\\& \txtl
  \int_E |\{\sum_{k=1}^m(u^{k}- \bart{u}^{k})(r,\xtx_{r-} + \beta(\xtx_{r-},e))-(u^{k}- \bart{u}^{k})(r,\xtx_{r-})\}\gamma_k(\xtx_{r},e)|\lambda(de)\}^2dr\big]
\end{aligned}
\end{equation*}
since for any $(t,x)\in [T-\eta_0, T]\times \R^k$, 
$u^i(t,x)=\tilde u^i(t,x)$. Arguing as previously to obtain that  
for any $\ii$ and $(t,x)\in [T-2\eta_0, T-\eta_0]\times \R^k$, 
$u^i(t,x)=\tilde u^i(t,x)$. Repeat now this procedure as many times as necessary to deduce that for any $\ii$, 
$u^i=\tilde u^i$, which completes the proof and henceforth, the Markovian solution of \eqref{sys2.13} is unique. \qed 
\bs

\section{Connection with systems of IPDEs with inter-connected obstacles}
In this section, we are going to study the existence and uniqueness of the viscosity solution of the IPDEs system \eqref{sys1.1}. The candidate for a solution are the functions $(u^i)_{\ii}$ defined in \eqref{repY1} by which we have the Feynman-Kac representation of $(Y^{i,t,x})_{\ii}$. To begin with, we recall the definition of the viscosity solution of the system. 
\begin{definition}\label{def4.1}
A family of deterministic continuous functions $\vec u:=(u^i)_{i\in \mi}$  is a viscosity supersolution (resp. subsolution) of \eqref{sys1.1} if: 
$\forall \ii$, 
\begin{align*}
& \mbox{a)} \,u^i(T,x) \geq (\mbox{resp.} \leq )\,\, h_i(x) , \,\,\forall x \in \R^k\,\,; \\
& \mbox{b) if}\,\phi \in {\cal C}^{1,2}([0,T] \times \R^k)\mbox{  is 
such that $(t,x) \in [0,T) \times \R^k$ a global minimum}\\
& \qquad \qquad \qquad \mbox{ (resp. maximum) point of $u^i - \phi$}
\end{align*}
then
\begin{align*}
\min \Big\lbrace &u^i(t,x) - \displaystyle \max_{j \in {\mathcal{I}^{-i}}}(u^j(t,x)-g_{ij}(t)) ;-\partial_t\phi(t,x) - \mathcal{L}\phi(t,x) - \mathcal{K}\phi(t,x)\\
 & \qq \qq - {\bar f_i}(t,x,(u^k(t,x))_{k=1,...,m},(\sigma^\top D_x\phi)(t,x), \mathcal{B}_iu^i(t,x))\Big\rbrace \geq \,\,(resp. \le )\,\,0. 
\end{align*}
We say that $\vec u:=(u^i)_{i\in \mi}$ is a viscosity solution of \eqref{sys1.1} if it is both  a supersolution and subsolution of \eqref{sys1.1}.
\end{definition}
Note that 
  in this definition, we have $\mathcal{B}_iu^i(t,x)$ instead of  $\mathcal{B}_i\phi(t,x)$ in the argument of $\bar f_i$, where $\phi$ is the test function. Indeed, $\mathcal{B}_iu^i(t,x)$ is well-posed since $u^i$ is bounded, $|\g_i(x,e)|\le \cig (1\wedge |e|)$ and $\lambda(.)$ integrates $(1\wedge |e|)_{e\in E}$.
\ms

We then have: 
\begin{theoreme} \label{thviscosite} Assume that Assumptions (H1)-(H4) are fulfilled. Then the bounded continuous functions $(u^i)_{\ii}$ of \eqref{repY1} 
are the unique viscosity solution of the system \eqref{sys1.1}.
\end{theoreme}
\proof We first show that $(u^i)_{\ii}$ of \eqref{repY13} are viscosity solution of \eqref{sys1.1}. So let us consider the following system of reflected BSDEs with jumps and interconnected obstacles: $\forall s\in [t,T]$ and $\ii$,
\begin{equation}\label{nvRBSDEpp3}
\begin{split}
\begin{cases}
\vspace{0.3cm} \underbar{Y}^{i,t,x}\in \ss,  \underbar{Z}^{i,t,x} \in \hdd, \underbar{V}^{i,t,x} \in \hld \mbox{ and } \underbar{K}^{i,t,x}\in \aa;\\ \vspace{0.3cm}
\underbar{Y}_s^{i,t,x} = h_i(X_T^{t,x})+ \int_s^T \bar f_i(r,X_r^{t,x},(\underbar{Y}_r^{k,t,x})_{k\in \mi},\underbar{Z}_r^{i,t,x}, \int_E \gamma_i(\xtx_r,e) \times \\\vspace{0.3cm}
 \quad \quad  \{u^{i}(r,\xtx_{r-}+ \beta(\xtx_{r-},e)) - u^{i}(r,\xtx_{r-})\} \lambda(de))dr+\underbar{K}_T^{i,t,x} - \underbar{K}_s^{i,t,x}\\\vspace{0.3cm} \quad -\int_s^T \underbar{Z}_r^{i,t,x}dB_r -\int_s^T \int_E \underbar{V}_r^{i,t,x}(e) \tilde{\mu}(dr,de);\\\vspace{0.3cm}
 \underbar{Y}_s^{i,t,x} \geqslant  \displaystyle \max_{j \in \mathcal{I}^{-i}}(\underbar{Y}_s^{j,t,x} -g_{ij}(s))\mbox{ and }\textstyle {\int_t^T (\underbar{Y}_s^{i,t,x} -\displaystyle \max_{j\in \mathcal{I}^{-i}}(\underbar{Y}_s^{j,t,x} -g_{ij}(s))) d\underbar{K}_s^{i,t,x} = 0.}
\end{cases}
\end{split}
\end{equation}
Thanks to the result in \cite{hamadene2015viscosity}, pp. 1745, this system admits a unique solution $(\underbar{Y}^{i,t,x}, \underbar{Z}^{i,t,x},\underbar{V}^{i,t,x},$
$\underbar{K}^{i,t,x})_{\ii}$. In fact, the generators of the system  \eqref{nvRBSDEpp3} do not depend on $\underbar{V}^{i,t,x}$.
Next by Theorem \ref{existence2}, there exist deterministic bounded continuous functions  $(\underbar{u}^{i})_{\ii}$, such that for any $(t,x)\in [0,T]\times \R^k,$
\begin{equation*}
  \forall s \in [t,T],\,\,\,\,\,\, \underbar{Y}_s^{i,t,x} = \underbar{u}^{i}(s,\xtx_s).
\end{equation*} 
Actually, mainly this due to the fact that for any $\ii$ the functions
$$(t,x,\vec y,z) \longmapsto \check F_i(t,x,\vec y,z)={\bar f_i}(t,x,\vec y,z,\textstyle \int_E \gamma_i(x,e)\{u^{i}(t,x+ \beta(x,e))  - u^{i}(t,x)\} \lambda(de))$$verify Assumption (H1) and 
$\check F_i(t,x,\vec 0,0)$ is bounded since $\bar f_i(.)$ verify (H1), $u^i$ is continuous bounded, $|\g_i(x,e)|\le \cig (1\wedge |e|)$ and $\lambda(.)$ integrates $(1\wedge |e|)_{e\in E}$. Next using the result by Hamad\`ene-Zhao \cite{hamadene2015viscosity}, pp.1745, we deduce that $(\underbar{u}^{i})_{\ii}$ is the unique viscosity solution of the following system:  For any $(t,x)\in [0,T] \times \R^k$,
\begin{equation}\label{eqnvpp3}
\begin{cases}
\min \lbrace \underbar{u}^i(t,x) - \displaystyle \max_{j \in {\mathcal{I}^{-i}}}(\underbar{u}^j(t,x)-g_{ij}(t)) ;  -\partial_t\underbar{u}^i(t,x) - \mathcal{L}\underbar{u}^i(t,x)- \mathcal{K}\underbar{u}^i(t,x)\\
\qquad - \bar f_i(t,x,(\underbar{u}^k(t,x))_{k=1,,...,m},(\sigma^\top D_x\underbar{u}^i)(t,x),\mathcal{B}_iu^i(t,x))\rbrace = 0;\\\\

\underbar{u}^i(T,x) = h_i(x).
\end{cases}
\end{equation}
Let us point out that, in this system \eqref{eqnvpp3}, the last component of $\bar{f}_i$ is   $\mathcal{B}_iu^i(t,x)$ and not $\mathcal{B}_i\underbar{u}^i(t,x)$.
Now, recall that $(Y^{i,t,x},Z^{i,t,x},V^{i,t,x},K^{i,t,x})_{\ii}$ solves the system  \eqref{sys2.13} and by Theorem \ref{existence2}, we know that for any $\tx$, $\ii$ and $s\in [t,T]$, 

$$\vspace{0.3cm}
V_s^{i,t,x}(e) =  u^{i}(s,\xtx_{s-}+ \beta(\xtx_{s-},e))- u^{i}(s,\xtx_{s-}),\,\, ds\otimes d\mathbb{P} \otimes d\lambda \mbox{ on } [t,T] \times \Omega \times E.
$$
Plug this relation in the second term of the right-hand side of the second equality of \eqref{sys2.13}, we obtain that
$(Y^{i,t,x},Z^{i,t,x},V^{i,t,x},K^{i,t,x})_{\ii}$ verify: for any $s\in [t,T]$ and $\ii$,
 \begin{equation}
\begin{split}
\begin{cases}
\vspace{0.15cm}
Y_s^{i,t,x} = h_i(X_T^{t,x}) +\int_s^T \bar f_i(r,X_r^{t,x},(Y_r^{k,t,x})_{k \in \mi},Z_r^{i,t,x},\int_E \gamma_i(\xtx_r,e)\times\\\vspace{0.15cm}
 \quad \quad \{u^{i}(r,\xtx_{r-}+ \beta(\xtx_{r-},e))   - u^{i}(r,\xtx_{r-})\}\lambda(de))dr+K_T^{i,t,x}- K_s^{i,t,x}\\\vspace{0.35cm}
 \qquad - \int_s^T Z_r^{i,t,x}dB_r -\int_s^T \int_E V_r^{i,t,x}(e) \tilde{\mu}(dr,de);\\\vspace{0.15cm}
Y_s^{i,t,x} \geqslant  \displaystyle \max_{j \in \mathcal{I}^{-i}}(Y_s^{,t,xj} -g_{ij}(s));\\\vspace{0.15cm}
\int_t^T (Y_s^{i,t,x} -\displaystyle \max_{j\in \mathcal{I}^{-i}}(Y_s^{j,t,x} -g_{ij}(s))) dK_s^{i,t,x} = 0.
\end{cases}
\end{split}
\end{equation}
Then, by uniqueness of the solution of the system  of reflected BSDEs   \eqref{nvRBSDEpp3}, we deduce that for any $s \in [t,T]$ and $\ii$, $\underbar{Y}^{i,t,x}_s = Y^{i,t,x}_s$. Then, for any  $\ii$, $ \underbar{u}^i = u^i$. Consequently, $(u^i)_{i \in \mi}$ is a viscosity solution of \eqref{sys1.1} according to Definition \ref{def4.1}. 
\ms
 
We now focus on the uniqueness of the viscosity solution of system \eqref{sys1.1}. 
\ms

So suppose that there exist bounded continuous functions denoted by $(\tilde u^i)_{i \in \mathcal{I}}$ viscosity solution of the system of IPDEs \eqref{sys1.1} according to Definition \ref{def4.1}.
Let us consider the following system of RBSDEs with jumps and interconnected obstacles : for any $i \in \mathcal{I}$ and $s\in [t,T]$,
\begin{equation}\label{eqBSDE2}
\begin{split}
\begin{cases}
\vspace{0.3cm} \bar{Y}^{i,t,x}\in \ss,  \bar {Z}^{i,t,x} \in \hdd, \bar{V}^{i,t,x} \in \hld, \mbox{ and } \bar{K}^{i,t,x}\in \aa;\\ \vspace{0.3cm}
\bar{Y}_s^{i,t,x} = h_i(X_T^{t,x})
+ \int_s^T \bar f_i(r,X_r^{t,x},(\bar{Y}_r^{k,t,x})_{k\in \mi},\bar{Z}_r^{i,t,x}, B_i \tilde u^i (r, \xtx_{r}))dr\\\vspace{0.3cm}
+\bar{K}_T^{i,t,x} - \bar{K}_s^{i,t,x}-\int_s^T \bar{Z}_r^{i,t,x}dB_r-\int_s^T \int_E \bar{V}_r^{i,t,x}(e) \bar{\mu}(dr,de);\\\vspace{0.3cm}
 \bar{Y}_s^{i,t,x} \geqslant  \displaystyle \max_{j \in \mathcal{I}^{-i}}(\bar{Y}_s^{j,t,x} -g_{ij}(s))\mbox{ and }\textstyle {\int_t^T (\bar{Y}_s^{i,t,x} -\displaystyle \max_{j\in \mathcal{I}^{-i}}(\bar{Y}_s^{j,t,x} -g_{ij}(s))) d\bar{K}_s^{i,t,x} = 0},
\end{cases}
\end{split}
\end{equation}
where we recall that $$\begin{array}{c}B_i\tilde u^i (r, \xtx_{r})=\int_E \gamma_i(X_r^{t,x},e) (\tilde{u}^{i}(r,\xtx_{r}+ \beta(\xtx_{r},e))- \tilde{u}^{i}(r,\xtx_{r})) \lambda(de).\end{array}$$
From Proposition 4.2 of Hamad\`ene-Zhao \cite{hamadene2015viscosity}, there exists a unique solution $(\bar Y^i,\bar Z^i, \bar V^i, \bar K^i)_{i\in {\cal I}}$ to the system of RBSDEs  \eqref{eqBSDE2} and a bounded continuous functions $(\bar {u}^i)_{i\in {\cal I}}$ such that
\begin{equation*}
 \forall s \in [t,T], \, \, \bar{Y}_s^{i,t,x} = \bar {u}^i(s,X_s^{t,x}).
\end{equation*}
The boundedness can be obtained as in the proof of Proposition \ref{firstresult}. Moreover $(\bar {u}^i)_{i\in {\cal I}}$ is the unique viscosity solution to the following system of IPDEs :
\begin{equation}\label{edpo1}
\begin{cases}
\min \lbrace  v^i(t,x) - \displaystyle \max_{j \in {\mathcal{I}^{-i}}}( v^j(t,x)-g_{ij}(t)) ;-\partial_t v^i(t,x) - \mathcal{L} v^i(t,x)- \mathcal{K} v^i(t,x)\\
  - f_i(t,x,( v^k(t,x))_{k=1,...,m},(\sigma^\top D_x v^i)(t,x), B_i \tilde u^i(t,x))\rbrace = 0, \,
 (t,x)\in [0,T) \times \R^k;\\\\
 v^i(T,x) = h_i(x),\, x\in \R^k.
\end{cases}
\end{equation}
Therefore by uniqueness, for any $\ii$, $\bar u^i=\tilde u^i$, and then, $B_i\bar u^i (t,x)=B_i\tilde u^i (t,x)$ and 
$$
\bar V^{i,t,x}=\bar u^{i}(s,\xtx_{s-}+\beta(\xtx_{s-},e))- \bar u^{i}(s,\xtx_{s-})=
\tilde u^{i}(s,\xtx_{s-}+\beta(\xtx_{s-},e))- \tilde u^{i}(s,\xtx_{s-})$$
since once more, as previously, 
$$
\bar F_i(t,x,\vec y,z):=
\bar f_i(t,x,\vec y,z,B_i\tilde u^i (t,x)), \ii,
$$
verify the assumptions of Theorem \ref{existence2}. Therefore 
$(\bar Y^{i,t,x}=\bar u^i(s,\xtx_s),\bar Z^{i,t,x},\bar V^{i,t,x}=
\bar u^{i}(s,\xtx_{s-}+\beta(\xtx_{s-},e))- \bar u^{i}(s,\xtx_{s-}),\bar K^{i,t,x})_{\ii}$ is a Markovian solution of \eqref{sys2.13}. Thus by uniqueness (see Theorem \ref{unic}) one has: For any $\ii$, 
$\forall s\in [t,T], \bar Y^{i,t,x}_s=Y^{i,t,x}_s=u^i(s,\xtx_s)=\bar u^i(s,\xtx_s)=\tilde u^i(s,\xtx_s).$
Finally in taking $s=t$ one deduces that for any $\ii$ and $\tx$, $u^i(t,x)=\bar u^i(t,x)=\tilde u^i(t,x)$ which means that the viscosity solution of \eqref{sys1.1} is unique. \qed
\section {Appendix}
\begin{lemme}\lb{lemmeccv} Let $f$ be  a function from $[0,T]\times \R^k\times \R^{\ell}\rw \R$. Assume that the function $(t,x)\in [0,T]\times \R^k \mapsto f(t,x,w)\in \R$ is continuous, uniformly w.r.t. $w$, i.e., for any $(t,x)\in \espo$, for any 
$\eps >0$ there exists $\eta_{t,x,\eps}>0$ such that if $|(t',x')-(t,x)|<\eta_{t,x,\eps}$ then 
$|f(t',x',w)-f(t,x,w)|<\eps$. Then 
for any $R>0$, there exists a non-decreasing concave function $\Phi_R: \R^+\rw \R^+$ such that $\Phi_R(0)=0$ and for any $t,t'\in [0,T]$, $x,x'\in B'(0,R)$, $w\in \R^{\ell}$, 
\begin{equation} \lb{modcontccv}
|f(t,x,w)-f(t',x',w)|\le \Phi_R(|t-t'|+|x-x'|).
\end{equation}
\end{lemme}

\nd $\proof$ Let $R>0$, $(t,x)\in [0,T]\times B'(0,R)$ and $\eps>0$. By definition, there exists 
$\eta_{t,x,\eps}>0$ such that if $(t',x')\in B((t,x),\eta_{t,x,\eps})$ (the open ball in $\espo$ with center $(t,x)$ and radius $\eta_{t,x,\eps}$) then 
$$|f(t',x',w)-f(t,x,w)|<\eps.$$
As $$[0,T]\times B'(0,R)\subset \bigcup_{(t,x)\in [0,T]\times B'(0,R)}B((t,x),\frac{\eta_{t,x,\eps}}{2})
$$
then by compacity one can find finitely many points $(t_1,x_1), ...., (t_m,x_m)$ such that 
$$[0,T]\times B'(0,R)\subset \bigcup_{i=1,m}B((t_i,x_i),\frac{\eta_{t_i,x_i,\eps}}{2}).$$
\ms

\nd Next let $R>0$ and $\Psi_R$ defined as follows: $$\begin{array}{ll}
\forall \g \geq 0, \Psi_R(\g):=&\!\!\!\!\!\sup\{|f(t,x,w)-f(t',x',w)|,\,|t-t'|+|x-x'|\leq \g,\\ &\qq\qq \qq\qq t,t'\in [0,T], x,x'\in B'(0,R),w\in \R^{\ell }\}.
\end{array}$$
Then:

\nd (i) First note that the function $\gamma \in \R^+\mapsto \Psi_R(\g)$ is non-decreasing and $\Psi_R(0)=0$. Next let us set $\eta=\min_{i=1,m}\frac{\eta_{t_i,x_i,\eps}}{3}$. Then $\eta>0$ and we claim that $\Psi_R(\eta)\in \R^+$. 

Indeed let $w\in \R^\ell$ and $(t,x)$, $(t',x')$ $\in [0,T]\times B'(0,R)$ such that 
$|t-t'|+|x-x'|\leq \eta$. Then there exists $i\in\{1,...,m\}$ such that $(t,x)\in B((t_i,x_i),\frac{\eta_{t_i,x_i,\eps}}{2})$. It follows that 
$$
|(t',x')-(t_i,x_i)|\leq |(t',x')-(t,x)|+|(t,x)-(t_i,x_i)|\leq \eta +\frac{\eta_{t_i,x_i,\eps}}{2}<\eta_{t_i,x_i,\eps}
$$
which implies that $(t',x')$ belongs also to $B((t_i,x_i),\eta_{t_i,x_i,\eps})$. Then by continuity we have:  
$$
|f(t,x,w)-f(t',x',w)|\leq |f(t,x,w)-f(t_i,x_i,w)|+|f(t_i,x_i,w)-f(t',x',w)|\leq 2\eps. 
$$
Taking the supremum to obtain that $\Psi_R(\eta)\leq 2\eps$. 
\ms

\nd (ii) As $\Psi_R$ is non-decreasing then $\Psi_R(\gamma)\leq 2\eps $ for any $\gamma\leq \eta$. Note that this property implies also that $\Psi_R(\gamma)\rw 0=\Psi_R(0)$ as $\gamma \rw 0$. 
\ms

\nd (iii) For any $\g_1$ and $\g_2$ in $\R^+$, $\Psi_R(\g_1+\g_2)\leq  \Psi_R(\g_1)+\Psi_R(\g_2)$.

Indeed let $w\in \R^\ell$ and $(t,x)$, $(t',x')$ elements of $[0,T]\times B'(0,R)$ such that 
$|t-t'|+|x-x'|\leq \g_1+\g_2$. Then there exists $(\bar t,\bar x)\in [0,T]\times B'(0,R)$ such that 
$
|(t,x)-(\bar t,\bar x)|\leq \g_1$ and $|(\bar t,\bar x)-(t',x')|\leq \g_2$. Therefore 
$$
|f(t,x,w)-f(t',x',w)|\leq |f(t,x,w)-f(\bar t,\bar x,w)|+|f(\bar t,\bar x,w)-f(t',x',w)|\leq  \Psi_R(\g_1)+\Psi_R(\g_2)
$$
which implies that $\Psi_R(\g_1+\g_2)\leq  \Psi_R(\g_1)+\Psi_R(\g_2)$.
\ms 

\nd (iv) For any $\gamma \in \R^+$, $\Psi_R(\gamma)\in \R^+$ and $\Psi_R(.)$ verifies (\ref{modcontccv}). 

Indeed by induction and (iii), for any $\gamma \in \R^+$ and $n\ge 1$, $\Psi_R(n\g)\leq n\Psi_R(\gamma)$. On the other hand, by (ii) and (iii), for any $\g\in \R^+$ one can find an integer $n$ such that $\Psi_R(\g)\leq \Psi_R(n\eta)\le n\Psi_R(\eta)$ which implies that $\Psi_R(\gamma)\in \R^+$. 
Now it is enough to take $\Phi_R$, for any $R>0$, the smallest concave non-decreasing function which majorizes the function $\Psi_R$ which exists (see e.g. \cite{sb}, Lemma 11). Finally 
(\ref{modcontccv}) is obviously satisfied and $$
|f(t,x,w)-f(t,x',w)|\leq
\Phi_R(|x-x'|)$$\ for all $t\in (0,T)$, $|x|$, $|x'|$,
$w\in \R^\ell$.  \qed  
\bs 

\nd {\bf Proof of Lemma \ref{truncation}}: We prove only \eqref{convergenceXn} and  \eqref{cvgence2} as \eqref{estimXn} is classical (see e.g. \cite{{fujiwara1985stochastic}}, Lemma 2.1). Let $p\ge 1$ and $m\ge n$. For any $s\in [0,T]$, we have:
 \begin{align*}
     {}^{n}\!X_s^{t,x}-{}^{m}\!X_s^{t,x} =& \int_0^s (b(r,{}^{n}\!X_r^{t,x})- b(r,{}^{m}\!X_r^{t,x}))dr + \int_0^s (\sigma(r,{}^{n}\!X_r^{t,x})- \sigma(r,{}^{m}\!X_r^{t,x}))dB_r \\[4pt]
     & \qquad \int_0^s \int_E (\beta(e,{}^{n}\!X_{r_-}^{t,x})\mathbf{1}_{\lbrace|e|\geq \frac{1}{n}\rbrace} - \beta(e,{}^{m}\!X_{r_-}^{t,x}))\tilde{\mu}_m(dr,de)
 \end{align*}
 Since $|a+b+c|^{2p} \leq 3^{2p-1}(|a|^{2p}+|b|^{2p}+|c|^{2p})$, for any $p\geq 1$ and $a$, $b$, $c\in \R$, then we have:
 \begin{equation*}
     \begin{aligned}
         &\E \Big[ \sup_{0\leq s\leq \eta}\big|{}^{n}\!\xtx_s - {}^{m}\!\xtx_s\big|^{2p}\Big] \\[4pt]
&  \leq 3^{2p-1} \E\Big[  \sup_{0\leq s\leq \eta}\big|\int_0^{s}  b(r, {}^{n}\!\xtx_{r}) - b(r, {}^{m}\!\xtx_{r} )dr\big|^{2p}  +\sup_{0\leq s\leq \eta}\big|\int_0^s \sigma(r, {}^{n}\!\xtx_{r}) - \sigma(r, {}^{m}\!\xtx_{r})  dB_r\big|^{2p}\\[4pt]
& \qquad +  \sup_{0\leq s\leq \eta} \big| \int_0^s \int_E (\beta(e,{}^{n}\!X_{r_-}^{t,x})\mathbf{1}_{\lbrace|e|\geq \frac{1}{n}\rbrace} - \beta(e,{}^{m}\!X_{r_-}^{t,x}))\tilde{\mu}_m(dr,de)\big|^{2p} \Big].
     \end{aligned}
 \end{equation*}
Next by the Cauchy-Schwarz, B-D-G and generalized B-D-G inequalities (see \cite{jacka-hernandez} for this latter) we have: $\forall \eta \in [0,T]$,
 \begin{equation}
     \begin{aligned}
         &\E \Big[ \sup_{0\leq s\leq \eta}\big|{}^{n}\!\xtx_s - {}^{m}\!\xtx_s\big|^{2p}\Big] \\[4pt]
& \leq C\,\Big\{ \E \Big[ \int_0^{\eta} \sup_{0\leq \tau\leq r} \big| b(\tau, {}^{n}\!\xtx_{\tau}) - b(\tau, {}^{m}\!\xtx_{\tau})\big|^{2p}dr\Big]+ \E \Big[ \Big(\int_0^{\eta}\sup_{0\leq \tau\leq r} \big| \sigma(\tau, {}^{n}\!\xtx_{\tau}) - \sigma(\tau, {}^{m}\!\xtx_{\tau})\big|^{2}dr\Big)^{p}\Big]\\[4pt]
&  \qq+  \E \Big[ \Big( \int_0^{\eta} \int_E \big|\beta(e,{}^{n}\!X_{r-}^{t,x})\mathbf{1}_{\lbrace|e|\geq \frac{1}{n}\rbrace} - \beta(e,{}^{m}\!X_{r-}^{t,x}))\big|^{2} \lambda_m(de)dr\Big)^{p}\Big] \\[4pt]
&  \qq+ \E \Big[ \int_0^{\eta} \int_{E}  \big|\beta(e,{}^{n}\!X_{r-}^{t,x})\mathbf{1}_{\lbrace|e|\geq \frac{1}{n}\rbrace}- \beta(e,{}^{m}\!X_{r-}^{t,x})\big|^{2p} \lambda_m(de)dr\Big]\Big\}=:C\{\Sigma_1+\Sigma_2+\Sigma_3+\Sigma_4\}
     \end{aligned}\lb{eqlemme31}
 \end{equation}
where $C$ is a constant which depends on $p$ and which may change from line to line. As $b$ is Lipschitz w.r.t $x$, then  we have:  For any $\eta \leq T$,
\begin{align*}
  \Sigma_1
\leq  C\E \Big[ \int_0^{\eta}\sup_{0\leq \tau \leq r} \big|{}^{n}\!\xtx_{\tau} - {}^{m}\!\xtx_{\tau}\big|^{2p}d\tau\Big].
\end{align*}
Besides using Jensen'inequality and the fact that $\sigma$ is Lipschitz w.r.t $x$, we get: For any $\eta \leq T$,
\begin{align*}
   \Sigma_2\leq C\E\Big[ \int_0^{\eta}\sup_{0\leq \tau\leq r} \big| \sigma(\tau, {}^{n}\!\xtx_{\tau}) - \sigma(\tau, {}^{m}\!\xtx_{\tau})\big|^{2p}d\tau\Big]\leq C\E\Big[ \int_0^{\eta}\sup_{0\leq \tau\leq r} \big|{}^{n}\!\xtx_{\tau} - {}^{m}\!\xtx_{\tau}\big|^{2p} d\tau\Big].
\end{align*}
Next using the fact that  $\beta$ verifies \eqref{hyp3}, the inequality $(|x|+|y|)^p\leq C(|x|^p+|y|^p)$ $(x,y\in \R)$ and Jensen's one to deduce, 
\begin{align*}
 \Sigma_3&:=\E\Big[\Big( \int_0^{\eta} \int_E \big|\beta(e,{}^{n}\!X_{r-}^{t,x})\mathbf{1}_{\lbrace|e|\geq \frac{1}{n}\rbrace}- \beta(e,{}^{m}\!X_{r-}^{t,x}))\big|^{2} \lambda_m(de)dr\Big)^{p}\Big] \\& 
 =\E\Big[\Big( \int_0^{\eta} \int_{|e|\ge \frac{1}{n}} \big|\beta(e,{}^{n}\!X_{r-}^{t,x})- \beta(e,{}^{m}\!X_{r-}^{t,x}))\big|^{2} \lambda(de)dr+
 \int_{\frac{1}{m}\le |e|<\frac{1}{n}} \big|\beta(e,{}^{m}\!X_{r-}^{t,x}))\big|^{2} \lambda(de)dr
 \Big)^{p}\Big] \\ 
&\leq C
\E\Big[\Big( \int_0^{\eta} \int_{|e|\ge \frac{1}{n}} \big|\beta(e,{}^{n}\!X_{\tau_-}^{t,x})- \beta(e,{}^{m}\!X_{\tau_-}^{t,x}))\big|^{2} \lambda(de)dr\Big)^p+\Big(
 \int_{\frac{1}{m}\le |e|<\frac{1}{n}} \big|\beta(e,{}^{m}\!X_{\tau_-}^{t,x}))\big|^{2} \lambda(de)dr
 \Big)^{p}\Big] 
 \\[4pt]
 & \quad \leq  C \E\Big[ \Big( \int_0^{\eta}\sup_{0\leq \tau \leq r} \big|{}^{n}\!\xtx_{\tau} - {}^{m}\!\xtx_{\tau}\big|^2 \int_E (1\wedge |e|)^2 \lambda(de)dr\Big)^{p}\Big]+C\Big(\int_{\lbrace\frac{1}{m} \leq |e| \leq  \frac{1}{n}\rbrace} (1\wedge|e|^2) \lambda(de)\Big)^{p}\\[4pt]
& \quad \leq  C \E\Big[ \int_0^{\eta}dr\sup_{0\leq \tau \leq r} \big|{}^{n}\!\xtx_{\tau} - {}^{m}\!\xtx_{\tau}\big|^{2p}\Big]+C\Big(\int_{\lbrace\frac{1}{m} \leq |e| <  \frac{1}{n}\rbrace} (1\wedge|e|^2) \lambda(de)\Big)^{p}.
 \end{align*}
The last inequality stems from the fact that $\lambda_n(.)$ is $\lambda(.)$ truncated and this latter integrates $(1\wedge |e|^2)_{e\in E}$. Finally, once more by \eqref{hyp3}, we have
\begin{align*}
 \Sigma_4&:=\E\Big[ \int_0^{\eta} \int_E \big|\beta(e,{}^{n}\!X_{r-}^{t,x})\mathbf{1}_{\lbrace|e|\geq \frac{1}{n}\rbrace}- \beta(e,{}^{m}\!X_{r-}^{t,x}))\big|^{2p} \lambda_m(de)dr\Big] \\& 
 =\E\Big[ \int_0^{\eta} \int_{|e|\ge \frac{1}{n}} \big|\beta(e,{}^{n}\!X_{r-}^{t,x})- \beta(e,{}^{m}\!X_{r-}^{t,x}))\big|^{2p} \lambda(de)dr+
 \int_{\frac{1}{m}\le |e|<\frac{1}{n}} \big|\beta(e,{}^{m}\!X_{\tau_-}^{t,x}))\big|^{2p} \lambda(de)dr
 \Big] \\[4pt]
 & \quad \leq  C \E\Big[\int_0^{\eta}dr\sup_{0\leq \tau \leq r} \big|{}^{n}\!\xtx_{\tau} - {}^{m}\!\xtx_{\tau}\big|^{2p} \int_E (1\wedge |e|)^{2p} \lambda(de)\Big]+C\int_{\lbrace\frac{1}{m} \leq |e| <  \frac{1}{n}\rbrace} (1\wedge|e|^{2p}) \lambda(de)\\[4pt]
& \quad \leq  C \E\Big[  \int_0^{\eta}dr\sup_{0\leq \tau \leq r} \big|{}^{n}\!\xtx_{\tau} - {}^{m}\!\xtx_{\tau}\big|^{2p}\Big]+C\int_{\lbrace\frac{1}{m} \leq |e| <  \frac{1}{n}\rbrace} (1\wedge|e|^{2p}) \lambda(de).
 \end{align*}
Plug now those four last inequalities in \eqref{eqlemme31} to obtain: $ \forall \,  \eta \leq T$,
\begin{align*}
&\E \Big[ \sup_{0\leq s\leq \eta}\big|{}^{n}\!\xtx_s - {}^{m}\!\xtx_s\big|^{2p}\Big] \\[4pt]
& \quad \leq C_p \E \Big[ \int_0^{\eta} dr\sup_{0\leq \tau \leq r} \big|{}^{n}\!\xtx_{\tau} - {}^{m}\!\xtx_{\tau}\big|^{2p}  + \Big(\int_{\lbrace\frac{1}{m} \leq |e| <  \frac{1}{n}\rbrace} (1\wedge|e|^2) \lambda(de)\Big)^{p} +\int_{\lbrace\frac{1}{m} \leq |e| <  \frac{1}{n}\rbrace} (1\wedge|e|^{2p}) \lambda(de)\Big].
\end{align*}
Finally, using Gronwall's inequality we obtain: For any $\eta \in [0,T]$, 
\begin{align}\label{estilmn}
\E \Big[ \sup_{0\leq s\leq \eta}\big|{}^{n}\!\xtx_s - {}^{m}\!\xtx_s\big|^{2p}\Big]  \leq C \Big\{\underbrace{\Big(\int_{\lbrace\frac{1}{m} \leq |e| <  \frac{1}{n}\rbrace} (1\wedge|e|^2) \lambda(de) \Big)^{p}+\int_{\lbrace\frac{1}{m} \leq |e| < \frac{1}{n}\rbrace} (1\wedge|e|^{2p}) \lambda(de)}_{\Lambda^{n,m}_p}\Big \}.
\end{align}
Then in taking $\eta=T$, we obtain the desired result.
\ms

Concerning the property \eqref{cvgence2}, it is enough to consider ${}^{n}\!\xtx - \xtx$ and by the same procedure as previously we obtain: For any $p\ge 1$, 
\begin{align}\label{estilmn}
\E \Big[ \sup_{0\leq s\leq T}\big|{}^{n}\!\xtx_s -\xtx_s\big|^{2p}\Big]  \leq C \Big\{\underbrace{\Big(\int_{\lbrace 0<|e| \leq  \frac{1}{n}\rbrace} (1\wedge|e|^2) \lambda(de) \Big)^{p}+\int_{\lbrace0< |e| \leq  \frac{1}{n}\rbrace} (1\wedge|e|^{2p}) \lambda(de)}_{\Lambda_{n}(p)}\Big \}.
\end{align}
The result follows since $\Lambda_n(p) \rw 0$ as $n\rw +\infty$.

\nd \begin{lemme}\lb{estima}Let $\d^*=(\tau_i^*,\xi_i^*)_{i\ge 0}$ be the optimal strategy for the switching problem associated with $((f_i(t,X^n_t,.))_{\ii},(g_{ij}(t)))_{\ij},(h_i(X^n_T)))_{\ii})$ where 
$(f_i)_{\ii}$, $(h_i)_{\ii}$ and $(g_{ij})_{\ij}$ satisfy the assumptions of Theorem \ref{firstresult}. Then for any $q\ge 1$, there exists positive constant $C_q$, which does not depend on $n$, such that:
\begin{equation} \lb{apdix2}\E[(A_T^{\delta^*})^q]\le C_q.
\end{equation}
\end{lemme}
\begin{proof}Recall that the construction of $\d^*=(\tau_i^*,\xi_i^*)_{i\ge 0}$ (with 
$\tau_0^*=0$ and $\xi_0^*=i$) stems from the solution $(Y^{i,n},Z^{i,n},V^{i,n},K^{i,n})_{\ii}$ of the system of reflected BSDEs (one can see \cite{hamadene2015systems}, pp.1652 for more details): $\forall i = 1,...,m$ and  $s \in [0,T]$,
 \begin{equation}
 \label{BSDEtronquebis}
\begin{split}
\begin{cases}
\vspace{0.3cm} {}^{n}\!\underbar Y^{i,t,x}\in \ss,  {}^{n}\!\underbar Z^{i,t,x} \in \hdd, {}^{n}\!\underbar V^{i,t,x} \in \hln, \mbox{ and } {}^{n}\!\underbar K^{i,t,x}\in \aa;\\ \vspace{0.3cm}
{}^{n}\!\underbar Y_s^{i,t,x} = h_i({}^{n}\!X_T^{t,x})+ \int_s^T  \widehat{f}_i(r,{}^{n}\!X_r^{t,x},({}^{n}\!Y_r^{j,t,x})_{j\in \mi},{}^{n}\!\underbar Z_r^{i,t,x})dr +{}^{n}\!\underbar K_T^{i,t,x} - {}^{n}\!\underbar K_s^{i,t,x} \\\vspace{0.3cm}
 \qquad \qquad   -\int_s^T {}^{n}\!\underbar Z_r^{i,t,x}dB_r -\int_s^T \int_E {}^{n}\!\underbar V_r^{i,t,x}(e) \tilde{\mu}_n(dr,de);\\\vspace{0.3cm}
 {}^{n}\!\underbar Y_s^{i,t,x} \geqslant  \displaystyle \max_{j \in \mathcal{I}^{-i}}({}^{n}\!\underbar Y_s^{j,t,x} -g_{ij}(s));\\\vspace{0.3cm}
 \textstyle {\int_0^T ({}^{n}\!\underbar Y_s^{i,t,x} -\displaystyle \max_{j\in \mathcal{I}^{-i}}({}^{n}\!\underbar Y_s^{j,t,x} -g_{ij}(s))) d{}^{n}\!\underbar K_s^{i,t,x} = 0}.
\end{cases}
\end{split}
\end{equation}
We are going first to show that for any $q\ge 1$ there exists a constant $C_q$ which depend only on the constants of boundedness of $h_i$, $g_{ij}$ and $\widehat f^i(t,x,\vec 0,0)$, $\ij$, and $q$ such that: 
\begin{equation}\lb{estilemanx}
\E\big[\big\{\int_0^T
|^n\underbar Z_r^{i,t,x}|^2dr\big \}^q+\big\{\int_0^T \{\int_E|^n\underbar V_r^{i,t,x}(e)|^2\mu_n(dr,de)\}dr\}^q\big]\le C_q.\end{equation}
Actually first note that for fixed $i$ in ${\cal I}$, 
$(Y^{i,n},Z^{i,n},V^{i,n},K^{i,n})$ is a solution of a reflected BSDE with bounded barrier and bounded 
terminal condition, and the function $\widehat f_i(t,x,\vec 0,0)$ as well. Next by It\^o's formula it holds that: For any $\ii$ and $\forall s\le T$,
\begin{align}
(^n\!\sl Y_s^{i,t,x})^2& +\int_s^T |{}^{n}\!\underbar Z_r^{i,t,x}|^2dr+\int_t^T\int_E|^n\!\sl V_r^{i,t,x}(e)|^2{\mu}_n(dr,de)\nnb\\&=h_i({}^{n}\!X_T^{t,x})^2+ 2\int_s^T {^n\!\sl Y}_r^{i,t,x} \widehat{f}_i(r,{}^{n}\!X_r^{t,x},({}^{n}\!Y_r^{j,t,x})_{j\in \mi},{}^{n}\!\underbar Z_r^{i,t,x})dr
+2\int_s^T {^n\!\sl Y}_r^{i,t,x}d{}^{n}\!\underbar K_r^{i,t,x} \nnb\\&-2\int_s^T {{}^{n}\!\sl Y_r^{i,t,x}}\,{}^{n}\!\underbar Z_r^{i,t,x}dB_r -2\int_s^T \int_E {^n\!\sl Y}_r^{i,t,x} \,{^n\!\sl V}_r^{i,t,x}(e) \tilde{\mu}_n(dr,de).\nnb
\end{align}
Now since the processes $(Y^{i,n})_{\ii}$ are bounded and since, $h^i(.)$ and $\hat f^i(t,x,\vec 0,0)$ are bounded then for any $\eta>0$ there exists a constant $C_\eta$ (which depends on $\eta$) such that: $\fst$, 
$$
h_i({}^{n}\!X_T^{t,x})^2+ 2\int_s^T {^n\!\sl Y}_r^{i,t,x} \widehat{f}_i(r,{}^{n}\!X_r^{t,x},({}^{n}\!Y_r^{j,t,x})_{j\in \mi},{}^{n}\!\underbar Z_r^{i,t,x})dr\le C_\eta +\frac{1}{\eta}\int_s^T|{}^{n}\!\underbar Z_r^{i,t,x}|^2dr
$$
and $$
\int_s^T {^n\!\sl Y}_r^{i,t,x}d{}^{n}\!\underbar K_r^{i,t,x}\leq C ({}^{n}\!\underbar K_T^{i,t,x}-{}^{n}\!\underbar K_s^{i,t,x}). 
$$
Therefore,  
\begin{align*}
&(^n\!\sl Y_s^{i,t,x})^2 +\int_s^T |{}^{n}\!\underbar Z_r^{i,t,x}|^2dr+\int_s^T\int_E|^n\!\sl V_r^{i,t,x}(e)|^2{\mu}_n(dr,de)\nnb\\&\le C_\eta +\frac{1}{\eta}\int_s^T|{}^{n}\!\underbar Z_r^{i,t,x}|^2dr+C ({}^{n}\!\underbar K_T^{i,t,x}-{}^{n}\!\underbar K_s^{i,t,x})\\&
\qq -2\int_s^T {{}^{n}\!\sl Y_r^{i,t,x}}\,{}^{n}\!\underbar Z_r^{i,t,x}dB_r -2\int_s^T \int_E {^n\!\sl Y}_r^{i,t,x} \,{^n\!\sl V}_r^{i,t,x}(e) \tilde{\mu}_n(dr,de)\\
&
\le (C_\eta+C) +\frac{1}{\eta}\int_s^T|{}^{n}\!\underbar Z_r^{i,t,x}|^2dr+
C\int_s^T{}^{n}\!\underbar Z_r^{i,t,x}dB_r +C\int_s^T \int_E {^n\!\sl V}_r^{i,t,x}(e) \tilde{\mu}_n(dr,de)\\&
 \qq +C\int_s^T|{}^{n}\!\underbar Z_r^{i,t,x}|dr  
 +2\sup_{s\le T}|\int_s^T  {^n\!\sl Y}_r^{i,t,x}{}^{n}\!\underbar Z_r^{i,t,x}dB_r| +2\sup_{s\le T}|\int_s^T \int_E {^n\!\sl Y}_r^{i,t,x} \,{^n\!\sl V}_r^{i,t,x}(e) \tilde{\mu}_n(dr,de)|
\end{align*}
in replacing $({}^{n}\!\underbar K_T^{i,t,x}-{}^{n}\!\underbar K_s^{i,t,x})$ and taking into account of the boundedness of 
$^n\!\sl Y^{i,t,x}$, $h_i$ and $\widehat{f}_i(t,x,0,0)$, $\ii$. Now in taking the power $q$ in each hand-side, using the inequality $(\sum_{i=1,7}a_i)^q\le 7^{q-1}(\sum_{i=1,7}a_i^q)$ and finally the BDG and Cauchy-Schwarz inequalities we obtain: For any $s\le T$, 
\begin{align*}
&\frac{1}{2}\E\big [\big\{\int_s^T |{}^{n}\!\underbar Z_r^{i,t,x}|^2dr\}^q+\{\int_s^T\int_E|^n\!\sl V_r^{i,t,x}(e)|^2{\mu}_n(dr,de)\big \}^q\big ]\\&
\le 7^{q-1}\Big\{(C_\eta+C)^q +\frac{1}{\eta^q}\E[(\int_s^T|{}^{n}\!\underbar Z_r^{i,t,x}|^2dr)^q]+
C_1\E[(\int_s^T|\underbar Z_r^{i,t,x}|^2dr)^{\frac{q}{2}}] \\&
\qq +C_2\E[(\int_s^T \int_E |{^n\!\sl V}_r^{i,t,x}(e)|^2 {\mu}_n(dr,de))^{\frac{q}{2}}] +C_3\E[(\int_s^T|{}^{n}\!\underbar Z_r^{i,t,x}|^2dr)  ^{\frac{q}{2}}]
 +C_4\E[(\int_s^T|{}^{n}\!\underbar Z_r^{i,t,x}|^2dr)  ^{\frac{q}{2}}]\\&
 \qq +C_5\E[(\int_s^T \int_E |{^n\!\sl V}_r^{i,t,x}(e)|^2 {\mu}_n(dr,de))^{\frac{q}{2}}]\Big \}.
\end{align*}
Next in taking $\eta^q=\eta_0^q=7^{q-1}.4$ and using the Cauchy-Schwarz inequality several times we obtain:
\begin{align*}
&\frac{1}{4}\E\big [\big\{\int_s^T |{}^{n}\!\underbar Z_r^{i,t,x}|^2dr\}^q+\frac{1}{2}\{\int_s^T\int_E|^n\!\sl V_r^{i,t,x}(e)|^2{\mu}_n(dr,de)\big \}^q\big ]\\&
\le 7^{q-1}\Big\{(C_{\eta_0}+C)^q +
(C_1+C_3+C_4)\sqrt{\E[(\int_s^T|\underbar Z_r^{i,t,x}|^2dr)^{q}] }\\&\qq \qq\qq +(C_2+C_5)\sqrt{\E[(\int_s^T \int_E |{^n\!\sl V}_r^{i,t,x}(e)|^2 {\mu}_n(dr,de))^q]}
\Big \}.
\end{align*}
Now, at least after a use of a localization procedure, we obtain that: $\forall s\le T$,
$$
\E\big [\big\{\int_s^T |{}^{n}\!\underbar Z_r^{i,t,x}|^2dr\}^q]+
\E[\{\int_s^T\int_E|^n\!\sl V_r^{i,t,x}(e)|^2{\mu}_n(dr,de)\big \}^q]\le C_q
$$
where the constant $C_q$ does not depend on $n$. 
\ms

Next $\d^*$ satisfies: For any $\ii$ and $\ell \ge 1$,
\begin{align}
^n\underbar Y^{i,t,x}_{0}=&\sum_{k=1}^{\ell}h_{\xi^*_{{k-1}}}(^n\xtx_T)\mathbf{1}_{[\tau _{k}^{\ast}=T]}\mathbf{1}_{[\tau_{k-1}^{\ast}<T]}-\sum_{k=1}^{\ell}g_{\xi^*_{k-1}\xi^*_k}(\tau _{k}^{\ast })\mathbf{1}_{[\tau^*_{k}<T]}+ {^n\underbar Y}_{\tau
_{\ell}^{\ast }}^{\xi_\ell^*}\mathbf{1}_{[\tau _{\ell}^{\ast}<T]}\nonumber\\&\qq+\sum_{k=1}^{\ell}\int_{\tau_{k-1}^{\ast}}^{\tau_k^{\ast}}
\widehat{f}_{\xi^*_{k-1}}(r,{}^{n}\!X_r^{t,x},({}^{n}\!Y_r^{j,t,x})_{j\in \mi},{}^{n}\!\underbar Z_r^{\xi^*_{k-1},t,x})dr-\sum_{k=1}^{\ell}\int_{\tau_{k-1}^{\ast}}^{\tau_k^{\ast}}{^n\underbar Z}^{\xi^*_{k-1}}_{r}dB_{r}\nn \\&\qq-\sum_{k=1}^{\ell}\int_{\tau_{k-1}^{\ast}}^{\tau_k^{\ast}}\int_E {}^{n}\!\underbar V_r^{\xi^*_{k-1},t,x}(e) \tilde{\mu}_n(dr,de)\lb{eqoptim3}
\end{align}
where 
${}^{n}\!\underbar Z_r^{\xi^*_{k-1},t,x}={}^{n}\!\underbar Z_r^{\ell,t,x}$ (resp. 
${}^{n}\!\underbar V_r^{\xi^*_{k-1},t,x}(e)={}^{n}\!\underbar V_r^{\ell,t,x}(e)$)
when $\xi^*_{k-1}=\ell$. Therefore we have:
\begin{align}
&\sum_{k=1}^{\ell}g_{\xi^*_{k-1}\xi^*_k}(\tau _{k}^{\ast })\mathbf{1}_{[\tau^*_{k}<T]}=-^n\underbar Y^{i,t,x}_{0}+\sum_{k=1}^{\ell}h_{\xi^*_{{k-1}}}(^n\!\!\xtx_T)\mathbf{1}_{[\tau _{k}^{\ast}=T]}\mathbf{1}_{[\tau_{k-1}^{\ast}<T]}+{^n\underbar Y}_{\tau
_{\ell}^{\ast }}^{\xi_\ell^*}\mathbf{1}_{[\tau _{\ell}^{\ast}<T]}\nonumber\\&\qq+\sum_{k=1}^{\ell}\int_{\tau_{k-1}^{\ast}}^{\tau_k^{\ast}}
\widehat{f}_{\xi^*_{k-1}}(r,{}^{n}\!\!X_r^{t,x},({}^{n}\!Y_r^{j,t,x})_{j\in \mi},{}^{n}\!\underbar Z_r^{\xi^*_{k-1},t,x})dr-\sum_{k=1}^{\ell}\int_{\tau_{k-1}^{\ast}}^{\tau_k^{\ast}}{^n\underbar Z}^{\xi^*_{k-1}}_{r}dB_{r}\nn \\&\qq-\sum_{k=1}^{\ell}\int_{\tau_{k-1}^{\ast}}^{\tau_k^{\ast}}\int_E {}^{n}\!\!\underbar V_r^{\xi^*_{k-1},t,x}(e) \tilde{\mu}_n(dr,de).\lb{eqoptim}
\end{align}
It follows that:
\begin{align}
&\sum_{k=1}^{\ell}g_{\xi^*_{k-1}\xi^*_k}(\tau _{k}^{\ast })\mathbf{1}_{[\tau^*_{k}<T]}\le |{^n\underbar Y}^{i,t,x}_{0}|+|\sum_{k=1}^{\ell}h_{\xi^*_{{k-1}}}(^n\xtx_T)\mathbf{1}_{[\tau _{k}^{\ast}=T]}\mathbf{1}_{[\tau_{k-1}^{\ast}<T]}|+|{^n\underbar Y}_{\tau
_{\ell}^{\ast }}^{\xi_\ell^*}\mathbf{1}_{[\tau _{\ell}^{\ast}<T]}|\nonumber\\&\qq+|\sum_{k=1}^{\ell}\int_{\tau_{k-1}^{\ast}}^{\tau_k^{\ast}}
\widehat{f}_{\xi^*_{k-1}}(r,{}^{n}\!X_r^{t,x},({}^{n}\!Y_r^{j,t,x})_{j\in \mi},{{}^{n}\!\underbar Z}_r^{\xi^*_{k-1},t,x})dr|+|\sum_{k=1}^{\ell}\int_{\tau_{k-1}^{\ast}}^{\tau_k^{\ast}} {{}^n\underbar Z}^{\xi^*_{k-1}}_{r}dB_{r}|\nonumber\\&\qq+|\sum_{k=1}^{\ell}\int_{\tau_{k-1}^{\ast}}^{\tau_k^{\ast}}\int_E {}^{n}\!\underbar V_r^{\xi^*_{k-1},t,x}(e) \tilde{\mu}_n(dr,de)|.\nonumber\lb{eqoptim2}
\end{align}
Now let us define, for any $r\le T$, $^n\!Z^{*,t,x}_{r}=\sum_{k\ge 1}{^n\!\underbar Z}^{{\xi^*_{k-1},t,x}}_{r}\mathbf{1}_{[\tau _{k-1}^{\ast}\leq r<\tau _{k}^{\ast}]}$ and \\
$^n\!V^{*,t,x}_{r}(e)=\sum_{k\ge 1}{^n\underbar V}^{{\xi^*_{k-1},t,x}}_{r}(e)\mathbf{1}_{[\tau _{k-1}^{\ast}\leq r<\tau _{k}^{\ast}]}$. 
Note that by \eqref{estilemanx}, for any $q\ge 1$ there exists some constant $C_q$
 such that: 
\begin{eqnarray}\label{majbo}
\E\big[\big\{\int_0^T
|^n\!Z^{*,t,x}_{r}|^2dr+\int_0^T \int_E|^n\!V^{*,t,x}_{r}(e)|^2\mu_n(de)dr\big \}^q\big]\le C_q.    
\end{eqnarray}
Then 
\begin{align}
&\sum_{k=1}^{\ell}g_{\xi^*_{k-1}\xi^*_k}(\tau _{k}^{\ast })\mathbf{1}_{[\tau^*_{k}<T]}\leq 2\max_{i\in\mathcal{I}}\sup_{s\in[0,T]}|
{^n\underbar Y}_{s}^{i,t,x}|+\int_{0}^{T}\big \{\sum_{i=1,m}
|\widehat{f}_i(r,{}^{n}\!X_r^{t,x},({}^{n}\!Y_r^{j,t,x})_{j\in \mi},{}^{n}\!\underbar Z_r^{*,t,x})|\big \}dr \nn\\&\qq
+\max_{i\in\mathcal{I}}|h_{i}(^nX^{t,x}_T)|+\sup_{s\in[0,T]}\big|\int_{0}^{s}{}^nZ^{*,t,x}_{r}dB_{r}\big|+\sup_{s\in[0,T]}\big|\int_{0}^{s}\int_E{^n}V^{*,t,x}_{r}(e)\tilde{\mu}_n(dr,de)\big|.\nn
\end{align}
On the other hand the monotonic convergence theorem yields:
$$\E[(A^{\d^*}_T)^q]=\E[\lim_{\ell \rw \infty}(\sum_{k=1}^{\ell}g_{\xi^*_{k-1}\xi^*_k}(\tau _{k}^{\ast })\mathbf{1}_{[\tau^*_{k}<T]})^q]=
\lim_{\ell \rw \infty}\E[(\sum_{k=1}^{\ell}g_{\xi^*_{k-1}\xi^*_k}(\tau _{k}^{\ast })\mathbf{1}_{[\tau^*_{k}<T]})^q]
$$
Take now the power to $q$ in each hand-side of the previous inequality and using the facts that: i) $(^nY^{i,t,x})_{\ii}$, $h^i(.)$ and $\widehat f^i(t,x,\vec 0,0)$ are bounded; ii) $\widehat f(.)$ is Lipschitz w.r.t $(y,z)$ ; iii) the inequality $(a+b+c+d+e)^q\le 5^{q-1}(
a^q+b^q+c^q+d^q+e^q)$ for any positive real constants $a,b,c,d$ and $e$; iv) estimates \eqref{estilemanx} and the Burkholder-Davis-Gundy inequality to deduce the existence of  constant $C_q$ such that: 
\begin{equation} \lb{apdix3}\E[(A_T^{\delta^*})^q]\le C_q
\end{equation}
which is the desired result. \qed
\end{proof}

In the same way we have also the following estimate for $A^{\bar \d}$ where $\bar \d$ is the switching strategy defined in \eqref{eqfd2}. 
\begin{corollaire}Assume that the assumptions $(i)-(ii)$ of Proposition \ref{firstresult} are fulfilled. Let $\bar \d$ be the optimal strategy for the switching problem associated with \\ $(f_i(t,X^n_t,.)\vee f_i(t,X^m_t,.),g_{ij},h_i(X^n_T)\vee h_i(X^m_T))$. Then for any $q\ge 1$, there exists a positive constant $C_q$ such that
\begin{equation} \lb{99}\E[(A_T^{\bar \delta})^q]\le C_q.
\end{equation}
\end{corollaire}
\def \pr {\underline{Proof}}

\end{document}